\documentclass{article}
\usepackage[utf8]{inputenc}

\title{Countability constraints in order-theoretic approaches to computability}
\author{Pedro Hack, Daniel A. Braun, Sebastian Gottwald}
\date{ }

\usepackage{comment}
\usepackage{enumitem}
\usepackage{mathtools}

\usepackage{graphicx}
\usepackage{amsmath}
\usepackage{amsfonts}
\usepackage{amssymb}
\usepackage{amsthm}
\usepackage[colorlinks=true, allcolors=blue]{hyperref}

\usepackage {tikz}
\usetikzlibrary {positioning}
\tikzset{main node/.style={circle,fill=blue!40,draw,minimum size=2cm,inner sep=0pt},}
\tikzset{other node/.style={circle,fill=blue!20,draw,minimum size=1.5cm,inner sep=0pt},}

\newtheorem{prop}{Proposition}
\newtheorem{coro}{Corollary}
\newtheorem{defi}{Definition}
\newtheorem{lem}{Lemma}
\newtheorem{rem}{Remark}

\newtheorem{teo}{Theorem}

\newcommand\twoheaduparrow{\mathrel{\rotatebox[origin=c]{90}{$\twoheadrightarrow$}}}

\newcommand\twoheaddownarrow{\mathrel{\rotatebox[origin=c]{270}{$\twoheadrightarrow$}}}

\usepackage{enumitem}

\usepackage{comment}
\usepackage{mathtools}

\usepackage{graphicx}
\usepackage{amsmath}
\usepackage{amsfonts}
\usepackage{amsthm}

\usepackage {tikz}
\usetikzlibrary {positioning}
\tikzset{main node/.style={circle,fill=blue!40,draw,minimum size=2cm,inner sep=0pt},}
\tikzset{other node/.style={circle,fill=blue!20,draw,minimum size=1.5cm,inner sep=0pt},}

\DeclareMathSymbol{\twoheadrightarrow}  {\mathrel}{AMSa}{"10}

\usepackage{csquotes}

\begin{document}

\maketitle

\begin{abstract}
Computability on uncountable sets has no standard formalization, unlike that on countable sets, which is given by Turing machines. Some of the approaches to define computability in these sets rely on order-theoretic structures to translate such notions from Turing machines to uncountable spaces. Since these machines are used as a baseline for computability in these approaches, countability restrictions on the ordered structures are fundamental.
Here, we show several relations between the usual countability restrictions in order-theoretic theories of computability and some more common order-theoretic countability constraints, like order density properties and functional characterizations of the order structure in terms of multi-utilities. As a result, we show how computability can be introduced in some order structures via countability order density and multi-utility constraints.
\end{abstract}

\section{Introduction}
The formalization of computation on the natural numbers was initiated by Turing \cite{turing1937computable,turing1938computable} with the introduction of Turing machines \cite{rogers1987theory}. Such an approach is taken as canonical today, since other attempts to formalize it have proven to be equivalent
\cite{cutland1980computability,rogers1987theory}.
Because of that, Turing machines are deployed as a baseline for computation from which it is transferred to other spaces of interest. The theory of numberings \cite{ershov1999theory,badaev2000theory}, for example, deals with computability on countable sets in general. The case of uncountable sets is more involved. In fact, despite several attempts \cite{abramsky1994domain,weihrauch2012computability,weihrauch2012computable}, no canonical way of introducing computability on uncountable sets has been established. This results, for example, in the absence of a formal definition of algorithm on the real numbers. More specifically, the choice of some model or another may result in changes regarding computability of certain elementary operations, like multiplication by $3$ \cite{di1996real}.

Among the most extended approaches to computability on uncountable spaces \cite{kreitz1985theory,weihrauch2012computability} some rely on order-theoretic structures \cite{scott1970outline,ershov1972computable,abramsky1994domain,keimel2017domain,hack2022computation}. Of particular importance are those dealing with computability on the real numbers \cite{di1996real,edalat1999domain}. These approaches are based on two main features: the mathematical structure they require and the countability restriction they impose on such a structure in order to translate computability from Turing machines to uncountable sets. We address the general mathematical structure in the accompanying paper \cite{hack2022computation} and deal with the countability restrictions here.
%Here, we are concerned with the countability restrictions imposed on the ordered structures used to introduce computability on uncountable sets.
In particular, we are interested in the relationship between such restrictions and both order density and multi-utilities. 

More specifically, we begin in Section \ref{order in compu} recalling a general order-theoretical approach to computable elements on uncountable sets which was recently introduced in \cite{hack2022computation}. Right after, in Section \ref{density in compu}, we relate the countability restrictions in that approach to order density properties. We continue, in Section \ref{uniform compu}, recalling the more extended, although narrower, approach to computability in domain theory, which we refer to as \emph{uniform computability}, and relating it to the appraoch in \cite{hack2022computation}. In Section \ref{density in uniform compu}, we connect order density properties with the countability restrictions in uniform computability. We follow this, in Section \ref{complete and continuous}, linking order density with order completeness and to a weak form of computability for functions, namely, Scott continuity. We finish, in Section \ref{func restrictions}, addressing the relation between countability restrictions and multi-utilities in both the uniform and non-uniform approaches.

\section{Computability via ordered sets}
\label{order in compu}

In this section, we briefly recall the fundamental notions of an order-theoretic approach to computability on uncountable sets which was recently introduced in \cite{hack2022computation}. We will define a structure which carries computability from Turing machines, namely directed complete partial orders with an effective weak basis. We do not address how computability can be translated from representatives of this structure to other spaces of interest (see \cite{hack2022computation} and the references therein).

%Take a set $X$ where we intend to introduce computability. In order to do so, we will define a structure which carries computability from Turing machines, directed complete partial orders with an effective weak basis, and, then, translate computability from a representative of our structure, $P$, to $X$ via a surjective partial map $\rho: P \rightarrow X$. We will first define the structure with its computability properties and only return to $\rho$ at the end of this section.

%(DEFINE PARTIAL COMPUTABLE FUNCTIONS)
Before introducing directed complete partial orders, we
include some definitions about
the formal approach to computability on $\mathbb{N}$
based on Turing machines.
%recall some definitions regarding.

\begin{defi}[Computable functions and recursively enumerable sets {\cite{turing1937computable,rogers1987theory}}]
A function $f:\mathbb{N} \to \mathbb{N}$ is \emph{computable} if there exists a Turing machine which, for all $ n \in \mathbb{N}$, halts on input $n$, that is, finishes after a finite amount of time, and returns $f(n)$. Note what we call a computable function is also referred to as a \emph{total recursive function} to differentiate it from functions $g:\mathbb{N} \to \mathbb{N}$ where $\text{dom} (g) \subset \mathbb{N}$ holds \cite{rogers1987theory}, which we call \emph{partially computable}.
%Since we only need total recursive functions here, we will call them \emph{computable functions} for simplicity.
A subset $A \subseteq \mathbb{N}$ is said to be \emph{recursively enumerable} if either $A=\emptyset$ or there exists a computable function $f$ such that $A=f(\mathbb{N})$.
\end{defi}

Recursively enumerable sets are, thus, the subsets of $\mathbb{N}$ whose elements can be produced in finite time, as we can introduce the natural numbers one by one in increasing order in a Turing machine and it will output one by one, each in finite time, all the elements in $A$ (possibly with repetitions). Note that there exist subsets of $\mathbb{N}$ which are not recursively enumerable \cite{rogers1987theory}. As we are also interested in computability on the subsets of $\mathbb{N}^2$, we translate the notion of recursively enumerable sets from $\mathbb{N}$ to $\mathbb{N}^2$ using pairing functions. A \emph{pairing function} $\langle \cdot,\cdot\rangle $ is a computable bijective function $\langle \cdot,\cdot\rangle: \mathbb{N} \times \mathbb{N} \rightarrow \mathbb{N}$.\footnote{Notice, just like we defined computable functions $f:\mathbb{N} \rightarrow \mathbb{N}$, we can define computable functions $f:\mathbb{N} \times \mathbb{N} \rightarrow \mathbb{N}$ via Turing machines.} Since it is a common practice \cite{rogers1987theory}, we fix in the following $\langle n,m \rangle= \frac{1}{2} ( n^2+ 2nm + m^2 + 3n + m)$, the \emph{Cantor pairing function}.

Before continuing, we introduce an important concept for the following, finite maps.

\begin{defi}[Finite map {\cite{hack2022computation}}]
We say a map $\alpha: \text{dom}(\alpha) \to A$, where $\text{dom}(\alpha) \subseteq \mathbb{N}$, is a finite map for $A$ or simply a finite map if $\alpha$ is bijective and both $\alpha$ and $\alpha^{-1}$ are effectively calculable.
\end{defi}

Finite maps aim to translate computability from the natural numbers to another countable set $A$. Note the definition of finite maps relies on the informal notion of \emph{effective calculability}. This is  the case since no general formal definition for \emph{computable} maps $\alpha: \mathbb{N} \to A$ is known \cite{cutland1980computability}. In fact, the struggle between formal and informal notions of computability, best exemplified by Church's thesis \cite{rogers1987theory}, lies at the core of computability theory and is responsible for the introduction of different formal notions of computability \cite[Chapter 3]{cutland1980computability}. Finite maps
are also known as \emph{effective denumerations} \cite{cutland1980computability} or \emph{effective enumerations} \cite{scott1970outline}.

We define now the order structure on which we rely to introduce computability in some (potentially uncountable) set $P$ and connect it, right after, with Turing machines.

\begin{defi}[Partial order {\cite{bridges2013representations}}]
A \emph{partial order} $\preceq$ on a set $P$ is a reflexive ($x \preceq x$ $\text{for all } x \in P$), transitive ($x \preceq y$ and $y \preceq z$ imply $x \preceq z$ $\text{for all } x,y,z \in P$) and antisymmetric ($x \preceq y$ and $y \preceq x$ imply $x = y$ $\text{for all } x,y \in P$) binary relation. We will call a pair $(P,\preceq)$ a \emph{partial order} and denote it simply by $P$.
\end{defi}

We may think of $P$ as a set of data and of $\preceq$ as a representation of the precision (or \emph{information}) relation between different elements in the set. Given $x,y \in P$, we may read $x \preceq y$ like \emph{$y$ is at least as informative as $x$} or like \emph{$y$ is at least as precise as $x$}.

We intend now to introduce the idea of some $x \in P$ being the limit of other elements in $P$, that is, the idea that one can generate some element $y \in P$ via a process which outputs other elements of $P$ (which approximate $y$ to arbitrary precision). This notion is formalized by the least upper bounds of directed sets.

\begin{defi}[Direct set and least upper bound {\cite{abramsky1994domain}}]
$A \subseteq P$ is a \emph{directed set} if, given $a,b \in A$, there exist some $c \in A$ such that $a\preceq c$ and $b\preceq c$. If $A \subseteq P$ is a directed set, then $b \in P$ is the \emph{least upper bound of $A$} if $a \preceq b$ $\text{for all } a \in A$ and, given any $c\in P$ such that $a \preceq c$ $\text{for all } a \in A$, then $b \preceq c$ holds. We denote the least upper bound of $A$ by $\sqcup A$ and also refer to it as the \emph{supremum of $A$}.
\end{defi}

Hence, we can generate some $x \in P$ by generating a directed set $A$ whose upper bound is $x$, $x = \sqcup A$. We have restricted ourselves to directed sets since we can think of them as the output of some computational process augmenting the precision or information given that, for any pair of outputs, there is a third which contains their information and, potentially, more. Directed sets are, thus, a formalization of a computational process having a direction, that is, processes gathering information in a consistent way.
%\footnote{After reading Definition \ref{eff weak basis}, the reader can consult some further discussion regarding directed sets in Appendix \ref{directed justi}}
Of particular importance are \emph{increasing sequences} or \emph{increasing chains}, subsets $A \subseteq P$ where $A=(a_n)_{n\geq0}$ and $a_n \preceq a_{n+1}$ $\text{for all } n \geq 0$. We can interpret increasing sequences as the output of some process where information increases every step.

Any process whose outputs increase information should tend towards some element in $P$, that is, any directed set $A \subseteq P$ should have a supremum $\sqcup A \in P$. A partial order with such a property is called \emph{directed complete} or a \emph{dcpo}. Note, for example, the partial order $P_0=\big((0,1),\leq\big)$ is not directed complete.
%, since there are computational processes in $P_0$ whose information increases and are headed nowhere in $P_0$, for example, one whose output is $\mathbb{Q} \cap [0,1)$.

Some subsets $B \subseteq P$ are able to generate all the elements in $P$ via the supremum of directed sets contained in $B$. We refer to them as \emph{weak bases}.

\begin{defi}[Weak basis {\cite{hack2022computation}}]
 A subset $B \subseteq P$ of a dcpo $P$ is a weak basis if, for each $x \in P$, there exists a directed set $B_x \subseteq B$ such that $x=\sqcup B_x$.   
\end{defi}

 We are particularly interested in dcpos where countable weak bases exist, since we intend to inherit computability from Turing machines. In case we have some computational process whose outputs are in $B$ and which is approaching some $x \in P\setminus B$, we would like to be able to provide, after a finite amount of time, the best approximation of $x$ so far. In order to do so, we need to distinguish the outputs we already have in terms of precision. This is possible if the weak basis is \emph{effective}.

 \begin{defi}[Effective weak basis {\cite{hack2022computation}}]
 A countable weak basis $B \subseteq P$ of a dcpo $P$ is effective if there exist both a finite map for $B=(b_n)_{n\geq0}$ and a computable function $f:\mathbb{N} \to \mathbb{N}$ such that $f(\mathbb{N}) \subseteq \{\langle n,m \rangle| b_n \preceq b_m\}$ and, for each $x \in P$, there is a directed set $B_x \subseteq B$ such that $\sqcup B_x = x$ and, if $b_n,b_m \in B_x\setminus \{x\}$,
then there exists some $b_p \in B_x$ such that $b_n,b_m \prec b_p$ and $\langle n,p \rangle, \langle m,p \rangle \in f(\mathbb{N})$.
 \end{defi}

The intuition behind the effectivity is that we can, by finite means, get progressively more \emph{informative} elements from some directed set.
Note we may show a countable weak basis $B=(b_n)_{n\geq0}$ is effective by proving the stronger property that
\begin{equation*}
\label{eff weak basis}
\{\langle n,m \rangle| b_n \preceq b_m\}
\end{equation*}
is recursively enumerable. If this stronger condition is satisfied, then, for any finite subset $(b_n)_{n=1}^N \subseteq B$ where all elements are related, we can find some $n_0 \leq N$ such that $b_n \preceq b_{n_0}$ $\text{for all } n \leq N$ and, since $\alpha: \mathbb{N} \to B$ is a finite map, we can determine the best approximation so far, $\alpha(n_0)$. We define now computable elements for dcpos with an effective weak basis.

\begin{defi}[Computable element {\cite{hack2022computation}}]
\label{def: comput ele}
If $P$ is a dcpo, $B\subseteq P$ is an effective weak basis and $\alpha$ is a finite map for $B$, then
an element $x \in P$ is computable if there exists some $B_x \subseteq B$ such that the properties in the definition of effectivity are fulfilled and $\alpha^{-1}(B_x) \subseteq \mathbb{N}$ is recursively enumerable.
\end{defi}

We have achieved the goal of deriving computability (for potentially uncountable sets) from Turing machines via dcpos. Note that computable elements generalize the approach by Turing to computability on $P_{inf}(\mathbb{N})$, the family of infinite subsets of $\mathbb{N}$ (see \cite{hack2022computation}). The dependence of computability on the order-theoretic model is also addressed in \cite{hack2022computation}.

To recapitulate, the main features of our picture are $(1)$ a map from the natural numbers to some countable set of \emph{finite} labels $B=(b_n)_{n\geq0}$ and $(2)$ a partial order $\preceq$ which can be somewhat encoded via a Turing machine and which allows us to both associate to some infinite element of interest $x \in P$ a subset of our labels $B_x \subseteq B$ which converges to it and, in some sense, to provide approximations of $x$ to arbitrary precision.
%decide what the best approximation of $x$ in $B_x$ is.
As a result, the computability of $x$ reduces to whether $B_x$ can be finitely described or not.

Note that the structure of the partial order $P$ is fundamental to address higher type computability. In case $P$ is \emph{trivial} (also known as \emph{discrete}), that is, $x \preceq y$ if and only if $x=y$ $\text{for all } x,y \in P$ \cite{abramsky1994domain}, we cannot extend computability beyond countable sets and we end up considering countable sets with finite maps towards the natural numbers. This situation, thus, reduces our approach to the \emph{theory of numberings} \cite{ershov1999theory,badaev2000theory,badaev2008computability}.

While the set of computable elements in a dcpo is countable, since the set of recursively enumerable subsets of $\mathbb{N}$ is countable \cite{rogers1987theory}, the cardinality of a dcpo with an effective weak basis is bounded by the cardinality of the continuum $\mathfrak{c}$ (see \cite{hack2022computation}). In fact, it is in the uncountable case where the order structure is of interest, since the theory of numberings is insufficient.

\subsection{Examples}
\label{examples}

To conclude this section, we list three examples of dcpos with effective weak bases which will be relevant in the following.
%the Cantor domain or Cantor set model \cite{blanck2008reducibility,martin2000foundation}, the interval domain \cite{edalat1999domain,di1996real,scott1970outline} and majorization \cite{martin2006entropy,marshall1979inequalities,gottwald2019bounded}.

\subsubsection{The Cantor domain}

If $\Sigma$ is any finite set of symbols, an \emph{alphabet}, we denote by $\Sigma^*$ the set of finite strings of symbols in $\Sigma$ and by $\Sigma^\omega$ the set of countably infinite sequences of symbols. The union of these last two sets
%$\Sigma^\infty \coloneqq \Sigma^*\cup \Sigma^\omega$
is called the \emph{Cantor domain} or the \emph{Cantor set model} \cite{martin2000foundation,blanck2008reducibility} when we equip it with the prefix order. That is, the Cantor domain is the pair $(\Sigma^{\infty},\preceq_C)$, where
%\begin{gather*}
\begin{equation}
\label{Cantor domain}
\begin{split}
    \Sigma^\infty &\coloneqq \Big\{x\Big|x:\{1,..,n\} \to \Sigma,\text{ }0\leq n\leq \infty\Big\}, \\
    x \preceq_C y &\iff |x| \leq |y| \text{ and } x(i)=y(i) \text{ } \text{for all } i \leq |x|,
   \end{split}
%\end{gather*}
\end{equation}
$|s|$ is the cardinality of the domain of $s \in \Sigma^\infty$, and $|\Sigma|<\infty$. One can see $\Sigma^*$ is an effective weak basis for $\Sigma^{\infty}$\cite{hack2022computation}.

\subsubsection{The interval domain}

The \emph{interval domain} \cite{scott1970outline,di1996real,edalat1999domain} consists of the pair $(\mathcal{I}, \sqsubseteq)$, where
%\begin{gather*}
\begin{equation}
\label{interval domain}
\begin{split}
    \mathcal{I}\coloneqq& \Big\{[a,b] \subseteq \mathbb{R}\Big| a,b \in \mathbb{R},\text{ }a \leq b\Big\}\cup \Big\{\perp\Big\}, \text{ and } \\
    x \sqsubseteq y \iff& x= \perp \text{ or } x=[a,b], y=[c,d],  a \leq c \text{ and } d \leq b.
    \end{split}
    \end{equation}
    %\end{gather*}
Note that one can see $B_{\mathcal{I}} \coloneqq \{[p,q] \subseteq \mathbb{R}|p \leq q,\text{ }p,q \in \mathbb{Q}\} \cup \{\perp\}$ is an effective weak basis for $(\mathcal{I}, \sqsubseteq)$. If $P$ is a partial order, we will denote by $\perp$ an element $x \in P$, if it exists, such that $x \preceq y$ $\text{for all } y \in P$, as we just did for the interval domain.

\subsubsection{Majorization}

For any $n \geq 2$, \emph{majorization} \cite{marshall1979inequalities,martin2006entropy,gottwald2019bounded} consists of the pair $(\Lambda^n,\preceq_M)$, where
\begin{equation}
\label{majorization}
\begin{split}
    \Lambda^n \coloneqq \bigg\{ x\in[0,1]^n \bigg| &\sum_{i=1}^n x_i = 1 \text{ and } (\text{for all } i<n)\text{ } x_i \geq x_{i+1} \bigg\}, \\
    x \preceq_M y &\iff (\text{for all } i<n)\text{ } s_i(x) \leq s_i(y),
    \end{split}
\end{equation}
and $s_i(x) \coloneqq \sum_{j=1}^i x_j$ $\text{for all } i<n$. Note majorization has an effective weak basis, as the following proposition (which we prove in the Appendix \ref{proof prop 1}) states.

\begin{prop}
\label{majo eff weak basis}
If $n\geq 2$, then $\mathbb{Q}^n \cap \Lambda^n$ is an effective weak basis for majorization.
\end{prop}

Proposition \ref{majo eff weak basis} relies on a stronger property fulfilled by majorization, which we state in Lemma \ref{arbitrary q} and prove in the Appendix \ref{proof lemma majo}.

\begin{lem}
\label{arbitrary q}
If $x \in \Lambda^n\setminus \{\perp\}$, then $\text{for all } \varepsilon >0$ there exists some $q \in \mathbb{Q}^n \cap \Lambda^n$ such that $s_k(x)-\varepsilon < s_k(q) < s_k(x)$ $\text{for all } k<n$.
\end{lem}

\section{Order density and weak bases}
\label{density in compu}

In the approach to computability from Section \ref{order in compu}, dcpos with countable weak bases are the fundamental structure allowing to define computability. In this section, we relate order density properties to countable weak bases. In order to do so, 
%Before introducing the order density properties of interest here,
we need some extra terminology. If $P$ is a partial order and $x,y \in P$, we say $y$ is \emph{strictly preferred} to $x$ or $x$ is \emph{strictly below} $y$ for $x,y \in P$ and denote it by $x \prec y$ if $x \preceq y$ and $\neg(y \preceq x)$ hold. In case we have $\neg(x \preceq y)$ and $\neg(y \preceq x)$, we say $x$ and $y$ are \emph{incomparable} and denote it by $x \bowtie y$.
%if $\neg(x \preceq y)$ and $\neg(y \preceq x)$.
We introduce now two order density properties which will play a major role in the following.
%: Debreu separability and Debreu upper separability.

\begin{defi}[Order density properties]
A subset $D \subseteq P$ of a partial order $P$ is \emph{Debreu dense} if, for any $x,y \in P$ such that $x \prec y$, there exists some $d \in D$ such that $x \preceq d \preceq y$. We say $P$ is \emph{Debreu separable} if there exists a countable Debreu dense subset $D \subseteq P$ \cite{debreu1954representation,bridges2013representations}. Moreover, we say $D \subseteq P$ is \emph{Debreu upper dense} if, $\text{for all } x,y \in P$ such that $x \bowtie y$, there exists some $d \in D$ such that $x \bowtie d \preceq y$ \cite{hack2022representing}. Lastly, we say $P$ is \emph{Debreu upper separable} if there exists a countable subset $D \subseteq P$ which is Debreu dense and Debreu upper dense \cite{hack2022representing}.
\end{defi}

To exemplify the previous properties, note that $\Sigma^*$ is a countable Debreu dense subset of the Cantor domain $\Sigma^\infty$ since, given $a,b \in \Sigma^\infty$ such that $a \prec_C b$, we have $a \in \Sigma^*$ and, taking $d=a$, we get $a \preceq_C d \preceq_C b$.  Moreover, note that $\Sigma^*$ is a countable Debreu upper dense subset of the Cantor domain since, given $a,b \in \Sigma^\infty$ such that $a \bowtie b$, we can either pick $d=b$ if $b \in \Sigma^*$ or some
$d \in \Sigma^*$ such that $a \bowtie d \prec b$ if $b \in \Sigma^\omega$. Lastly, in particular, the Cantor domain is Debreu upper separable.

As a first relation between order density and weak bases, we show, in Proposition \ref{weak basis implications}, that
having a countable weak basis implies there is a countable Debreu upper dense subset and also a countable set which fulfills a weak form of Debreu density. In fact, although we are primarily interested in the countable case, we show a more general relation, where countability plays no role.

\begin{prop}
\label{weak basis implications}
If $P$ is a dcpo and $B \subseteq P$ is a weak basis, then $B$ is a Debreu upper dense subset. Furthermore, if $x \prec y$, then there exists some $b \in B$ such that $b \preceq y$ and either $x \prec b$ or $x \bowtie b$ holds.
    \end{prop}
    
    \begin{proof}
    For the first statement consider $x,y \in P$ such that $x\bowtie y$. By definition of weak basis, there exists $B_y \subseteq B$ such that $\sqcup B_y = y$. If we have $b \preceq x$ $\text{for all } b \in B_y$, then, by definition of supremum, we would have $y \preceq x$, a contradiction. There exists, thus, $b_0 \in B_y$ such that $\neg(b_0 \preceq x)$. Moreover, $x \preceq b_0$ also leads to contradiction as we would have, by transitivity, $x \preceq y$. Thus, we have $x \bowtie b_0 \preceq y$ and $B$ is a Debreu upper dense subset of $P$. For the second statement, consider $x,y \in P$ such that $x \prec y$ and $B_y \subseteq B$ such that $\sqcup B_y = y$. Notice $b \preceq y$ $\text{for all } b \in B_y$ by definition while $b \preceq x$ $\text{for all } b \in B_y$ implies $y \preceq x$, a contradiction. There exists, thus, $b_0 \in B_y$ such that $\neg(b_0 \preceq x)$. Then, $((x \prec b_0)\cup (x \bowtie b_0))\cap(b_0 \preceq y)$ holds.
    \end{proof}
    
Because of the proof of Proposition \ref{weak basis implications}, it seems having a countable weak basis is insufficient for a countable Debreu dense subset to exist. This is indeed the case, as we show in Proposition \ref{weak basis not debreu sep} via a counterexample.
    
    \begin{prop}
    \label{weak basis not debreu sep}
    There exist dcpos with countable weak bases and no countable Debreu dense subset.
    %separable.
    \end{prop}
    
    \begin{proof}
    Consider $P \coloneqq ([0,1]\cup [2,3],\preceq)$ where
    \begin{equation*}
    %\label{example}
    x \preceq y \iff
    \begin{cases}
    x\leq y \text{ and } x,y \in [0,1], \text{ or}\\
    x \leq y  \text{ and } x,y \in [2,3], \text{ or}\\ x+2 \leq y.
    \end{cases}
\end{equation*}
%and the definition of $\preceq$ finishes including the transitive relations induced by \eqref{example}.
(See Figure \ref{figure} for a representation of $P$.) Since $P$ is clearly a partial order, we show now it is directed complete. Take $A \subseteq P$ directed. If we have either $A \subseteq [0,1]$ or $A \subseteq [2,3]$ then $\sqcup A$ exists and is the supremum of $A$ in the usual $\mathbb{R},\leq)$ sense. If $A$ has elements in both, then it is easy to see that the supremum of $A \cap [2,3]$ with respect to $\leq$ is the supremum of $A$ with respect to $\preceq$.
%whenever . and
Analogously, $B \coloneqq \mathbb{Q} \cap P$ is a countable weak basis. To conclude, assume $D \subseteq P$ is Debreu dense set. There exist, then, $d_x \in D$ such that $x \preceq d_x \preceq x+2$ $\text{for all } x \in [0,1]$. By definition, we have $d_x \in \{x,x+2\}$.
%or $d_x=x+2$.
In particular, $d_x \neq d_y$ $\text{for all } x,y \in [0,1]$ $x \neq y$ and $D$ is uncountable.
%and $P$ is not Debreu separable.
    \end{proof}
    
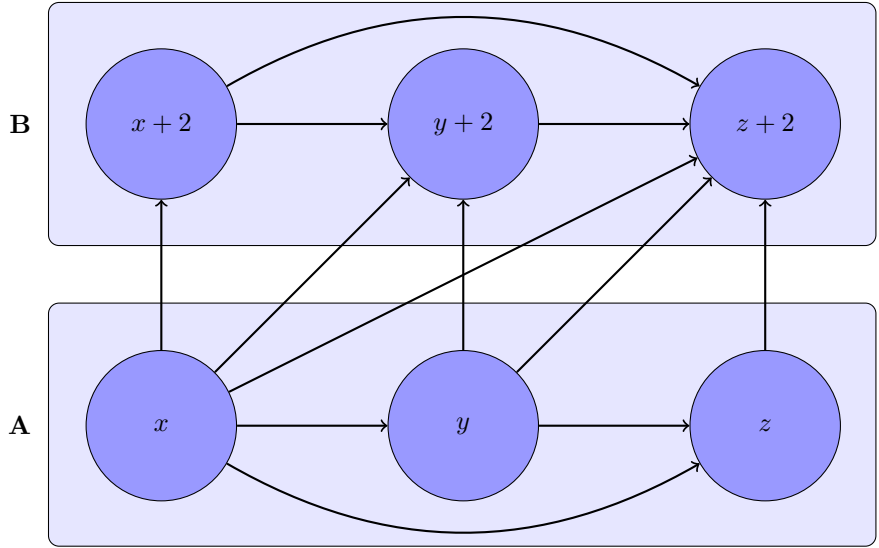
\begin{figure}[!tb]
\centering
\begin{tikzpicture}
\node[rounded corners, draw,fill=blue!10, text height = 3cm, minimum width = 11cm,xshift=4cm,label={[anchor=west,left=.1cm]180:\textbf{B}}] {};
\node[rounded corners, draw,fill=blue!10, text height = 3cm, minimum width = 11cm,xshift=4cm,yshift=-4cm,label={[anchor=west,left=.1cm]180:\textbf{A}}] {};
    \node[main node] (1) {$x+2$};
    \node[main node] (2) [right = 2cm  of 1]  {$y+2$};
    \node[main node] (3) [right = 2cm  of 2]  {$z+2$};
    \node[main node] (4) [below = 2cm  of 1] {$x$};
    \node[main node] (5) [right = 2cm  of 4] {$y$};
    \node[main node] (6) [right = 2cm  of 5] {$z$};

    \path[draw,thick,->]
    (4) edge node {} (2)
    (4) edge node {} (3)
    (5) edge node {} (3)
    (4) edge node {} (5)
    (5) edge node {} (6)
    (1) edge node {} (2)
    (2) edge node {} (3)
    (4) edge node {} (1)
    (6) edge node {} (3)
    (5) edge node {} (2)
        (1) edge [bend left] node {} (3)
    (4) edge [bend right] node {} (6)
    ;
\end{tikzpicture}
\caption{Representation of a dcpo, defined in Proposition \ref{weak basis not debreu sep}, with a countable weak basis and no countable Debreu dense subset. In particular, we show $A \coloneqq [0,1]$, $B \coloneqq [2,3]$ and how $x,y,z \in A$, $x<y<z$, are related to $x+2,y+2,z+2 \in B$. Notice an arrow from an element $w$ to an element $t$ represents $w \prec t$.}
\label{figure}
\end{figure}
    
    We consider now the converse of Proposition \ref{weak basis implications}, that is, whether countable weak bases exist for any Debreu separable dcpo. Some extra terminology is needed for this purpose.

    \begin{defi}[Trivial directed sets and elements]
\label{trivial elements}
    We say $x \in P$ has a \emph{non-trivial} directed set if there exists a directed set $A \subseteq P$ such that $\sqcup A=x$ and $x \not \in A$. Accordingly, given a weak basis $B \subseteq P$, we call the set of elements which have non-trivial directed sets $A \subseteq B$ the \emph{non-trivial} elements of $B$ and denote it by $\mathcal{N}_B$. We define the \emph{trivial} elements of $B$ equivalently and denote them by $\mathcal{T}_B$. In case we take $B=P$ as weak basis, we refer to the trivial (non-trivial) elements of $B$ simply as trivial (non-trivial) elements.
    \end{defi}

    Intuitively, non-trivial (trivial) elements are those for which there is (there is not) a computational process which converges to them without ever actually outputting them. This is important since we may have some element that requires infinite precision, like say $\pi$ in decimal representation, which may be achievable to arbitrary precision via some algorithm which only outputs elements with a finite representation. 
    
    We prove, in Theorem \ref{thm 1} that, if $P$ is a dcpo with a weak basis $B\subseteq P$ and a countable Debreu dense subset, then there exists a dcpo $Q$ with a countable weak basis such that  for the non-trivial set of $B$ is included in $Q$, $\mathcal{N}_B \subseteq Q$. In order to do so, we first need two lemmas. We start, in Lemma \ref{always next}, recalling the straightforward fact that, whenever the supremum of a directed set $A$ is not contained in the set, we can find $\text{for each } a \in A$ some $b \in A$ that is strictly preferred to $a$, $a \prec b$ (see \cite[Lemma 1]{hack2022computation}). Right after, in Lemma \ref{countable covers}, we show that, if $x \in P$ has a non-trivial directed set $A \subseteq P$ and there exists a countable Debreu dense subset $D\subseteq P$, then $x$ has a non-trivial increasing sequence contained in $D$.  In order to do so, we first consider the elements from $D$ which are \emph{between} (in $\preceq$) those in $A$ and then profit from the countability of $D$ to build an increasing sequence converging to the supremum of $A$. 
    
\begin{lem}[\cite{hack2022computation}]
\label{always next}
If $P$ is a dcpo and $A \subseteq P$ is a directed set such that $\sqcup A \notin A$, then, for all $a \in A$, there exists some $b \in A$ such that $a \prec b$.
\end{lem}

%Before proving Theorem \ref{thm 1}, we show, in Lemma \ref{countable covers}, a useful property for its proof and for later results.

%\begin{lem}
%\label{countable covers}
%If $P$ is a dcpo with a countable Debreu dense subset $D\subseteq P$ and $x \in P$ has a non-trivial directed set, then $x$ has a} non-trivial increasing sequence whose elements are in $D$.
%\end{lem}

\begin{lem}
\label{countable covers}
If $P$ is a dcpo with a countable Debreu dense subset $D\subseteq P$, $x \in P$ is an element with a non-trivial directed set $A \subseteq P$, and
\begin{equation}
\label{def D_A}
D_A:=\{d \in D| \exists a,b \in A \text{ s.t. } a \preceq d \preceq b \}
\end{equation}
is a subset of $D$ with a numeration $D_A=(d_n)_{n\geq 0}$,
then $(d'_n)_{n\geq 0}$ is a non-trivial increasing sequence for $x$, where
\begin{equation}
\label{def d'_n}
\begin{split}
    d_0' &\coloneqq d_0,\\
    d_n' &\coloneqq d_{m_n} \text{ for all } n \geq 1,
    \end{split}
\end{equation}
and $m_n \geq 0$ fulfills $d_{n-1}' \preceq d_{m_n}$ and $d_n \preceq d_{m_n}$ for all  $n \geq 1$.
\end{lem}

\begin{proof}
Take $x \in P$ and $A \subseteq P$ a non-trivial directed set for $x$.
%such that $\sqcup A = x$.
Notice, since $\sqcup A \notin A$, then for every $a \in A$ there is some  $b \in A$ such that $a \prec b$ by Lemma \ref{always next}. Consider, also, the set $D_A$ defined in \eqref{def D_A}.
%\begin{equation*}
%D_A:=\{d \in D| \exists a,b \in A \text{ s.t. } a \preceq d \preceq b \}.
%\end{equation*}
Notice $D_A$ is countable and $x \notin D_A$. We notice, first, $D_A$ is directed. Given $d, d' \in D_A$, there exist $b, b' \in A$ such that $d \preceq b$ and $d' \preceq b'$ by definition. Since $A$ is directed, there exists $c \in A$ such that $b \preceq c$ and $b' \preceq c$. Also, by Lemma \ref{always next}, there exists some $e \in D$ such that $c \prec e$. Thus, there exists $d''\in D_A$ such that $c\preceq d''$ by Debreu separability and, by transitivity, we have $d,d' \preceq d''$.
%and $d'\preceq d''$ implying
We conclude $D_A$ is directed. To finish, we will show that the increasing sequence $D_A'=(d'_n)_{n\geq 0} \subseteq D_A$ from \eqref{def d'_n} is well-defined and fulfills $\sqcup D_A'=x$. To show $D_A'$ is well-defined it suffices to notice that we can consider a numeration $D_A=(d_n)_{n\geq 0}$ since $D_A \subseteq D$ and that
%. We take $d_0' \coloneqq d_0$ and define recursively $\text{for all } n \geq 1$ $d_n' \coloneqq d_{m_n}$, where we pick $m_n \geq 0$ such that $d_{n-1}' \preceq d_{m_n}$ and $d_n \preceq d_{m_n}$ hold. Notice such an
$m_n$ exists $\text{for all } n \geq 1$ since, as we showed, $D_A$ is directed. Since $D_A'$ is an increasing sequence by construction, we only need to show $\sqcup D_A'=\sqcup A$. Notice $d'_n \preceq \sqcup A$ by definition. Moreover, given any $a \in A$, there exists some $n \geq 0$ such that $a \preceq d'_n$ by Lemma \ref{always next}, Debreu separability and definition of $D_A'$. In particular, if there is some $z \in P$ such that $d'_n \preceq z$ $\text{for all } n \geq 0$, then $a  \preceq z$ $\text{for all } a \in A$ and, as a result, $\sqcup A \preceq z$. Thus, $\sqcup D_A' = \sqcup A$.
\end{proof}

Note, as a consequence of Lemma \ref{countable covers}, whenever a dcpo $P$ is Debreu separable, any directed set $A \subseteq P$ contains an increasing sequence $(a_n)_{n\geq 0} \subseteq A$ such that $\sqcup (a_n)_{n\geq 0}=\sqcup A$ (see Proposition \ref{exists chain inside} in the Appendix \ref{appen proofs}).

Using Lemmas \ref{always next} and \ref{countable covers}, we can now prove Theorem \ref{thm 1}, our first construction of countable weak bases using order density properties. Intuitively, Theorem \ref{thm 1} considers $Q$, a subset of some dcpo $P$, and takes all the elements in some basis that can be achieved non-trivially and profits from the sequences in Lemma \ref{countable covers} to achieve them via a countable Debreu dense subset. Lastly, it adds elements to $Q$ in order to assure it is a dcpo. 
%in Theorem \ref{thm 1}. In particular, we show, if $P$ is a dcpo with a countable Debreu dense subset $D \subseteq P$ and a weak basis $B\subseteq P$, then there exists a dcpo with a countable weak basis $Q$ such that $D \cup \mathcal{N}_B \subseteq Q$.

%We want to prove now $\sqcup A = \sqcup D_A$. By definition $d \preceq \sqcup A$ $\text{for all } d \in D_A$ since $b \preceq \sqcup A$ $\text{for all } b \in A$. Assume now there exists $z \in P$ such that $d \preceq z$ $\text{for all } d \in D_A$. Given any $a \in A$ there exists $d \in D_A$ such that $a \preceq d$ by Lemma \ref{always next} and Debreu separability. In particular, $a \preceq z$ $\text{for all } a \in A$ implying $\sqcup A \preceq z$. This leads to $\sqcup D_A = \sqcup A$.

%\begin{prop}
%\label{thm 1}
%If $P$ is a dcpo with a countable Debreu dense set $D$ and a weak basis $B$ then there is a countable weak basis for the non-trivial elements of $B$.
%\end{prop}

\begin{teo}
\label{thm 1}
If $P$ is a dcpo with a countable Debreu dense set $D\subseteq P$ and a weak basis $B\subseteq P$, then
the dcpo $(Q,\preceq)$ fulfills $D \cup \mathcal{N}_B \subseteq Q$ and has $D$ as countable weak basis, where
\begin{equation*}
    Q \coloneqq D \cup \{x \in P\setminus D| \exists \text{ a directed set } A \subseteq D \text{ s.t. } \sqcup A=x\}
\end{equation*}
and $\preceq$ is the restriction to $Q$ of the partial order in $P$.
\end{teo}
\begin{proof}
We will show $(Q,\preceq)$ is a dcpo with a countable weak basis such that $D \cup \mathcal{N}_B \subseteq Q$.
%, where $Q \coloneqq D \cup \{x \in P\setminus D| \exists \text{ a directed set } A \subseteq D \text{ s.t. } \sqcup A=x\}$ and $\preceq$ is the restriction to $Q$ of the partial order in $P$.
Note $\mathcal{N}_B \subseteq Q\setminus D$ by Lemma \ref{countable covers}. To conclude the proof, we only need to show $Q$ is directed complete as, after this is established, it is clear by definition that $D \subseteq Q$ is a countable weak basis of $Q$. Take, thus, a directed set $A \subseteq Q$ with $\sqcup A \not \in A$, as the opposite is straightforward. Note $\sqcup A \in P$, since $P$ is a dcpo, and we intend to show $\sqcup A \in Q$. To begin with, consider
\begin{equation*}
    A' \coloneqq A \cup \{d \in D| \exists a,b \in A \text{ s.t. } a \preceq d \preceq b\} \subseteq Q,
\end{equation*}
%$A' \coloneqq A \cup \{d \in D| \exists a,b \in A \text{ s.t. } a \preceq d \preceq b\} \subseteq Q$,
which is straightforwardly a directed set.
%by the proof of Lemma \ref{countable covers}.
Note $\sqcup A'$ exists in $P$ and, actually, $\sqcup A'=\sqcup A$, as we have $(1)$ $\text{for all } a' \in A'$ there exists some $a \in A$ such that $a'\preceq a$, which leads to $a' \preceq \sqcup A$ $\text{for all } a' \in A$, and $(2)$ if $a'\preceq y$ $\text{for all } a'\in A'$ then $a \preceq y$ $\text{for all } a \in A$ and we get $\sqcup A \preceq y$. Consider now
 \begin{equation*}
A'' \coloneqq \Big(A' \cap D\Big) \bigcup_{x \in A'\setminus D} B_x, 
 \end{equation*}
 where, $\text{for all } x \in A'\setminus D$, $B_x \subseteq D$ is a directed set, which exists by definition of $Q$, such that $\sqcup B_x=x$. We conclude by showing that $(1)$ $A''$ is directed and that $(2)$ $\sqcup A''=\sqcup A$, which imply, since $A'' \subseteq D$, that $\sqcup A = \sqcup A'' \in Q$. As a result, $Q$ is directed complete, as we intended to prove. To show $(1)$, we take $a,b \in A''$ and
 %need to show for any pair $a,b \in A''$ there exists some $c \in A''$ such that $a,b \preceq c$. We
 consider four cases: (a) $a,b \in B_x$ for some $x \in A'\setminus D$, (b) $a,b \in A' \cap D$, (c) $a \in B_x, b \in B_{x'}$ with $x,x' \in A'\setminus D$ $x \neq x'$ and (d) $a \in A' \cap D$ and $b \in B_x$ for some $x \in A'\setminus D$. In (a) there exists some $c \in B_x$ such that $a,b \preceq c$ since $B_x$ is a directed set. (b) holds as there exists some $y \in A$ such that $a,b \preceq y$ and some $y' \in A$ such that $y \prec y'$ by Lemma \ref{always next}. We obtain there exists some $c \in A' \cap D$ such that $a,b\preceq y \preceq c \preceq y'$. (c) holds as there exists some $y \in A$ such that $x,x'\preceq y$ and we can follow (b) to get some $c \in A' \cap D$ such that $a \preceq x \preceq c$ and $b \preceq x' \preceq c$. (d) holds similarly to (c). In order to finish, we only need to show $(2)$ holds. Since $\text{for all } a'' \in A''$ there exists some $a' \in A'$ such that $a ''\preceq a'$, we have $a ''\preceq \sqcup A'=\sqcup A$ $\text{for all } a'' \in A''$. Moreover, if we have $a''\preceq z$ $\text{for all } a''\in A''$, then we have, by definition of $A''$, $a' \preceq z$ $\text{for all } a'\in A' \cap D$ and, by definition of $B_{a'}$, $a'\preceq z$ $\text{for all } a'\in A'\setminus D$. Thus, $\sqcup A=\sqcup A' \preceq z$ and, hence, $\sqcup A''=\sqcup A'=\sqcup A$.
\end{proof}

\begin{rem}[Implication for computability]
    By of Theorem \ref{thm 1}, we can define computable elements (in the sense of Definition \ref{def: comput ele}) on $D \cup \mathcal{N}_B \subseteq P$, where $D\subseteq P$ is a countable Debreu dense subset and $\mathcal{N}_B$ are the non-trivial elements of some weak basis $B \subseteq P$ (which may be uncountable), whenever $D$ is effective.
\end{rem}

 Theorem \ref{thm 1} has the following straightforward strengthening.

\begin{coro}
\label{coro 1}
If $P$ is a dcpo with a countable Debreu dense subset $D \subseteq P$ and $B \subseteq P$ is a weak basis, then the set of trivial elements of $B$, $\mathcal{T}_B$, is countable if and only $D \cup \mathcal{T}_B$ is a countable weak basis for $P$.
\end{coro}

%Note, in particular, $D \cup \mathcal{T}_B$ is a countable weak basis for $P$ if the first clause holds.
From the proof of Corollary \ref{coro 1}, which relies on Lemma \ref{countable covers}, Debreu separability alone seems to be insufficient to build a countable weak basis. This is precisely the case, as we show in Proposition \ref{cont Deb upper sep but not w-cont} via a counterexample. In fact, our counterexample shows even strengthening the hypothesis to Debreu upper separability is insufficient.

\begin{prop}
\label{cont Deb upper sep but not w-cont}
There exist Debreu upper separable dcpos without countable weak bases.
\end{prop}

\begin{proof}
Take $P \coloneqq (\Sigma^* \cup \Sigma^\omega,\preceq)$, where $\Sigma$ is finite and
\begin{equation*}
    x \preceq y \iff 
    \begin{cases}
    x=y, \text{ or}\\
    x  \in \Sigma^*,\text{ } y \in \Sigma^\omega \text{ and } x \preceq_C y,
    \end{cases}
\end{equation*}
with $\preceq_C$ the partial order from the Cantor domain. (See Figure \ref{figure 2} for a representation with $\Sigma=\{0,1\}$.) As we directly see $P$ is a partial order, we begin by showing $P$ is directed complete. Consider a directed set $A \subseteq P$ with $|A|\geq 2$. (If $A=\{a\}$ for some $a \in P$, then $\sqcup A= a$ and we have finished.)
%If $|A|=1$, then $\sqcup A$ exists. Assume now $|A|\geq 2$.
If $A \cap \Sigma^\omega = \emptyset$, notice $A$ is not directed, since given $x,y \in A \cap \Sigma^*$ $x \neq y$ there is no $z \in \Sigma^*$ such that $x,y \preceq z$. Thus, there exists some $x_A \in A \cap \Sigma^\omega$. Notice we actually have $A \cap \Sigma^\omega=\{x_A\}$, as given $y \in \Sigma^\omega$ $y \neq x_A$ there is no $z \in P$ such that $x_A,y\preceq z$. Analogously, we obtain $y \preceq x_A$ $\text{for all } y \in A$ and, thus,
%if $|A|\geq 2$, there exists some $x_A \in A \cap \Sigma^\omega$ and $\text{for all } y \in A$ we have $y \preceq x_A$ implying
$\sqcup A=x_A$. We conclude $P$ is directed complete. 
%then we will show  Notice for any directed set $D \subseteq P$ we have $D \subseteq d(x)$ where $x \in \Sigma^\omega$ and $d(x)\coloneqq \{y \in P| y \preceq x\}$ which means $P$ is directed complete as the possibilities are $D=\{x\}$ where $x \in \Sigma^\omega \cup \Sigma^*$ or $|D| \geq 2$ which implies, since by directness there cannot be two elements in $\Sigma^\omega \cap D$,  either $D=\{x,y\}$ where $x \in \Sigma^\omega$ and $y \in \Sigma^*$ $y \prec x$ or for any $x,y \in D \cap \Sigma^*$ there exists by directness some $z \in D \cap \Sigma^\omega$ $x,y \prec z$ where $z$ is the same for any pair as, again, having two elements in $\Sigma^\omega \cap D$ violates directness.
We notice now $\Sigma^*$ is a countable Debreu dense and Debreu upper dense subset of $P$. If $x \prec y$ for some $x,y \in P$, then $x \in \Sigma^*$ by definition and we conclude $\Sigma^*$ is Debreu dense. If $x \bowtie y$ with $x,y \in \Sigma^\omega$, then there exists some $d \in \Sigma^*$ such that $x \bowtie_c d$ and $d \preceq_c y$.
%. Taking $d=y(1)y(2)..y(n)$ we get $x \bowtie d \preceq y$.
If $y \in \Sigma^\omega$ and $x \in \Sigma^*$ we consider some $d \in \Sigma^*$ such that $d \preceq_C y$ and get $x \bowtie d \preceq y$. If either $x \in \Sigma^\omega$ and $y \in \Sigma^*$ or $x,y \in \Sigma^*$, then we take $d=y$.
We obtain 
%get $\Sigma^*$ is Debreu upper dense as well implying
$P$ is Debreu upper separable. To finish, we will show %Consider now $B$ a weak basis for $P$. We will now show $B$ is uncountable which implies there are no countable weak bases.
any weak basis $B \subseteq P$ is uncountable.
By definition of weak basis, there exists some directed set $A_x \subseteq B$ such that $x = \sqcup A_x$ $\text{for all } x \in \Sigma^\omega$.
%Notice $x \in A_x$ as, otherwise, $A_x$ would not be directed as $A_x \subseteq \Sigma^*$, since $a \in A_x$ implies $a \preceq x$, and $z \bowtie y$ $\text{for all } z,y \in \Sigma^*$ $z \neq y$.
As discussed above, this implies $x \in A_x$.
Hence, $\Sigma^\omega \subseteq B$ and $B$ is uncountable.
\end{proof}

Note that, in the proof of \cite[Proposition 4 (iii)]{hack2022classification}, we introduced the counterexample we used in Proposition \ref{cont Deb upper sep but not w-cont}. Note, also, $\Sigma^\omega \subseteq \mathcal{T}_B$ for any weak basis $B$ of the dcpo in the proof of Proposition \ref{cont Deb upper sep but not w-cont}. Before we continue relating weak bases to order density properties, we define the minimal elements of a partial order. If $P$ is a partial order, then the \emph{set of minimal elements of $P$}, denoted by $\text{min}(P)$, is
\begin{equation*}
\text{min}(P) \coloneqq \{x \in P| \text{ there is no } y \in P \text{ such that } y \prec x\}.
\end{equation*} 
Notice, for the Cantor domain, $\text{min}(P)=\{\perp\}$. $\text{min}(P)$ is related to order density properties since, as we note in Lemma \ref{minP countable}, it is countable whenever countable Debreu upper dense subsets exist.

\begin{lem}
\label{minP countable}
If $P$ is a partial order with a countable Debreu upper dense subset $D\subseteq P$, then $\text{min}(P)$ is countable.
\end{lem}

\begin{proof}
If $|\text{min}(P)|\leq 1$, we are done. If there exist $x,y \in \text{min}(P)$ $x \neq y$, then $x \bowtie y$, since $x \prec y$ ($y\prec x) $ contradicts the fact $y \in \text{min}(P)$ ($x \in \text{min}(P)$). Thus, there exists some $d \in D$ such that $x \bowtie d \preceq y$. By definition of $\text{min}(P)$, we have $d=y$ and, hence, $\text{min}(P) \subseteq D$ is countable.
\end{proof}
Clearly, we cannot eliminate the hypothesis in Lemma \ref{minP countable} as, for example, $\text{min}(P)$ is uncountable for any uncountable set with the trivial partial order.

\begin{figure}[!tb]
\centering
\begin{tikzpicture}
\node[rounded corners, draw,fill=blue!10, text height = 3cm, minimum width = 11cm,xshift=4cm,label={[anchor=west,left=.1cm,thick, font=\fontsize{15}{15}\selectfont, thick]180:\textbf{$\Sigma^\omega$}}] {};
\node[rounded corners, draw,fill=blue!10, text height = 3cm, minimum width = 11cm,xshift=4cm,yshift=-4cm,label={[anchor=west,left=.1cm, thick, font=\fontsize{15}{15}\selectfont, thick]180:\textbf{$\Sigma^*$}}] {};
    \node[main node] (1) {$010101..$};
    \node[main node] (2) [right = 2cm  of 1]  {$000000..$};
    \node[main node] (3) [right = 2cm  of 2]  {$010000..$};
    \node[main node] (4) [below = 2cm  of 1] {$0101$};
    \node[main node] (5) [right = 2cm  of 4] {$0$};
    \node[main node] (6) [right = 2cm  of 5] {$01$};

    \path[draw,thick,->]
    (4) edge node {} (1)
    (5) edge node {} (2)
    (5) edge node {} (1)
    (5) edge node {} (3)
    (6) edge node {} (3)
    (6) edge node {} (1)
    ;
\end{tikzpicture}
\caption{Representation of a dcpo, defined in Proposition \ref{cont Deb upper sep but not w-cont}, which is Debreu upper separable and has no countable weak basis. In particular, we show $\Sigma^\omega$ and $\Sigma^*$ for $\Sigma=\{0,1\}$ and how $010101..,000000..,010000.. \in \Sigma^\omega$ are related to $0101, 0, 01 \in \Sigma^*$. Notice an arrow from an element $w$ to an element $t$ represents $w \prec t$.}
\label{figure 2}
\end{figure}
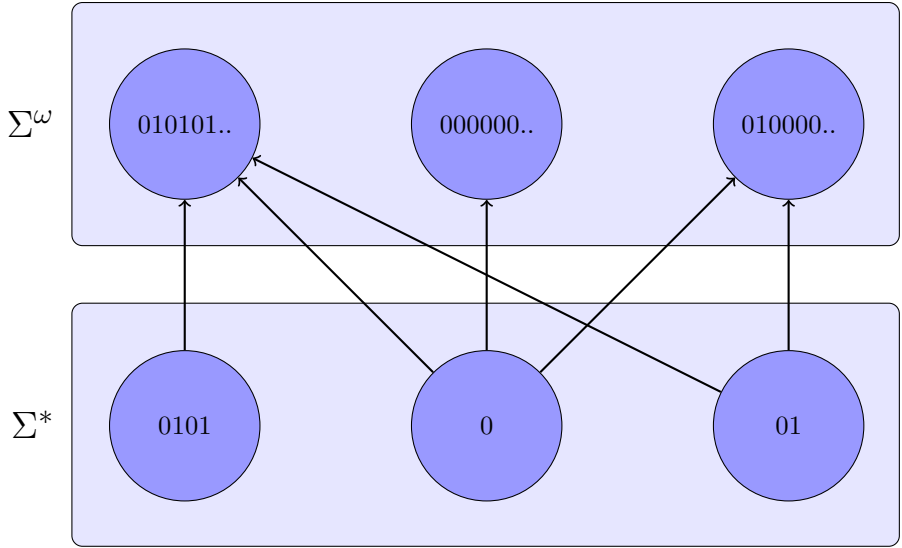

Since, by Proposition \ref{cont Deb upper sep but not w-cont}, Debreu upper separability is not enough for countable weak bases to exist, we consider stronger order density properties.

\begin{defi}[Order density properties II]
We say $D \subseteq P$ is \emph{order dense} if, for any pair $x,y \in P$ such that $x \prec y$, there exists some $d \in D$ such that $x \prec d \prec y$ \cite{ok2002utility}. We say $P$ is \emph{order separable} if it has a countable order dense subset \cite{mehta1986existence}.
\end{defi}

Notice that, although the Cantor domain $\Sigma^\infty$ is Debreu separable, it is not order separable since, if $s \in \Sigma^*$ and $\beta \in \Sigma$, then we have $s \prec s\beta$ and, $\text{for all } t \in \Sigma^*$ such that $t \prec s\beta$, $t \preceq s$ also holds. As we show in Proposition \ref{order sep has count weak basis}, order separability is sufficient to build a countable weak basis for $P\setminus \text{min}(P)$ the basic idea is similar to the one for Lemma \ref{countable covers} and, taking into account Lemma \ref{minP countable}, we can extend the result to $P$ by also assuming the existence of a countable Debreu upper dense subset.

\begin{prop}
\label{order sep has count weak basis}
If $P$ is a dcpo, then countable order dense subsets $D \subseteq P$ are countable weak bases for $P\setminus \text{min}(P)$.
%If $P$ is an order separable dcpo, then $P\setminus \text{min}(P)$ is a dcpo with a countable weak basis.
Furthermore, if $P$ also has a countable Debreu upper dense subset, then $D \cup \text{min}(P)$ is a countable weak basis for $P$.
\end{prop}

\begin{proof}
We begin showing the first statement. Notice, if $P$ is a dcpo, then $P\setminus \text{min}(P)$ is a dcpo as well. Take $D \subseteq P$ a countable order dense subset of $P$, which we can choose w.l.o.g. such that $D \cap \text{min}(P) = \emptyset$. We will show $D$ is a countable weak basis for $P\setminus \text{min}(P)$. In particular, we will show %the following lemma.
%\begin{lem}
%\label{order dense chain}
%If $P$ is a dcpo, $D \subseteq P$ is a countable order dense subset and $y \in P\setminus \text{min}(P)$, then there exists an increasing chain $D' =(d_n')_{n \geq 0} \subseteq D$ such that $\sqcup D' =y$.
%\end{lem}
%}
the following lemma.
\begin{lem}
\label{order dense chain}
If $P$ is a dcpo, $D \subseteq P$ is a countable order dense subset and $x,y \in P$ such that $x \prec y$, then $D' =(d_n')_{n \geq 0} \subseteq D$ is an increasing chain such that $\sqcup D' =y$, where
\begin{equation}
\label{def d'_n II}
\begin{alignedat}{2}
    d_0' &\coloneqq d_{m_0} \text{ and } &&m_0 \coloneqq \text{min} \{n \geq0|x \prec d_n \prec y\},\\
    d_n' &\coloneqq d_{m_n} \text{ and } &&m_n \coloneqq \text{min}\{n \geq m_{n-1}|d_{n-1}' \prec d_n \prec y\} \text{ for all } n \geq 1.
\end{alignedat}
\end{equation}
\end{lem}

\begin{proof}
 Take some $y \in P\setminus \text{min}(P)$ and $D=(d_n)_{n \geq 0}$ a numeration of $D$. Since $y \not \in \text{min}(P)$, there exists some $x \in P$ such that $x \prec y$. Moreover, by definition of order separability, we have that $m_n$ exists for all $n \geq 0$ and, hence, $D' =(d_n')_{n \geq 0} \subseteq D$ in \eqref{def d'_n II} is well-defined.
 %We define $d_0' \coloneqq d_{m_0}$, where $m_0 \coloneqq \text{min} \{n \geq0|x \prec d_n \prec y\}$. Notice, by definition of order separability, $m_0$ exists. We define $d'_n$ recursively $\text{for all } n\geq1$, that is, $d_n' \coloneqq d_{m_n}$, where $m_n \coloneqq \text{min}\{n \geq m_{n-1}|d_{n-1}' \prec d_n \prec y\}$ $\text{for all } n\geq1$. Again, by order separability, $m_n$ exists $\text{for all } n\geq1$.
 Since $D'=(d_n')_{n\geq0}$ is an increasing sequence by construction, it has a supremum $\sqcup D'$. Given that $y$ is an upper bound of $(d_n')_{n\geq0}$ by construction, we have $\sqcup D' \preceq y$. To finish the proof of the first statement, we assume $\sqcup D' \prec y$ and get a contradiction. If that was the case, there would be some $\overline{n} \geq0$ such that $\sqcup D' \prec d_{\overline{n}} \prec y$ by order separability. Consider, thus, $\overline{m} \coloneqq \text{max} \{ n < \overline{n}| d_n \in D'\}$. Since $d_{\overline{m}+1},..,d_{\overline{n}-1} \not \in D'$ by definition of $\overline{m}$ and $d_{\overline{m}} \prec d_{\overline{n}} \prec y$ by transitivity, we would have, assuming $d_n'=d_{\overline{m}}$ for some $n\geq0$ w.l.o.g., $d_{n+1}'=d_{\overline{n}}$, a contradiction. Thus, $D'$ is an increasing chain such that $ \sqcup D' = y$.
 \end{proof}
 Hence, by Lemma \ref{order dense chain}, $D$ is a countable weak basis for $P\setminus \text{min}(P)$. Regarding the second statement, notice we can take $A_x= \{x\}$ as directed set with $\sqcup A_x=x$ $\text{for all } x \in \text{min}(P)$ and, since $\text{min}(P)$ is countable by Lemma \ref{minP countable}, we have $D \cup \text{min}(P)$ is a countable weak basis for $P$.
\end{proof}

\begin{rem}[Implication for computability]
By Proposition \ref{order sep has count weak basis}, we can define computable elements (in the sense of Definition \ref{def: comput ele}) on a dcpo $P$, even if it is uncountable, whenever we have a countable order dense subset $D_1$ and a countable Debreu upper dense subset $D_2$, provided $D_1 \cup D_2$ is effective.
\end{rem}

Since order density of $P$ is enough to introduce computability on $P\setminus \text{min}(P)$ while, again by Proposition \ref{cont Deb upper sep but not w-cont}, Debreu separability is not, we take a stronger version of Debreu separability as hypothesis in Proposition \ref{imme succ}. In particular, we show there are countable weak bases for $P\setminus \text{min}(P)$ whenever countable Debreu dense subset satisfying a specific property exist. To express such a property, we recall some definitions in \cite{bridges2013representations}. If $x,y \in P$, then we say $y$ is an \emph{immediate successor} of $x$ and $x$ is an \emph{immediate predecessor} of $y$ if both $x \prec y$ and $(x, y) \coloneqq \{z \in P|x \prec z \prec y\} = \emptyset$ hold. In this scenario, $(x, y)$ is called a \emph{jump}.

\begin{prop}
\label{imme succ}
If $P$ is a dcpo
and $D \subseteq P$ is
%If $P$ is a dcpo with
a countable Debreu dense subset whose elements have a finite number of immediate successors, then
$(Q_1\cup D)\setminus \text{min}(P)$ is a countable weak basis for $P\setminus \text{min}(P)$, where
\begin{equation}
\label{def Q1}
Q_1 \coloneqq \left\{ y \in P| \exists x \in P \text{ s.t. } (x,y)=\emptyset\right\}.
\end{equation}
%$P\setminus \text{min}(P)$ is a dcpo with a countable weak basis. 
Furthermore, if $P$ also has a countable Debreu upper dense subset, then $Q_1 \cup D \cup \text{min}(P)$ is a countable weak basis for $P$.
\end{prop}

\begin{proof}
We begin showing the first statement. Notice, if $P$ is a dcpo, then $P\setminus \text{min}(P)$ is also a dcpo. Take $D\subseteq P$ a countable Debreu dense subset of $P$ whose elements have a finite number of immediate successors and $y \in P\setminus \text{min}(P)$. Notice $y \in Q_0 \cup Q_1$, where
%We can divide such a $y$ is two categories: either for every $x \in P$ such that $x \prec y$ there exists some $z \in P$ such that $x \prec z \prec y$ or there exists some $t \in P$ such that $(t,y)=\emptyset$.
%$Q_0 \coloneqq \{ y \in P| \text{for all } x \in P \text{ s.t. } x \prec y \text{ } \exists z \in P \text{ s.t. } x \prec z \prec y\}$
%$Q_1 \coloneqq \{ y \in P| \exists x \in P \text{ s.t. } (x,y)=\emptyset\}$
\begin{equation*} 
Q_0 \coloneqq \{ y \in P| \text{for all } x \in P \text{ s.t. } x \prec y \text{ } \exists z \in P \text{ s.t. } x \prec z \prec y\}
\end{equation*}
and $Q_1$ is defined as in \eqref{def Q1}.
If $y \in Q_0$, we take some $x \prec y$ and we can emulate the construction of $(d'_n)_{n \geq 0}$ in Lemma \ref{order dense chain}, \eqref{def d'_n II}, to construct an increasing sequence contained in $D$ whose supremum is $y$.
%: we consider $(d_n)$ a numeration of $D$ and construct an increasing chain $(d'_n) \subseteq D$ such that $x \prec d'_n \prec y$. For the last step the definition of $y$ is important as there exists some $z \in P$ and $m \in \mathbb{N}$ such that $x \prec z \preceq d_m \prec y$. We take $d'_0=d_m$ and define $d'_n$ $\text{for all } n>0$ in the same fashion substituting $x$ by $d'_{n-1}$.
We just have to notice, by Debreu separability and definition of $Q_0$, whenever $x \prec y$ for $x,y \in P$, there exist $z \in P$ and $m\geq0$ such that $x \prec z \preceq d_m \prec y$.
Conversely, if $y \in Q_1$,
%$y$ belongs to the second category, that is,
there exists some $x \in P$ such that $(x,y)=\emptyset$. If $y \not \in D$, then $x \in D$ by Debreu separability. If we denote by $(s_n)_{n\geq0}$ a numeration of the immediate successors of the elements in $D$, which can be constructed since each member of $D$ has a finite number of immediate successors and $D$ is countable by hypothesis, we have
%$Y \coloneqq \{y \in P\setminus D| \exists t \text{ such that } (t,y)= \emptyset\} \subseteq (s_n)_{n\geq0}$
$Q_1 \subseteq D\cup(s_n)_{n\geq0}$ and, thus,
$Q_1$ is countable.
%After noticing $Y \cap \text{min}(P)=\emptyset$, we get $B=Y \cup D'$
We conclude $(Q_1\cup D)\setminus \text{min}(P)$ is a countable weak basis for $P\setminus \text{min}(P)$.
%where $D'=D\setminus \text{min}(P)$.
For the second statement, since $\text{min}(P)$ is countable by Lemma \ref{minP countable}, we conclude $Q_1 \cup D \cup \text{min}(P)$ is a countable weak basis for $P$.
\end{proof}

\begin{rem}[Implication for computability]
By Proposition \ref{imme succ}, we can define computable elements (in the sense of Definition \ref{def: comput ele}) on a dcpo $P$ with a countable Debreu dense subset $D_1$ and a countable Debreu upper dense subset $D_2$ whenever the elements in $D_1$ have a finite number of immediate successors and $D_1 \cup D_2$ is effective.
\end{rem}

Note the hypothesis in the first statement of Proposition \ref{imme succ} is weaker than order separability since, whenever a partial order is dense, immediate successors do not exist.\footnote{We say a partial order is \emph{dense} if, for all $x,y \in X$ such that $x \prec y$, there exists some $z \in X$ such that $x \prec z \prec y$ \cite{bridges2013representations}. That is, a partial order is \emph{dense} if it has an order dense subset.}
%there is a finite map $\alpha: \mathbb{N} \to D_1 \cup D_2$ and $\{\langle n,m \rangle| \alpha(n)\preceq \alpha(m)\}$ is recursively enumerable.
Notice, also, the Cantor domain satisfies the hypotheses in Proposition \ref{imme succ}, since $\Sigma^*$ is a countable Debreu dense subset of $\Sigma^\infty$ and each $s \in \Sigma^*$ has exactly $|\Sigma|$ immediate successors. The converse of Proposition \ref{imme succ} is, again by Proposition \ref{weak basis not debreu sep}, false.

To summarize, the main results in this section are Theorem \ref{thm 1} and Propositions \ref{order sep has count weak basis} and \ref{imme succ}. In the first one, we show how one can profit from Debreu density to introduce a dcpo that includes computability on the non-trivial elements of any weak basis $B$. In the second, we use the stronger property of order separability to extend computability to all non-minimal elements. In the last one, we show that we can achieve the same results as in Proposition \ref{order sep has count weak basis} by asking for the existence of a countable Debreu separable subset whose elements only have a finite number of immediate successors. Lastly, we extend computability to the whole dcpo in Propositions \ref{order sep has count weak basis} and \ref{imme succ} by also requiring the existence of a countable Debreu upper dense subset.

\section{Uniform computability via ordered sets}
\label{uniform compu}

Following \cite{hack2022computation}, in Section \ref{order in compu}, we considered an element $x$ in some dcpo $P$ with a countable weak basis $B\subseteq P$ to be computable if there exists some directed set $B_x \subseteq B$ such that $\sqcup B_x =x$ and whose associated subset of the natural numbers $\alpha^{-1}(B_x)$ is recursively enumerable.
%Such an approach is, however, not very practical as showing some element is not computable is not easy.
We consider now a stronger approach, where we
%Instead, we would like to
associate to each $x \in P$ a unique directed set $B_x \subseteq B$ such that $\alpha^{-1}(B_x)$ is recursively enumerable if and only if 
$x$ is computable. In order to do so, $B_x$ should be fundamentally related to $x$, that is, the information in every $b \in B_x$ should be gathered by any process which computes $x$. The differences between the approach in Section \ref{order in compu} and the one here are discussed in \cite{hack2022computation}. Crucially, while the more general approach only enables to introduce computable elements, computable functions can also be defined in this stronger framework, which was introduced by Scott in \cite{scott1970outline}. In fact, the uniform approach to computability we discuss in this section is known as \emph{domain theory} \cite{scott1982lectures,stoltenberg1994mathematical,mislove1998topology,edalat1999domain,stoltenberg2001notes,gierz2012compendium,cartwright2016domain}. Note that, as in Section \ref{order in compu}, we only introduce the order structure used to define computability and do not address how to translate it to other spaces of interest (see \cite{hack2022computation} and the references therein).

 Before we continue, 
 we introduce the Scott topology, which will play a major role in the following.
 
 \begin{defi}[Scott topology {\cite{scott1970outline,abramsky1994domain}}]
 If $P$ is a dcpo, we say a set $O \subseteq P$ is \emph{open} in the \emph{Scott topology}
 if it is \emph{upper closed} (if $x \in O$ and there exists some $y \in P$ such that $x \preceq y$, then $y \in O$) and \emph{inaccessible by directed suprema} (if $A\subseteq P$ is a directed set such that $\sqcup A \in O$, then $A \cap O \neq \emptyset$). We denote by $\sigma(P)$ the Scott topology on $P$.
 \end{defi}
 
 Note that
 %An important property of
 the Scott topology characterizes the partial order in $P$, that is,
\begin{equation}
\label{charac order by topo}
    x \preceq y \iff  x \in O \text{ implies } y \in O \text{ } \text{for all } O \in \sigma(P)
\end{equation}
\cite[Proposition 2.3.2]{abramsky1994domain}.
Note that, by \eqref{charac order by topo}, the Scott topology satisfies the $T_0$ topological separation axiom. (We say a topological space $(X,\tau)$ is a $T_0$ space or a \emph{Kolmogorov} space if, for every pair of distinct points $x,y \in X$, there exists an open set $O \in \tau$ such that either $x \in O$ and $y \not \in O$ or $y \in O$ and $x \not \in O$ hold \cite{kelley2017general}.)

We can think of the Scott topology as the family of properties on the data set $P$ which allow us to distinguish the elements in $P$ \cite{smyth1983power}. In particular, by definition, if some computational processes has a limit, then any property of the limit is verified in finite time and, since $\sigma(P)$ is $T_0$, these properties are enough to distinguish between the elements of $P$.

In order to attach
to each element in a dcpo a unique subset of $\mathbb{N}$ which is equivalent to it for all computability purposes, we recall the way-below relation.

\begin{defi}[Way-below(-above) relation {\cite{scott1972continuous,abramsky1994domain}}]
If $x,y \in P$, we say $x$ is \emph{way-below} $y$ or $y$ is \emph{way-above} $x$ and denote it by $x \ll y$ if, whenever $y \preceq \sqcup A$ for a directed set $A\subseteq P$, then there exists some $a \in A$ such that $x \preceq a$. The set of element way-below (way-above) some $x$ is denoted by $ \twoheaddownarrow x$ ($\twoheaduparrow x$).
\end{defi}

 We think of an element way-below another as containing \emph{essential information} about the latter, since any computational process producing the latter cannot avoid gathering the information in the former. For example, if $A_1 \subseteq \twoheaddownarrow x$ is a directed set such that $\sqcup A_1 =x$, then, if $A_2 \subseteq P$ is a directed set where $\sqcup A_2=x$, there exists $\text{for all } a_1 \in A_1$ some $a_2 \in A_2$ such that $a_1 \preceq a_2$.
We can regard $A_1$, thus, as a canonical computational process that yields $x$. (To be more precise, one can show that the computability, in the sense of Definition \ref{def: comp ele II}, of an element is equivalent to the computability of a specific subset of the basis. This is formalized in \cite[Proposition 2]{hack2022computation}.) The significance of the introduction of $\ll$ for computability is discussed in \cite{hack2022computation}. We simply note, for the moment, that $x \ll y$ implies $x \preceq y$ and that, if $x \preceq y$ and $y \ll z$, then $x \ll z$ $\text{for all } x,y,z\in P$.
%Note we are not requiring each element in $A_1$ to be in $A_2$, only their information. Actually, we can even have disjoint directed sets $A_1 \cap A_2 = \emptyset$ where $A_1,A_2 \subseteq \twoheaddownarrow x$ and $\sqcup A_1=\sqcup A_2 =x$. Take, for example, the dcpo $P \coloneqq ((0,1],\leq)$ and $A_1 \coloneqq \{q \in (0,1] \cap \mathbb{Q}|q<x\}$, $A_2 \coloneqq \{r \in (0,1]\setminus \mathbb{Q}|r<x\}$ for any $x \in P$.

We introduce now, in a canonical way, computable elements in a dcpo. To do so, we first define effective bases.

\begin{defi}[Basis {\cite{abramsky1994domain}}]
A subset $B \subseteq P$ is called a \emph{basis} if, for any $x \in P$, there exists a directed set $B_x \subseteq \twoheaddownarrow x \cap B$ such that $\sqcup B_x = x$.
\end{defi}

Note that bases are exactly like weak bases except for the fact they achieve any element $x \in P$ using elements of $P$ which contain essential information about $x$. Recall, if $B$ is a basis, then it is a weak basis and $(\twoheaduparrow b)_{b \in B}$ is a topological basis for $\sigma(P)$ \cite{hack2022computation,abramsky1994domain}. A dcpo is called \emph{continuous} or a \emph{domain} if bases exist. As in the case of weak bases, we are interested in dcpos with a countable basis or \emph{$\omega$-continuous} dcpos, since computability can be introduced via
%can be given computability notions using
Turing machines there. Note that $\Sigma^\infty$ is $\omega$-continuous, as $B=\Sigma^*$ is a countable basis.

\begin{defi}[Effective basis {\cite{edalat1999domain,stoltenberg2008computability}}]
We say a basis $B \subseteq P$ is \emph{effective} if there is a finite map which enumerates $B$, $B=(b_n)_{n\geq0}$, there is a bottom element $\perp \in B$ and
\begin{equation*}
\{\langle n,m \rangle| b_n \ll b_m\}
\end{equation*}
is recursively enumerable.
\end{defi}

Effectivity for bases has the same purpose as for weak bases, although it poses stronger requirements. A discussion in this regard can be found in \cite{hack2022computation}. 
We assume w.l.o.g. $b_0 = \perp$ in the following. Note that the relevance of the bottom element is discussed in \cite{hack2022computation}. We can now define computable elements.

\begin{defi}[Computable element {\cite{edalat1999domain}}]
\label{def: comp ele II}
If $P$ is a dcpo with an effective basis $B=(b_n)_{n\geq0}$, we say $x \in P$ is \emph{computable} provided 
\begin{equation*}
\{n \in \mathbb{N}|b_n \ll x\}
\end{equation*}
is recursively enumerable.
\end{defi}

The dependence of the set of computable elements on the model in this approach is addressed in \cite{hack2022computation}.

Before introducing computable functions, we need some extra terminology.
We call a map $f:P \rightarrow Q$ between dcpos $P,Q$ a \emph{monotone} if $x \preceq y$ implies $f(x) \preceq f(y)$. Furthermore, we call $f$ \emph{continuous} if it is monotone and, for any directed set $A \subseteq P$, we have $f(\sqcup A) = \sqcup f(A)$ \cite{scott1970outline,abramsky1994domain}. Note that this definition is equivalent to the usual topological definition of continuity (see \cite{kelley2017general}) applied to the Scott topology. In fact, dcpos constitute the objects of the category DCPO, whose morphisms are the functions which are continuous in the Scott topology \cite{lawson1998computation}. Note, also, if $B \subseteq P$ is a weak basis and $f$ is continuous, then
$f(P)$ is determined by $f(B)$. This is the case since, if $x \in P$, then there exists some directed set $B_x \subseteq B$ such that $\sqcup B_x=x$. Thus, by continuity of $f$,
$f(x)=f(\sqcup B_x)=\sqcup f(B_x)$. We can consider the monotonicity of $f$ in the definition of continuity as a mere technical requirement to make sure $\sqcup f(A)$ exists for any directed set $A \subseteq P$.

We introduce now computable functions.

\begin{defi}[Computable function {\cite{edalat1999domain}}]
\label{def: comp func}
We say a function $f:P \to Q$ between two dcpos $P$ and $Q$ with effective bases $B=(b_n)_{n\geq0}$ and $B'=(b'_n)_{n\geq0}$, respectively, is \emph{computable} if $f$ is continuous and the set
 \begin{equation}
 \label{compu funct def}
\{\langle n,m \rangle| b'_n \ll f(b_m)\}
 \end{equation}
is recursively enumerable.
\end{defi}

%This notion of computable function is independent of the finite maps considered for both $P$ and $Q$ (see Proposition \ref{model indep bases} $(1)$ in the Appendix \ref{indep domain}) and, under some broad hypotheses (see Appendix \ref{indep domain}), also independent of the chosen bases (see Proposition \ref{model indep bases} $(iii)$ in the Appendix \ref{indep domain}). Because of that, we may specify the chosen bases and call an element $x \in P$ \emph{$B$-computable}, instead of just computable, where $B \subseteq P$ is an effective basis. Analogously, we may call a function $f:P \to Q$ \emph{$(B,B')$-computable}, where $B \subseteq P$ and $B' \subseteq Q$ are effective bases. Note, also, for a fixed pair of effective bases, there are a countable number of computable functions between dcpos (see Proposition \ref{count compu f} in the Appendix \ref{appen proofs}).

Intuitively, the idea that any formal definition of computable function is meant to capture is that of a map that sends
computable inputs to computable outputs \cite{braverman2005complexity,ko2012complexity}. Indeed, this is the case for this definition (see \cite[Theorem 9]{edalat1999domain} and \cite{hack2022computation}). 
Moreover, we can think of continuity in $\sigma(P)$ as a weak form of computability.
The dependence of the set of computable functions on the model in this approach is addressed in \cite{hack2022computation}.

%These uniform notions of computable elements and functions can be translated to a set $X$, as in Section \ref{order in compu}, via a partial surjective map $\rho: P \rightarrow X$, which is known as a \emph{domain representation} \cite{stoltenberg2008computability,blanck2008reducibility}. We say $x \in X$ is $\rho$-computable if there exists some computable $p \in P$ such that $x=\rho(p)$ \cite{stoltenberg2008computability}. Note, as in Section \ref{order in compu}, we are restricted to sets $X$ with, at most, the cardinality of the continuum (see \cite{hack2022computation}). Furthermore, given another set $Y$, we say $f:X \rightarrow Y$ is $(\rho,\rho')$-computable if there exists a computable function $g:P \rightarrow Q$ such that $g(\text{dom}(\rho)) \subseteq \text{dom}(\rho')$ and for each $x \in \text{dom}(\rho)$ we have $f(\rho(x))=\rho'(g(x))$ \cite{stoltenberg2008computability}. As they are fundamental in the introduction of computability notions in topological spaces, domain representations have been studied and classified in the literature \cite{stoltenberg2008computability,stoltenberg2001notes,blanck2000domain,blanck2008reducibility,hamrin2005admissible}, emphasizing the case where $X$ is a topological space. Special attention has been given to domain representations $\rho$ such that $\text{dom}(\rho)=max(P)$ \cite{martin1998domain,waszkiewicz2003domains}, where $max(P) \coloneqq \{x \in P| \not \exists y \in P \text{ s.t. } x \prec y \}$.

To conclude, we can return to the examples from Section \ref{examples}. In particular, one can see $\Sigma^*$ is an effective basis for $\Sigma^\infty$ \cite{hack2022computation} and $B_\mathcal{I}$  is known to be an effective basis of $(\mathcal{I},\sqsubseteq)$ \cite{edalat1999domain}. Moreover, $B_n \coloneqq \Lambda^n \cap \mathbb{Q}^n$ is a countable basis for majorization for any $n \geq 2$ and the following lemma, which we prove in Appendix \ref{lemma 5 proof} (see also Appendix \ref{alternative}), holds.

\begin{lem}\label{example domains}
The following statements hold:
%\vspace{-10pt}
\begin{enumerate}[label={(\arabic*)}]
\item Although $(\mathcal{I}, \sqsubseteq)$ is $\omega$-continuous, any Debreu dense subset $Z \subseteq \mathcal{I}$ has the cardinality of the continuum $\mathfrak{c}$.
\item If $n=2$, then $\mathbb{Q}^n \cap \Lambda^n$ is a countable basis for $(\Lambda^n,\preceq_M)$ and $(\Lambda^n,\preceq_M)$ is Debreu upper separable.
\item If $n \geq 3$, then $\mathbb{Q}^n \cap \Lambda^n$ is a countable basis for $(\Lambda^n,\preceq_M)$ and any Debreu dense subset $Z \subseteq \Lambda^n$ has the cardinality of the continuum $\mathfrak{c}$.
\end{enumerate}
\end{lem}

Furthermore, one can slightly modify the proof of Proposition \ref{majo eff weak basis} in the Appendix \ref{proof prop 1} to conclude $B_n$ is an effective basis. The key property we use is \cite[Lemma 5.1]{martin2008technique}, where it was shown
\begin{equation}
\label{way-below majo}
x \ll_M y \iff x = \perp \text{ or } s_k(x)<s_k(y) \text{ } \text{for all } k<n
\end{equation}
$\text{for all } x,y \in \Lambda^n$. Notice, Lemma \ref{example domains} improves upon \cite[Theorem 1.3]{martin2006entropy},
establishing $\omega$-continuity instead of just continuity for
%where it was shown
all $(\Lambda^n,\preceq_M)$
%is continuous for any
with $n \geq 2$.
%, as it shows $(\Lambda^n,\preceq_M)$ is actually $\omega$-continuous.

\section{Order density and bases}
\label{density in uniform compu}

 In this section, we continue relating order density properties to computability, specifically, to the uniform approach introduced in Section \ref{uniform compu}. We begin, in Section \ref{cont Debreu separable}, mimicking the relations in Section \ref{density in compu}, this time for bases. After establishing the discrepancy between bases and order density in general, we introduce, in Section \ref{upper comp Debreu}, a property under which they coincide, namely, conditional connectedness. Importantly, the Cantor domain fulfills this property. Notice, for generality, we will assume bases do not necessarily have a bottom element in this section.

 We can summarize the relation between Sections \ref{density in compu} and \ref{density in uniform compu} as follows: First, we show, in Propositions \ref{cont order sep and min coun implies w cont} and \ref{imme succ 2}, respectively, that the results regarding the existence of countable weak bases in Propositions \ref{order sep has count weak basis} and \ref{imme succ} can be emulated for bases by restricting ourselves to \emph{continuous} dcpos. Then, we show, in Theorem \ref{thm 2}, that the positive and negative results regarding the relation between weak bases and order density properties (see Propositions \ref{weak basis implications}, \ref{weak basis not debreu sep} and \ref{cont Deb upper sep but not w-cont}, and Theorem \ref{thm 1}) can be improved upon to reach equivalence by additionally asking for conditional connectedness.

\subsection{Continuous Debreu separable dcpos}
\label{cont Debreu separable}

We consider here continuous dcpos and extend the relationship between order density and computability from Section \ref{density in compu}, focusing on connecting the former with countable bases. We begin introducing a definition and two lemmas that will be useful in the following. We say an element $x \in P$ is \emph{compact} if $x \ll x$ holds and denote by $\text{K}(P)$ the set of compact elements of $P$ \cite{abramsky1994domain}. Note that $\text{K}(P) \subseteq B$ for any basis $B \subseteq P$.
%We will use \emph{continuous dcpo} instead of \emph{domain} in the following as the former has different definitions in the literature \cite{blanck2008reducibility,martin2000foundation}.
%We first notice, in Lemma \ref{compact implies contained if supremum} and Lemma \ref{second count = w-cont}, two properties we will use in the following.

\begin{lem}
\label{compact implies contained if supremum}
All compact elements in a dcpo $P$ are trivial $\text{K}(P) \subseteq \mathcal T_P$. The converse, however, is not true.
\end{lem}

\begin{proof}
Take $x \in \text{K}(P)$ and a directed set $A$ such that $\sqcup A=x$. By definition of compact element, there exists some $a \in A$ such that $x \preceq a$. As $a \preceq \sqcup A$ by definition, then, by antisymmetry, we get $x =\sqcup A = a  \in A$. To show the converse is false, take the dcpo $([0,1],\preceq)$, where $x \preceq y$ if $x \leq y$ $\text{for all } x,y \in (0,1]$, $0 \preceq 0$ and $0 \preceq 1$. Note that $0 \bowtie y$ $\text{for all } y \in (0,1)$. While the only directed set $A$ with $\sqcup A=0$
%has $0$ as upper bound
is $A=\{0\}$, we also have $A' = (0,1)$ is directed with $\sqcup A'=1$ and $0 \bowtie a$ $\text{for all } a \in A'$. Thus, $\neg(0 \ll 1)$ and $0 \not \in \text{K}(P)$ as, otherwise, we would have $0 \ll 0 \preceq 1$, hence $0 \ll 1$.
\end{proof}

\begin{lem}
\label{second count = w-cont}
If $P$ is a continuous dcpo, then $\sigma(P)$ is second countable if and only if $P$ is $\omega$-continuous. However, there exist dcpos without bases where $\sigma(P)$ is second countable.
\end{lem}

\begin{proof}
The first statement is already known, see \cite[Proposition 3.4]{lawson1998computation}. For the second statement, take $P \coloneqq \big(\big[0,1\big],\preceq\big)$, where, for all $x,y \in [0,1]$,
\begin{equation*}
    x \preceq y \iff 
    \begin{cases}
    x \leq y &\text{ if } x,y \in \big[0,\frac{1}{2}\big],\\
    y \leq x & \text{ if } x,y \in \big[\frac{1}{2},1\big].
    \end{cases}
\end{equation*}
It is easy to see $P$ has no basis \cite{hack2022computation}. To conclude, we note that $\big\{(p,q), [0,q),(p,1]\big|0 \leq p < \frac{1}{2}<q \leq 1, p,q \in \mathbb{Q}\big\}$ is a countable basis for $\sigma(P)$. Take, thus, $x \in O \in \sigma(P)$ and assume w.l.o.g. $x < \frac{1}{2}$. If $x >0$, since $O \in \sigma(P)$ and $x \in O$, then there exists some $y<x$ such that $y \in O$, otherwise $D_x \coloneqq \{y \in P|y<x\}$ fulfills $\sqcup D_x=x$ and $D_x \cap O = \emptyset$, a contradiction. Arguing analogously, we have there exists some $z>\frac{1}{2}$ such that $z \in O$. Since $O$ is upper closed, $x \in (p_x,q_x) \subseteq O$, where $p_x,q_x \in \mathbb{Q}$, $y \leq p_x \leq x$ and $\frac{1}{2} \leq q_x \leq z$. If $x=0$, we can follow argue analogously and conclude $x \in [0,q) \subseteq O$ for some $q \in \mathbb{Q}$, $ \frac{1}{2}<q$.
\end{proof}

%Note we borrow  Lemma \ref{second count = w-cont} from \cite[Proposition 3.4]{lawson1998computation}.
Regarding the Cantor domain, Note that $( \twoheaduparrow x)_{x \in \Sigma^*}$ is a countable topological basis for its Scott topology. Lemma \ref{second count = w-cont} establishes that, in order to relate order density with countable bases, we can relate the former to countability axioms on $\sigma(P)$. 
%Notice Lemma \ref{compact implies contained if supremum} cannot be reversed. For example, consider $[0,1]$ where $x \prec 1$ $\text{for all } x \in [0,1)$ while $0 \bowtie y$ $\text{for all } y \in (0,1)$ and $x \preceq y$ if and only if $y-x \geq 0$ $\text{for all } x,y \in (0,1)$. Notice while the only directed set which has $0$ as upper bound is $A=\{0\}$, it is also true $A' = (0,1)$ is directed, $\sqcup A'=1$ and $0 \bowtie a$ $\text{for all } a \in A'$. Thus, $0 \not \in \text{K}(P)$.
As a starting link, we relate, in Proposition \ref{count Debreu separable},
%We begin showing we can recover Theorem \ref{thm 3} for
continuous Debreu separable dcpos with
%(erase both definitions?) Before doing so, we need some terminology. We say a topology is \emph{first countable} if given some $x \in X$ there exists a countable family of open sets $(U^x_n)_{n \in \mathbb{N}}$ such that whenever $x \in O \in \tau$ then there exists some $n$ such that $x \in U_n \subseteq O$. A topology is second countable if there exists a countable family of open sets $(U_n)_{n \in \mathbb{N}}$ such that for any $x \in X$ whenever $x \in O \in \tau$ there exists some $n$ such that $x \in U_n \subseteq O$.
these axioms. In particular, 
we show Debreu separability of $P$ implies $\sigma(P)$ is first countable for continuous dcpos. Furthermore, whenever the set of compact elements is countable, we show $P$ is $\omega$-continuous, that is, its Scott topology is second countable.

\begin{prop}
\label{count Debreu separable}
Take a dcpo $P$ and a countable Debreu dense subset $D \subseteq P$. If $P$ is continuous, then $\sigma(P)$ is first countable. Moreover, if $\text{K}(P)$ is countable, then $D \cup \text{K}(P)$ is a countable basis for $P$.
\end{prop}

\begin{proof}
We begin with the first statement. Take $x \in P$, a basis $B\subseteq P$ and a Debreu dense subset $D\subseteq P$. We will show there exists a countable family $(U_n^x)_{n\geq0} \subseteq \sigma(P)$ such that, if $x \in O \in \sigma(P)$, then there exists some $n_0 \geq 0$ such that $x \in U_{n_0}^x \subseteq O$. If $x \in \text{K}(P)$, we take $(U_n^x)_{n\geq0} \coloneqq \{\twoheaduparrow x\} \subseteq \sigma(P)$ and notice, if $O \in \sigma(P)$, then we have $\twoheaduparrow x \subseteq O$ whenever $x \in O$, as $\twoheaduparrow x = \uparrow x \subseteq O$ by definition. If $x \not \in \text{K}(P)$, then there exists some directed set $B_x \subseteq \twoheaddownarrow x \cap B$ such that $\sqcup B_x=x \not \in B_x$ by definition of basis. By Lemma \ref{countable covers}, there exists an increasing sequence $(d'_n)_{n\geq 0} \subseteq D$ (defined by \eqref{def d'_n} and \eqref{def D_A}) such that $\sqcup (d'_n)_{n\geq 0}=x \not \in (d'_n)_{n\geq 0}$. Since, by construction,
%for any $d'_n$
there exists some $b_n \in B_x$ such that $d_n \preceq b_n \ll x$, we have $d'_n \ll x$ $\text{for all } n \geq 0$. Thus, if $x \not \in \text{K}(P)$, we take
%For the case $x \not \in \text{K}(P)$, we show now we can take
$U_n^x \coloneqq \twoheaduparrow d'_n$ $\text{for all } n\geq0$ since, given some $O \in \sigma(P)$ such that $x \in O$, then, by definition, there exists some $n_0 \geq0$ such that $d'_{n_0} \in O$ because $\sqcup (d'_n)_{n\geq0} =x$. In particular, $x \in \twoheaduparrow d'_{n_0} \subseteq O$.
Regarding the second statement, we notice, following the first part, $\text{for all } x \in P$, $O \in \sigma(P)$ such that $x \in O$ there exists some $n_0\geq0$ such that $x \in U_{n_0} \subseteq O$, where
\begin{equation}
 \label{count basis}   
(U_n)_{n\geq0} \coloneqq (\twoheaduparrow d)_{d \in D} \cup ( \twoheaduparrow x)_{x \in \text{K}(P)}
\end{equation}
 since
%such that, as we did for the first part,
%and $(\twoheaduparrow d)_{d \in D} \cup ( \twoheaduparrow x)_{x \in \text{K}(P)}$ is countable as
$\text{K}(P)$ is countable by assumption.
%we either consider $ (\twoheaduparrow d'_n)$ with $(d'_n) \subseteq D$ as countable local basis if $x \not \in \text{K}(P)$ or $\{\twoheaduparrow x\}$ if $x \in \text{K}(P)$. Thus, $(\twoheaduparrow d)_{d \in D} \cup ( \twoheaduparrow x)_{x \in \text{K}(P)}$ is a countable basis for $\sigma(P)$ if $\text{K}(P)$ is countable.
We conclude $\sigma(P)$ is second countable and, by Lemma \ref{second count = w-cont}, $P$ is $\omega$-continuous. In particular, $D \cup \text{K}(P)$ is a countable basis. 
\end{proof}

\begin{rem}[Implication for computability]
By Proposition \ref{count Debreu separable}, we can define computable elements (in the sense of Definition \ref{def: comp ele II}) on a continuous dcpo $P$ with a countable Debreu dense subset $D$ and countable compact elements $\text{K}(P)$ whenever $D \cup \text{K}(P)$ is effective. Moreover, if that is the case for two dcpos $P_0,P_1$, then we can define computable functions (in the sense of Definition \ref{def: comp func}) between them $f:P_0 \to P_1$.
\end{rem}

%We can apply now a characterization of first countability for continuous dcpos in \cite[Proposition 3.1.2]{martin2000foundation} to get $\sigma(P)$ is first countable.

Note that there exist Debreu separable dcpos where $\sigma(P)$ is not first countable (take, for example, the dcpo in \cite[Proposition 5]{hack2022computation}) and continuous Debreu separable dcpos where $\text{K}(P)$ is uncountable (like $P \coloneqq (A \cup B,\preceq)$ with $A \coloneqq \mathbb{R}$, $B \coloneqq (0,1]$ and $x \preceq y$ if and only if $x \leq y$ with $x,y \in B$ or $x \in A$ and $y \in B$, where $\mathbb{Q} \cap B$ is a Debreu dense subset and $\text{K}(P)=A$ is uncountable). Strengthening the order density assumption in Proposition \ref{count Debreu separable}, we can 
%We notice now, in Proposition \ref{cont order sep and min coun implies w cont}, we can
improve upon Proposition \ref{order sep has count weak basis} and construct countable \emph{bases} instead of \emph{weak bases}, as we show in Proposition \ref{cont order sep and min coun implies w cont}.
%if we assume the dcpo is also continuous.

\begin{prop}
\label{cont order sep and min coun implies w cont}
If $P$ is a continuous dcpo with a countable order dense subset $D\subseteq P$, then $D$ is a countable basis for $P\setminus \text{min}(P)$. Furthermore, if $P$ also has a countable Debreu upper dense subset, then $D \cup \text{min}(P)$ is a countable basis for $P$.
\end{prop}

\begin{proof}
We begin with the first statement. In order to do so, we start by showing the following lemma.
\begin{lem}
\label{lem: charac K(P)}
If $P$ is a continuous dcpo with a countable order dense subset $D\subseteq P$, then
\begin{equation*}
    \text{K}(P) = \text{min}(P).
\end{equation*}
\end{lem}

\begin{proof}
$(\subseteq)$ If $y \in P\setminus \text{min}(P)$, then we can construct an increasing chain $(d_n')_{n\geq0} \subseteq D$ (defined as in Lemma \ref{order dense chain}, \eqref{def d'_n II}) with $D\subseteq P$ a countable order dense subset
%which we be
chosen w.l.o.g. with the property $D \cap \text{min}(P)=\emptyset$, such that $\sqcup (d_n')_{n\geq0}=y$ and $y \not \in (d_n')_{n\geq0}$. Thus, $y \not \in \text{K}(P)$ by Lemma \ref{compact implies contained if supremum} and, hence, $\text{K}(P) \subseteq \text{min}(P)$. $(\supseteq)$ This is the case since $P$ is continuous and, hence, we must have, for all $x \in \text{min}(P)$, $B_x \coloneqq \{x \} \subseteq \twoheaddownarrow x \cap B$ for some basis $B \subseteq P$. That is, $x \in \text{K}(P)$. 
\end{proof}

 We finish noticing $Q \coloneqq P\setminus \text{min}(P)$ is a dcpo where $K(Q)=\emptyset$.
Thus, by \eqref{count basis}, 
%We can follow, thus, the proof of Proposition \ref{count Debreu separable} to obtain
$D \subseteq Q$ is a countable basis for $Q$. For the second statement, we notice, if $P$ has a countable Debreu upper dense subset, then $\text{min}(P)$ is countable by Lemma \ref{minP countable} and $\text{K}(P)$ is also countable, as $\text{K}(P) = \text{min}(P)$ by Lemma \ref{lem: charac K(P)}. Thus, by Proposition \ref{count Debreu separable}, $P$ has a countable basis. In fact, we have that $D \cup \text{min}(P)$ is a countable basis for $P$.
\end{proof}

\begin{rem}[Implication for computability]
By Proposition \ref{cont order sep and min coun implies w cont}, we can define computable elements and functions (in the sense of Definitions \ref{def: comp ele II} and \ref{def: comp func}) on a continuous dcpo $P$ with a countable order dense subset $D$ and a countable Debreu upper dense subset whenever $D \cup \text{min}(P)$ is effective.
\end{rem}

As we did with Proposition \ref{imme succ} for Proposition \ref{order sep has count weak basis}, we complement Proposition \ref{cont order sep and min coun implies w cont} with Proposition \ref{imme succ 2}, where we substitute order separability by a weaker property, namely, the existence of a countable Debreu dense subset where each element has a finite number of immediate successors. Notice, unlike Proposition \ref{cont order sep and min coun implies w cont}, Proposition \ref{imme succ 2} holds for the Cantor domain.

\begin{prop}
\label{imme succ 2}
%If $P$ is a continuous dcpo with a countable Debreu dense and Debreu upper dense subset $D \subseteq P$ such that each member of $D$ has a finite number of immediate successors then $P\setminus \text{min}(P)$ is $\omega$-continuous.
If $P$ is a continuous dcpo with a countable Debreu dense subset $D \subseteq P$ whose elements have a finite number of immediate successors, then $(D \cup \text{K}(P)) \setminus \text{min}(P)$ is a countable basis for $P\setminus \text{min}(P)$.
%is an $\omega$-continuous dcpo.
Furthermore, if $P$ also has a countable Debreu upper dense subset, then $D \cup \text{K}(P)$ is a countable basis for $P$.
\end{prop}

\begin{proof}
The first statement can be shown by emulating Lemma \ref{order dense chain} (as in the proof of Proposition \ref{imme succ}) to obtain, after applying Lemma \ref{compact implies contained if supremum} as in Lemma \ref{lem: charac K(P)}, that $\text{K}(P) \subseteq \text{min}(P) \cup Q_1$, with $Q_1$ defined as in \eqref{def Q1}, that is,
\begin{equation*}
    Q_1 = \{ y \in P| \exists x \in P \text{ s.t. } (x,y)=\emptyset\}.
\end{equation*} 
Since $Q_1$ is countable, we follow the proof of Proposition \ref{count Debreu separable} and conclude $P\setminus \text{min}(P)$ has a countable basis. For the second statement, we rely again on Lemma \ref{minP countable} and obtain that $\text{min}(P)$ is countable. We conclude that $\text{K}(P)$ is also countable and, thus, that $D \cup \text{K}(P)$ is a countable basis for $P$.
%conclude, as in Proposition \ref{cont order sep and min coun implies w cont}, that} $P$ is $\omega$-continuous.
\end{proof}

\begin{rem}[Implication for computability]
By Proposition \ref{imme succ 2}, we can define computable elements and functions (in the sense of Definitions \ref{def: comp ele II} and \ref{def: comp func}) on a continuous dcpo $P$ with a countable Debreu dense subset $D$ whose elements have a finite number of immediate successors and a countable Debreu upper dense subset whenever $D \cup \text{K}(P)$ is effective.
\end{rem}

As we show in Proposition \ref{strong prop 4}, we cannot improve upon Proposition \ref{imme succ 2} eliminating the assumption about the finiteness of the immediate successors of $D$. In particular, there are continuous Debreu upper separable dcpos where $\text{K}(P)$, and, hence, any basis, is uncountable. Actually, the result remains false if we add the assumption that $x \preceq y$ implies $x \ll y$ $\text{for all } x,y \in P$. We will return to this assumption in Section \ref{upper comp Debreu}.

\begin{prop}
\label{strong prop 4}
There exist continuous Debreu upper separable dcpos $P$ where $x \preceq y$ implies $x \ll y$ $\text{for all } x,y \in P$, but for which $\text{K}(P)$ is uncountable and, thus, no countable bases exist.
\end{prop}

\begin{proof}
Take the counterexample from the proof of Proposition \ref{cont Deb upper sep but not w-cont}, $P$. We will show every element in $P$ is compact, $\text{K}(P)=P$. If $x \in \Sigma^\omega$ and $x \preceq \sqcup A$ for some directed set $A$, then $x=\sqcup A$, as there is no other element above $x$. Thus, $x \in A$ since, as we showed in Proposition \ref{cont Deb upper sep but not w-cont}, $A$ would not be directed in the opposite case. Hence, $x \in \text{K}(P)$ and $\Sigma^\omega \subseteq \text{K}(P)$. If $x \in \Sigma^*$ and $x \preceq \sqcup A$ for some directed set $A$, then either $\sqcup A=x$ or $\sqcup A \in \Sigma^\omega$. If $\sqcup A=x$, then $A=\{x\}$ as there is no other element below $x$. If $\sqcup A \in \Sigma^\omega$, then $\sqcup A \in A$, as we argued above, and we have $x \in \text{K}(P)$. Thus, $\Sigma^*\subseteq \text{K}(P)$. We conclude $\text{K}(P)=P$. To finish, we only need some remarks. $\text{K}(P)=P$ implies $P$ is continuous, as we can take the directed set $A =\{x\} \subseteq \twoheaddownarrow x$ $\text{for all } x \in P$. Moreover, since $\text{K}(P) \subseteq B$ for any basis $B$ and $\text{K}(P)$ is uncountable, $P$ is not $\omega$-continuous. Notice, also, if $x \preceq y$, then we have $x \in \text{K}(P)$, which means $x \ll x \preceq y$ and $x \ll y$ $\text{for all } x,y \in P$.
\end{proof}

Since we tried to construct countable bases in Proposition \ref{cont order sep and min coun implies w cont} and Proposition \ref{imme succ 2} using order density properties, we intend now to do the converse. Notice, by Proposition \ref{weak basis implications}, $\omega$-continuous dcpos have countable Debreu upper dense subsets. However, there are $\omega$-continuous dcpos which are not Debreu separable. To see this, we can take the counterexample in Proposition \ref{weak basis not debreu sep} (as the countable weak basis defined there is actually a basis). Alternatively, we can use either majorization for any $n\geq3$ or the interval domain since, as we show in Lemma \ref{example domains} in the Appendix \ref{lemma 5 proof}, both are $\omega$-continuous and any of their Debreu dense subsets has cardinality $\mathfrak{c}$. We summarize this paragraph in the following statement. %in Proposition \ref{w-cont implies Deb upper dense}.

\begin{prop}
\label{w-cont implies Deb upper dense}
If $P$ is a dcpo with a basis $B \subseteq P$, then $B$ is a Debreu upper dense subset of $P$. However, there are $\omega$-continuous dcpos which are not Debreu separable. %In particular, $(\Lambda^n,\preceq_M)$ for any $n\geq 3$ and $(\mathcal{I}, \sqsubseteq)$ are $\omega$-continuous dcpos without a countable Debreu dense subset.
\end{prop}

%\begin{proof}
 %The first part of the statement comes from Proposition \ref{weak basis implications} as a basis is a weak basis while we can show the counterexample in Proposition 2 is $\omega$-continuous and use it for the second statement. We show in Lemma \ref{example domains} in the Appendix, both the set of classical states under majorization and the interval domain are $\omega$-continuous and, in both cases, any Debreu dense set has the cardinality of the continuum $\mathfrak{c}$.
%\end{proof}

\subsection{Conditionally connected Debreu upper separable dcpos}
\label{upper comp Debreu}

Although Debreu upper separability and $w$-continuity are not equivalent for general continuous dcpos (cf. Proposition \ref{strong prop 4} and \ref{w-cont implies Deb upper dense}), they do coincide for the Cantor domain $\Sigma^\infty$. In this section, we use one of its properties, \emph{conditional connectedness}, to show, in Theorem \ref{thm 2}, Debreu upper separability and $w$-continuity are equivalent for any conditionally connected dcpo. We begin by defining conditional connectedness. 

\begin{defi}[Conditionally connected partial order]
\label{cond connect}
A partial order $P$ is conditionally connected if, for any pair $x,y \in P$ for which there exists some $z \in P$ such that $x \preceq z$ and $y \preceq z$, we have $x$ and $y$ are comparable, i.e. $\neg(x \bowtie y)$.
\end{defi}

Note that conditional connectedness amounts to comparable elements conforming equivalent classes. We could have named partial orders fulfilling Definition \ref{cond connect} \emph{upward} conditionally connected to distinguish them from \emph{downward} conditionally connected, which has been used in order-theoretical approaches to thermodynamics \cite{roberts1968axiomatic,landsberg1970main,giles2016mathematical} and is referred to as \emph{conditionally connected}. Note that it is straightforward to characterize conditionally connected partial orders as those where every directed set is a chain (see Lemma \ref{upper comp charact} in the Appendix \ref{appen proofs}).

An example of a conditionally connected partial order is the Cantor domain since, if $x,y,z \in \Sigma^{\infty}$ are three strings such that $x$ and $y$ are prefixes of $z$, then it is clear that either $x$ is a prefix of $y$ or vice versa. A negative example comes from majorization
%is conditionally connected for $n=2$, the same is not true for any
$\text{for all } n\geq 3$, as one can pick, for example, the pair $x,y \in \Lambda^n$ where  $x=(0.6,0.2,0.2,0,..,0)$ and $y=(0.5,0.4,0.1,0,..,0)$, and, although $x,y \preceq (1,0,..,0) \in \Lambda^n$ holds, we have $x \bowtie y$.
%Aside from the Cantor domain being conditionally connected, the interest of this property lies in the fact it
Note that conditional connectedness weakens a common property in order theory and its applications, namely, totality \cite{debreu1954representation,lieb1999physics,bridges2013representations}.\footnote{A partial order $P$ is said to be \emph{total} if any pair $x,y \in P$ is comparable \cite{bridges2013representations}.}
%Notice order-theoretic approaches to thermodynamics succeeding \cite{landsberg1970main} assume totality instead of conditional connectedness \cite{lieb1999physics}.

Even though conditionally connected dcpos may be more suitable to be interpreted in terms of information, many relevant order-theoretical models of uniform computability are not conditionally connected, like majorization, and have no conditionally connected straightforward variation.
%version of them may not be straightforward to define.
Take for example the interval domain \eqref{interval domain}. It is not conditionally connected since, for any $x \in \mathbb{R}$ and $0<\varepsilon_i,\varepsilon'_i$ $i=1,2$ with $\varepsilon_1> \varepsilon'_1$ and $\varepsilon_2 <\varepsilon'_2$, we have $a \coloneqq [x-\varepsilon_1,x+\varepsilon_2] \sqsubset x$ and  $b \coloneqq [x-\varepsilon_1',x+\varepsilon_2'] \sqsubset x$ although $a \bowtie b$. However, in case $\varepsilon_1+\varepsilon_2 < \varepsilon'_1+\varepsilon'_2$, one would say $a$ contains more information than $b$ about \emph{any} $y \in a \cap b$, which is not
reflected by $\sqsubseteq$ although it is meant to represent information. The natural modification of the interval domain which takes care of this is, however, not a dcpo. We would like to define the binary relation $(\mathcal{I},\sqsubseteq_2)$, where $[a,b] \sqsubseteq_2 [c,d] \iff [a,b] \cap [c,d] \neq \emptyset$ and $\ell([a,b]) \geq \ell([c,d])$ with $\ell([a,b]) \coloneqq b-a$ being the length of the interval. Notice in order to turn $\sqsubseteq_2$ into a partial order we need to identify all intervals which share some intersection and have the same length. One can easily see, however, transitivity does not hold, hence, $\sqsubseteq_2$ is not even a partial order.

%We denote the intervals identified by this and the resultin partial order as $(I,\sqsubseteq_2)$. Notice it is directed complete as for any directed set $D$ we have $\sqcup D=i$ where $i$ is the equivalence class of intervals which contain some $x \in \bigcap_{d \in D} d \neq \emptyset$ (THIS MAY BE WRONG!) and $\ell(i)= \inf_{d \in D} \ell(d)$.

In the remainder of this section, we study the properties of conditionally connected dcpos to end up showing, in Theorem \ref{thm 2}, the equivalence of Debreu upper separability and $w$-continuity for them. From Proposition \ref{strong prop 4}, we know Debreu upper separability is insufficient to assure $\text{K}(P)$ is countable in a continuous dcpo, which is necessary in order to hope for a countable basis $B \subseteq P$, since $\text{K}(P) \subseteq B$. Nevertheless, as we will see in Proposition \ref{K(P) countable}, $\text{K}(P)$ is countable for conditionally connected Debreu separable dcpos. In order to arrive to that result, we first characterize both the way-below relation (Proposition \ref{way-below upper comp}) and $\text{K}(P)$ (Proposition \ref{charact compact}) for conditionally connected dcpos. %It is important, thus, to characterize $\text{K}(P)$ in a simple way. In order to do so, we first characterrize $\ll$ for conditionally connected dcpos in Proposition \ref{way-below upper comp}.

\begin{prop}[Conditionally connected characterization of $\ll$]
\label{way-below upper comp}
If $P$ is a conditionally connected dcpo, then we have, for all $ x, y \in P$,
 \begin{equation*}
 x \ll y \iff
     \begin{cases}
      x \prec y \text{, or}\\
    x = y \text{ and there is no directed set } A \subseteq P\setminus \{x\}\\
    \text{ such that } \sqcup A = x. 
     \end{cases}
 \end{equation*}
\end{prop}

\begin{proof}
Take $x,y \in P$. If $x \ll y$ then by definition we have $x \preceq y$ and either $x \neq y$, implying $x \prec y$, or $x \in \text{K}(P)$ and there is no directed set $A \subseteq P\setminus \{x\}$  such that $\sqcup A = x$  by Lemma \ref{compact implies contained if supremum}. For the converse, we first show, if $P$ is a conditionally connected dcpo,
%$x, y \in P$ such that
$x\preceq y$ and $A$ is a directed set such that $y \preceq \sqcup A$, then either there exists $a \in A$ such that $x \preceq a$ or $x=y=\sqcup A$. Consider, thus, a directed set $A$ such that $y \preceq \sqcup A$ and $x \preceq y$. We have $x \preceq \sqcup A$ by transitivity and $\neg(x \bowtie a)$ $\text{for all } a \in A$ by conditional connectedness. If there exists some $a \in A$ such that $x \preceq a$, then we have finished. Otherwise, we have $a \prec x$ $\text{for all } a \in A$, thus, $\sqcup A \preceq x$ by definition of supremum and $x=y= \sqcup A$ by antisymmetry. We can now use this property to finish the proof. If $x \prec y$, then, given any directed set $A$ such that $y \preceq \sqcup A$, we get there exists some $a \in A$ such that $x \preceq a$ by the property we proved above, since we have $x \neq y$. As a consequence, $x \ll y$ by definition of $\ll$. Consider now some $x \in P$ such that there is no directed set $A \subseteq P\setminus \{x\}$ fulfilling $\sqcup A = x$. Given some directed set $A$ such that $x \preceq \sqcup A$, we get $\neg(x \bowtie a)$ $\text{for all } a \in A$ by conditional connectedness. If there exists $a \in A$ such that $x \preceq a$, then $x \ll x$. If the contrary holds, then $a \prec x$ $\text{for all } a \in A$ and $\sqcup A \preceq x$ by definition of supremum. Thus, $x = \sqcup A$ by antisymmetry and $x \in A$ by assumption, contradicting the fact $a \prec x$ $\text{for all } a \in A$.
\end{proof}

Note that the characterization of compact elements in Proposition \ref{way-below upper comp} also holds under the hypothesis that $P$ is continuous, which is, as we show in Proposition \ref{upper comp is conti}, weaker than conditional connectedness. However, $x \prec y$ does not necessarily imply $x \ll y$ in case $P$ is just continuous. (To convince ourselves about this, we can take, for example, majorization and compare \eqref{majorization},
\begin{equation*}
    x \preceq_M y \iff (\text{for all } i<n)\text{ } s_i(x) \leq s_i(y),
\end{equation*}
with \eqref{way-below majo},
\begin{equation*}
x \ll_M y \iff x = \perp \text{ or } s_k(x)<s_k(y) \text{ } \text{for all } k<n.)
\end{equation*}
Notice, also, there is no non-essential information in conditionally connected dcpos, that is, whenever a process converges to some $x\in P$, it must eventually gather the information of any element which contains information about $x$, since $y \prec x$ implies $y \ll x$ $\text{for all } x,y \in P$. Using Proposition \ref{way-below upper comp}, we can characterize the compact elements $\text{K}(P)$ for conditionally connected dcpos. We do so in Proposition \ref{charact compact}, for which we need another definition. 
%which uses some extra terminology.
We say $x \in P$ is \emph{isolated} if there exists some $v_x \in P$ such that $v_x \prec x$ and, $\text{for all } y \in P$, we have $y \preceq v_x$ provided $y \prec x$. We denote the set of isolated elements by $\text{I}(P)$ and note $I(\Sigma^\infty)= \Sigma^*\setminus \{\perp\}$.

\begin{prop}[Conditionally connected characterization of $\text{K}(P)$]
\label{charact compact}
If $P$ is a conditionally connected dcpo, then
\begin{equation*} 
\text{K}(P)= \text{I}(P) \cup \text{min}(P).
\end{equation*}
\end{prop}

\begin{proof}
$(\supseteq)$ $\text{min}(P) \subseteq \text{K}(P)$ holds by the characterization of $\text{K}(P)$ in Proposition \ref{way-below upper comp} since, if $x \in \text{min}(P)$, the only directed set $A$ such that $\sqcup A=x$ is $A=\{x\}$. Take $x \in \text{I}(P)$ and assume $x \not \in \text{K}(P)$. Then, there exists a directed set $A \subseteq P\setminus \{x\}$ such that $\sqcup A=x$ by Proposition \ref{way-below upper comp}. Since $a \preceq v_x$ $\text{for all } a \in A$ by definition of isolated element, we would have $x= \sqcup A \preceq v_x$ by definition of supremum, a contradiction. Thus, $\text{I}(P) \subseteq \text{K}(P)$. $(\subseteq)$ Take $x \not \in \text{I}(P) \cup \text{min}(P)$ and define $A_x:= \{z \in P| z \prec x\}$. Since $x \notin \text{min}(P)$, we have $A_x \neq \emptyset$ and, by conditional connectedness, $A_x$ is a chain. In particular, $\sqcup A_x $ exists and, since $a \prec x$ $\text{for all } a \in A_x$, $\sqcup A_x \preceq x$ by definition of supremum. If $\sqcup A_x = x$, then $x \notin \text{K}(P)$ by Proposition \ref{way-below upper comp}, since $x \not \in A_x$. Conversely, $\sqcup A_x \prec x$ would contradict the fact $x \not \in \text{I}(P)$ as, by definition of $\sqcup A_x$, any $y \in P$ such that $y \prec x$ would fulfill $y \preceq \sqcup A_x$. Thus, we could take $v_x \coloneqq \sqcup A_x$ and conclude $x \in \text{I}(P)$, a contradiction. In summary, $x \in \text{K}(P)$ implies $x \in \text{I}(P) \cup \text{min}(P)$.
\end{proof}

Notice, although $\text{I}(P) \cup \text{min}(P) \subseteq \text{K}(P)$, $\text{K}(P) \not \subseteq \text{I}(P) \cup \text{min}(P)$ if we weaken the hypothesis from \emph{conditionally connected} to \emph{continuous} in
Proposition \ref{charact compact} (take, for example, the dcpo in Proposition \ref{strong prop 4}). As we already know from Lemma \ref{minP countable}, $\text{min}(P)$ is countable whenever countable Debreu upper dense subsets exist. Moreover, for conditionally connected dcpos, the set of isolated elements $\text{I}(P)$ is countable as well, if Debreu upper separability holds. Because of the characterization of the compact elements in Proposition \ref{charact compact}, these two facts result in Proposition \ref{K(P) countable}, where we show conditionally connected dcpos which are Debreu upper separable have a countable number of compact elements.

\begin{prop}
\label{K(P) countable}
If $P$ is a conditionally connected and Debreu upper separable dcpo, then $\text{K}(P)$ is countable.
\end{prop}

\begin{proof}
Since $\text{K}(P)=\text{min}(P)\cup \text{I}(P)$ by Proposition \ref{charact compact} and $\text{min}(P)$ is countable whenever there exists some countable Debreu upper dense set by Lemma \ref{minP countable}, we only need to show $\text{I}(P)$ is countable to get the result. Take $D \subseteq P$ a countable Debreu dense subset and $D':= \text{I}(P)\setminus D$, where $\text{I}(P)$ the set of isolated points. If $x \in D'$, then there exists some $d \in D$ such that $v_x \preceq d \preceq x$, where $v_x \in P$ has the property that $\text{for all } y \in P$ such that $y \prec x$ we have $y \preceq v_x$. By definition of $v_x$ and since $x \not \in D$, $v_x =d$ holds. We define now a map
  \begin{alignat*}{3}
    f: \text{ } &D' &&\rightarrow &&
    \text{ }D\\
    &x &&\mapsto &&\text{ }v_x.
\end{alignat*}
%$f:D' \rightarrow D$, $x \mapsto v_x$.
If we show $f$ is injective, then we get $D'$ is countable and, thus, $\text{I}(P)$ also, since we have $\text{I}(P) \subseteq D' \cup D$. Take, thus, $x,y \in D'$ such that $f(x) = f(y)$ and assume $x \neq y$. Notice $x \prec y$ contradicts the definition of $v_y$, since it would mean $v_y = v_x \prec x \prec y$. Equivalently, we can discard $y \prec x$ by definition of $v_x$. If $x \bowtie y$, then, by Debreu upper density, there exists some $d \in D$ such that $x \bowtie d \preceq y$. Since $y \not \in D$, we have $x \bowtie d \prec y$. Thus, $d \preceq v_y = v_x \prec x$ which contradicts, by transitivity, the fact that $x \bowtie d$. Since $x \neq y$ leads to contradiction, we conclude $f(x) = f(y)$ implies $x = y$ $\text{for all } x,y \in D'$ and, hence, $f$ is injective.
\end{proof}

To conclude $\text{I}(P)$ is countable in Proposition \ref{K(P) countable}, it is sufficient for $P$ to be a partial order instead of a dcpo. Notice, also, we do not need the elements of $D$ in its proof to have a finite number of immediate successors as in Proposition \ref{imme succ 2}. In fact, there exist conditionally connected dcpos where any such a $D$ has elements with an infinite number of immediate successors
%of an element in $D$ could be infinite under the hypotheses of Proposition \ref{K(P) countable}
(take, for example, $P$ the dcpo defined analogously to \eqref{Cantor domain} but using the natural numbers as alphabet $\Sigma \coloneqq \mathbb{N}$ and note that any such $D\subseteq P$ will have elements with a countably infinite number of immediate successors).
%(take, for example, the dcpo $P \coloneqq (\mathbb{Z},\preceq)$ where $n \preceq m \iff n=m \text{ or } n=0$ and notice, while $P$ fulfills the hypotheses in in Proposition \ref{K(P) countable}, any element in $\mathbb{Z}\setminus \{0\}$ is an immediate successor of $0$ THIS DOESNT WORK.. WE DO NOT REALLY NEED TO TAKE 0 AS PART OF THE DEBREU SET.. RIGHT??).
Actually, Proposition \ref{K(P) countable} shows any conditionally connected and Debreu upper separable partial order has a countable number of jumps, since the set of jumps is equinumerous to the set of isolated elements for conditionally connected dcpos (one can see that the map
 \begin{alignat*}{3}
    \phi: \text{ } &\text{J}(P) &&\rightarrow &&
    \text{ }\text{I}(P)\\
    &(x,y) &&\mapsto &&\text{  }y
\end{alignat*}
%$\phi: J(P) \to \text{I}(P)$, $(x,y) \mapsto y$ 
is bijective, where $\text{J}(P)$ denotes the set of jumps of $P$). We improve, thus, on the classical result that the number of jumps is countable for any Debreu separable \emph{total} order \cite{debreu1954representation} (see also \cite[Proposition 1.4.4]{bridges2013representations}), as total partial orders are conditionally connected and Debreu separability coincides with Debreu upper separability for total orders. Notice, if we eliminate conditional connectedness, the number of jumps could be uncountable even if Debreu upper separability holds (take, for example, the dcpo in the proof of Proposition \ref{cont Deb upper sep but not w-cont}). In fact, there are continuous Debreu upper separable dcpos where $\text{K}(P)$ is uncountable (see Proposition \ref{strong prop 4}) and conditionally connected dcpos where $\text{K}(P)$ is uncountable (for example, any uncountable set with the trivial partial order).

Before proving Theorem \ref{thm 2}, we show, in Proposition \ref{upper comp is conti}, conditionally connected dcpos are continuous. This fact plus the countability of $\text{K}(P)$ from Proposition \ref{K(P) countable} and some results from Sections \ref{density in compu} and \ref{cont Debreu separable} will allow us to derive Theorem \ref{thm 2}.

\begin{prop}
\label{upper comp is conti}
If $P$ is a conditionally connected dcpo, then $P$ is a basis for $P$.
\end{prop}

\begin{proof}
We will show $B \coloneqq P$ is a basis whenever $P$ is conditionally connected. We need to show there exists a directed set $B_x \subseteq \twoheaddownarrow{x} \cap P$ such that $x = \sqcup B_x$ $\text{for all } x \in P$.
%and $a \ll x$ $\text{for all } a \in A$.
If $x \in \text{K}(P)$, we can take $B_x \coloneqq \{x\}$. If $x \not \in \text{K}(P)$, then, by Proposition \ref{way-below upper comp}, there exists a directed set $B_x \subseteq P\setminus \{x\}$ such that $\sqcup B_x = x$.
%and $x \not \in A$
Thus, $b \prec x$ and, again by Proposition \ref{way-below upper comp}, we have $b \ll x$ $\text{for all } b \in B_x$.
\end{proof}

Notice, if $P$ is conditionally connected, then, by the characterization of $\text{K}(P)$ in Proposition \ref{charact compact} plus the fact it has a basis $B\subseteq P$ (Proposition \ref{upper comp is conti}), any $x \in P\setminus (\text{I}(P) \cup \text{min}(P))$ is a non-trivial element of $B$. If we also assume $P$ has a countable Debreu dense subset $D \subseteq  P$, we can then find an increasing sequence $(x_n)_{n\geq0} \subseteq D$ such that $x = \sqcup (x_n)_{n\geq0}$ and $x \not \in (x_n)_{n\geq0}$ (see Lemma \ref{countable covers}) and we can easily see $x_n \ll x$ $\text{for all } n \geq0$. Thus, if we assume $P$ is a conditionally connected Debreu upper separable dcpo, we can construct a countable basis using $\text{K}(P)$ (countable by Proposition \ref{K(P) countable}) and a countable Debreu dense subset. We show this and its converse in Theorem \ref{thm 2}. 

\begin{teo}
\label{thm 2}
If $P$ is a conditionally connected dcpo, then the following statements are equivalent:
%\vspace{-20pt}
%\begin{enumerate}[label=(\textit{\roman*})]
\begin{enumerate}[label={(\arabic*)}]
\item $P$ is Debreu upper separable.
\item $P$ is $\omega$-continuous.
\item $P$ is Debreu separable and $\text{K}(P)$ is countable.
\end{enumerate}
\end{teo}

\begin{proof}
We begin showing $(1)$ implies $(2)$. Given that $\text{K}(P)$ is countable by Proposition \ref{K(P) countable}, we can apply Proposition \ref{count Debreu separable}, since conditionally connected dcpos are continuous by Proposition \ref{upper comp is conti}, and get the result. We proceed now to show $(2)$ implies $(3)$. If $B \subseteq P$ is a countable basis, then $\text{K}(P)\subseteq B$ is also countable.
%since $K(P) \subseteq B$.
Take $x,y \in P$ such that $x \prec y$. We can follow Proposition \ref{weak basis implications} to get some $b_0 \in B$ such that $b_0 \preceq y$ and $\neg(b_0 \preceq x)$. Since we have both $b_0 \preceq y$ and $x \prec y$, then $\neg(x \bowtie b_0)$ holds by conditional connectedness.
%, we also get $\neg(x \bowtie b_0)$ as we have both $b_0 \preceq y$ and $x \preceq y$.
As a result, we have $x \prec b_0 \preceq y$. Thus, $B$ is a countable Debreu dense subset of $P$.   %There exists $B_y \subseteq B$ such that $\sqcup B_y = y$. Notice $b \preceq y$ $\text{for all } b \in B_y$ and if $b \preceq x$ $\text{for all } b \in B_y$ then $y \preceq x$ leading to contradiction. There exists, thus, $b_0 \in B_y$ such that $\neg( b_0 \preceq x)$. Since $b_0 \preceq y$ and $x \prec y$ by conditional connectedness $\neg(x \bowtie b_0)$ leading to $x \prec b_0$, implying for any pair $x,y \in P$ there exists some $b \in B$ such that $x \preceq b \preceq y$. Thus, $B$ is a countable Debreu dense subset of $P$.
We finish showing $(3)$ implies $(1)$. We will show $D' \coloneqq D \cup \text{K}(P)$ is a countable Debreu upper dense subset, where $D\subseteq P$ is a countable Debreu dense subset. Given any pair $x,y \in P$ such that $x \bowtie y$, we need to show there exists some $d \in D'$ such that $x \bowtie d \preceq y$. If $y \in \text{K}(P)$, then we take $d \coloneqq y$ and have finished. If $y \not \in \text{K}(P)$, then there exists a directed set $A \subseteq P$ such that $\sqcup A = y$ and $y\not \in A$ by Proposition \ref{way-below upper comp}. By Lemma \ref{countable covers}, thus, there is a directed set $D_y \subseteq D$ such that $\sqcup D_y = y$ and $y \not \in D_y$. Analogously to the proof of
%follow again
Proposition \ref{weak basis implications}, we can conclude there exists some $d \in D_y$ such that $d \prec y$ and $\neg(d \preceq x)$. Since $x \prec d$ would imply $x \prec y$ by transitivity, which is a contradiction, we obtain $x \bowtie d \prec y$.
\end{proof}

\begin{rem}[Implication for computability]
The proof of Theorem \ref{thm 2} shows we can introduce uniform computability (for both elements and functions) in conditionally connected dcpos which have a countable subset $D$ which is both Debreu dense and Debreu upper dense provided the countable basis $D \cup \text{I}(P) \cup \text{min}(P)$ is effective.
\end{rem}

We conclude from Theorem \ref{thm 2} the coincidence between Debreu upper separability and $w$-continuity is not a specific fact of the Cantor domain but it holds for any conditionally connected dcpo. Note that we cannot substitute \emph{conditional connectedness} by the weaker pair of hypotheses which includes both continuity of $P$ and the fact $x \preceq y$ implies $x \ll y$ $\text{for all } x,y \in P$. In that scenario, $(1)$ does not imply $(2)$ in Theorem \ref{thm 2} (see Proposition \ref{strong prop 4}). This is relevant, since the second hypothesis in the pair is a rather strong property of conditionally connected dcpos.

In Corollary \ref{upper separable},
we refine Theorem \ref{thm 2},
characterizing a stronger order density property for conditionally connected dcpos. We say a partial order is \emph{upper separable} if there exists a countable subset $D \subseteq P$ which is order dense and such that, for any pair $x,y \in P$ where $x \bowtie y$, there exists some $d \in D$ such that $x \bowtie d \prec y$ \cite{ok2002utility}. Note that the Cantor domain is not upper separable, as it is not order separable. Moreover, if $s \in \Sigma^*$ and $\beta, \gamma \in \Sigma$ $\beta \neq \gamma$, then we have $s\beta \bowtie s\gamma$, although, $\text{for all } t \in \Sigma^*$ such that $t \prec s\beta$, we have $t \prec s\gamma$.

\begin{coro}
\label{upper separable}
If $P$ is a conditionally connected dcpo, then the following statements are equivalent:
%\vspace{-10pt}
\begin{enumerate}[label={(\arabic*)}]
    \item $P$ is upper separable
    \item $P$ is $\omega$-continuous and either $\text{K}(P)=\{\perp\}$ or $\text{K}(P)=\emptyset$ holds.
\end{enumerate}
\end{coro}

\begin{proof}
We begin by showing $(1)$ implies $(2)$. Take $D \subseteq P$ a countable upper dense subset. By Theorem \ref{thm 2}, we have $P$ is $\omega$-continuous and, by Lemma \ref{lem: charac K(P)}, we obtain that $\text{K}(P) = \text{min}(P)$. We finish by noticing, if we have $x,y \in \text{min}(P)$ $x \not = y$, then $x \bowtie y$ by definition and there exists $d \in D$ such that $x \bowtie d \prec y$ by upper separability, contradicting the fact $y \in \text{min}(P)$. Thus, we have $|\text{min}(P)| \leq 1$. Assume there exists some $x \in \text{min}(P)$ and take $y \in P\setminus \text{min}(P)$. Notice $y \prec x$ contradicts the fact $x \in \text{min}(P)$ and so does, by upper separability, $y \bowtie x$. Thus, we have $x \prec y$ and $x=\perp$. Hence, we either have 
$\text{K}(P)=\{\perp\}$, or $\text{K}(P)=\emptyset$. We finish showing $(2)$ implies $(1)$. Take $B \subseteq P$ a countable basis and $x,y \in P$ such that $x \prec y$. We can follow the argument in Theorem \ref{thm 2} ($(2)$ implies $(3)$) to get some $b_0 \in B$ such that $x \prec b_0 \ll y$. Since $y \neq \perp$, we have $y \not \in \text{K}(P)$ and, thus, $x \prec b_0 \prec y$. Take now $x,y \in P$ such that $x \bowtie y$. Since we have a countable basis, we fulfill the hypotheses in Theorem \ref{thm 2} $(3)$ and we can follow the argument there (when proving $(3)$ implies $(1)$) to conclude there exists some $b_0 \in B$ such that $x \bowtie b_0 \prec y$.
%Notice, if we have either $\text{K}(P) = \emptyset$ or $\text{K}(P)= \{\perp\}$, then $\text{I}(P)=\emptyset$ by Proposition \ref{charact compact}. We consider $D \subseteq P$ a countable Debreu dense subset which exists by Theorem \ref{thm 2}. We notice first for any pair $x,y \in P$ $x \prec y$ there exist some $z_0 \in P$ such that $x \prec z_0 \prec y$ as, otherwise, $y \in \text{I}(P)$: if we assume this is not the case, then $\text{for all } z \prec y$ we would have, as $z_0 \bowtie z$ contradicts conditional connectedness, $z \preceq z_0$ and $y \in \text{I}(P)$ with $v_y = z_0$. Thus, applying the same argument twice, we get there exist $z_0,z_1 \in P$ such that $x \prec z_0 \prec z_1 \prec y$ and, by Debreu separability, there is some $d \in D$ such that $x \prec z_0 \preceq d \preceq z_1 \prec y$. Thus, there exists some $d \in D$ such that $x \prec d \prec y$ which means $D$ is a countable order dense subset of $P$. Notice, for any pair $x,y \in P$ $x \bowtie y$, we have $y \not \in \text{K}(P)$ as, in case  $y \in \text{K}(P) \neq \emptyset$, then $y=\perp$ and $y \preceq x$, a contradiction. Thus, if $x \bowtie y$, we have $y \not \in \text{K}(P)$ and we can follow the proof of $(3)$ implies $(1)$ in Theorem \ref{thm 2} to conclude there is some $d \in D$ such that $x \bowtie d \prec y$.
\end{proof}

%Notice Corollary \ref{upper separable} applies to $(I,\sqsubseteq_2)$ since it is upper separable: given $i \sqsubset j$ we take some $x \in i \cap j$ and some $q \in \mathbb{Q}$ $\ell(j)<q<\ell(1)$. There always exists some $[p,p']$ where $p-p'=q$ and $x \in [p.p']$; give $i \bowtie j$ we have $i_2 < j_1$, say, and we consider some $i_2<p<j_1$ and some $p<p'$ where $p'-p>\ell(j)$. 

Notice, a dcpo to which Corollary \ref{upper separable} applies can be found in the proof of Proposition \ref{not w-algebraic}. To finish this section, we notice the Cantor domain is not only $\omega$-continuous but $\omega$-algebraic. A dcpo $P$ is \emph{algebraic} if it has a basis $B\subseteq P$ consisting of compact elements $B = \text{K}(P)$ and \emph{$\omega$-algebraic} if such a basis is countable \cite{abramsky1994domain}. The Cantor domain is actually $\omega$-algebraic, as $B \coloneqq \{\perp\} \cup \Sigma^*$ is a countable basis and $B=\text{min}(P) \cup \text{I}(P) = \text{K}(P)$. Nevertheless, we cannot substitute \emph{$\omega$-continuous} by \emph{$\omega$-algebraic} in Theorem \ref{thm 2}, as we show in Proposition \ref{not w-algebraic}.

\begin{prop}
\label{not w-algebraic}
There exist conditionally connected Debreu upper separable dcpos which are not $\omega$-algebraic.
\end{prop}

\begin{proof}
We can consider $P$ the interval $[0,1]$ with the usual order $\leq$. Clearly, it is conditionally connected (it is total) and Debreu upper separable (because of $\mathbb{Q} \cap [0,1]$). However, while $B \coloneqq \mathbb{Q} \cap [0,1]$ is a countable basis, we have $\text{K}(P)=\{0\}$. Thus, $P$ is not $\omega$-algebraic.
\end{proof}
%To show this we introduce a new dcpo. We denote by $\Sigma^\infty \cup (0,1)^{\mathbb{N}}$ the partial order that results from adding $(0,1)$ between any consecutive $a,b \in \Sigma^*$. Given $I_{ab}$ the interval $(0,1)$ between $a,b \in \Sigma^*$ $a \prec b$ then $x \preceq y$ for $x,y \in I_{ab}$ iff $y-x \geq 0$ and $a \prec y \prec b$ for any $y \in I_{ab}$. We finish by adding the missing relations for transitivity to hold.

%We finish noticing $\Sigma^\infty \cup (0,1)^{\mathbb{N}}$ supports our claim: it inherits conditional connectedness from $\Sigma^\infty$ and  has $\mathbb{Q}^{\mathbb{N}}$ upper separable dcpo, it is not $\omega$-algebraic}.

To summarize, the main results in this section are Propositions \ref{count Debreu separable}, \ref{cont order sep and min coun implies w cont} and \ref{imme succ 2}, and Theorem \ref{thm 2}. In the first one, we show how one can profit from Debreu density and the countability of $\text{K}(P)$ to introduce computability on a continuous dcpo. In the second, we avoid the requirement on $\text{K}(P)$ and use the stronger property of order separability to introduce computability on all non-minimal elements. In the third one, we show that we can achieve the same results as in Proposition \ref{cont order sep and min coun implies w cont} by asking for the existence of a countable Debreu separable subset whose elements only have a finite number of immediate successors. Moreover, we extend computability to the whole dcpo in Propositions \ref{cont order sep and min coun implies w cont} and \ref{imme succ 2} by also requiring the existence of a countable Debreu upper dense subset. Lastly, in Theorem \ref{thm 2}, we use conditional connectedness to show the equivalence between having a countable basis and being Debreu upper separable.

\section{Relating order density with completeness and continuity}
\label{complete and continuous}
%\subsection{Order density and  other countability restrictions in computability}

We have focused, in Sections \ref{density in compu} and \ref{density in uniform compu}, on the relationship between order density properties and both countable weak bases and bases. However, they are related to other properties of these approaches to computability. %influence of order density in the approach to computability using dcpos goes beyond it.
In this section, we relate them to both order completeness and to a weak form of computable functions, namely, continuity in the Scott topology.
%explore this influence further in this section.

\subsection{Completeness}

Regarding completeness, we show, in Proposition \ref{equi D-sep partial orders}, Debreu separable partial orders where increasing sequences have a supremum are directed complete. Note that this implication was obtained in \cite[Proposition 2.5.1]{martin2000foundation} under a different hypothesis , namely, second countability of the Scott topology. Since second countability of $\sigma(P)$ is equivalent to $\omega$-continuity for continuous dcpos $P$ by Lemma \ref{second count = w-cont}, there are Debreu separable partial orders whose Scott topology in not second countable (see Proposition \ref{strong prop 4}) and vice versa (see Proposition \ref{w-cont implies Deb upper dense}).

\begin{prop}
\label{equi D-sep partial orders}
If $P$ is a Debreu separable partial order, then the following conditions are equivalent:
%\vspace{-10pt}
%\begin{enumerate}[label=(\textit{\roman*})]
\begin{enumerate}[label={(\arabic*)}]
\item Every directed set has a supremum.
\item Every increasing sequence has a supremum.
\end{enumerate}
\end{prop}

\begin{proof}
Since $(1)$ implies $(2)$ by definition, we only show the converse is true. Take $A\subseteq P$ a directed set and $D\subseteq P$ a countable Debreu dense subset. If there exists $a' \in A$ such that $a \preceq a'$ $\text{for all } a \in A$, then $\sqcup A=a'$ and we have finished. In the opposite case, we can argue as in Lemma \ref{always next} and obtain that, for every $a \in A$, there exists some $b \in A$ such that $a \prec b$. We can, thus, construct an increasing sequence $(d'_n)_{n\geq0} \subseteq D$ as in Lemma \ref{countable covers} for which $\sqcup (d'_n)_{n\geq0}$ exists by hypothesis. Since $\text{for all } a \in A$ there exists some $n\geq0$ such that $a \preceq d'_n$ by construction of $(d'_n)_{n\geq0}$, we have, by transitivity, that $a \preceq \sqcup (d'_n)_{n\geq0}$ $\text{for all } a \in A$. If we assume there exists some $x \in P$ such that $\text{for all } a \in A$ we have $a \preceq x$, then, as $\text{for all } n \geq0$ there exists some $a_n \in A$ such that $d'_n \preceq a_n$ by construction of $(d'_n)_{n\geq0}$, we have $d'_n \preceq x$ $\text{for all } n \geq0$ by transitivity. Thus, by definition of supremum, $\sqcup (d'_n)_{n\geq0} \preceq x$. We conclude that $\sqcup A = \sqcup (d'_n)_{n\geq0}$. In particular, $\sqcup A$ exists.
\end{proof}

Notice, whenever a partial order is Debreu separable, we can associate to each directed set an increasing sequence (defined by \eqref{def D_A} and \eqref{def d'_n}), which is intimately related to it in the sense that, by Proposition~\ref{equi D-sep partial orders}, the directed set has some element as supremum if and only if the increasing sequence has the same supremum.
%by both Lemma \ref{countable covers} and, conversely, whenever the increasing sequence has a supremum, the directed set has the same supremum by Propostion \ref{equi D-sep partial orders}.

%Notice, to show $(2)$ implies $(1)$, we use a countable Debreu dense subset to construct an increasing chain for each directed set which is intimately connected to it: whenever the directed set has a supremum, the increasing set has the same supremum (Lemma \ref{countable covers}) and, whenever the increasing sequence has a supremum, the directed set has the same supremum (Propostion \ref{equi D-sep partial orders}).
%In order to see this, one can use three results we will state in the following: the counterexamples in Proposition \ref{strong prop 4} and Proposition \ref{w-cont implies Deb upper dense} plus Lemma \ref{second count = w-cont}.

\subsection{Continuity}

Debreu separability is also closely related to the Scott topology. In particular, it is related to computable maps between dcpos as well. More precisely, to the wider set of continuous functions in the Scott topology. As we show in Theorem \ref{thm 3},
 %While the image of a countable weak basis completely determines a continuous function between dcpos,
 continuous functions between dcpos $f:P \to Q$ coincide with sequentially continuous functions whenever $P$ is Debreu separable. In fact, under the same hypothesis, they also coincide with monotones which preserve suprema of increasing sequences. Note that we say a monotone function $f:P \to Q$ \emph{preserves suprema of increasing sequences} if, given any increasing sequence $(x_n)_{n\geq0}$, we have $\sqcup (f(x_n))_{n\geq0} = f(\sqcup (x_n)_{n\geq0})$.

\begin{teo}
\label{thm 3}
If $P$ is a Debreu separable dcpo, $Q$ is a dcpo and $f:P \rightarrow Q$ is a map, then the following are equivalent:
%\begin{enumerate}[label=(\textit{\roman*})]
\begin{enumerate}[label={(\arabic*)}]
%\vspace{-10pt}
\item $f$ is sequentially continuous.
\item $f$ is monotone and preserves suprema of increasing sequences.
\item $f$ is continuous.
\end{enumerate}
%. If $P$ is Debreu separable, then $f$ is continuous if and only if it is sequentially continuous.
\end{teo}

\begin{proof}
We begin showing $(1)$ implies $(2)$. We first prove $f$ is monotone. We will do so by contrapositive, that is, we will show, if $f$ is not monotone, then it is not sequentially continuous. If $f$ is not monotone, there exist $x,y \in P$ such that $x \preceq y$ and $\neg(f(x) \preceq f(y))$. Consider, then, the sequence $(x_n)_{n\geq0}$, where $x_n \coloneqq y$ $\text{for all } n \geq 0$. Note that $(x_n)_{n\geq0}$ converges to $x$ since, given some $O \in \sigma(P)$ such that $x \in O$, we have $x_n = y \in O$ $\text{for all } n \geq 0$ since $x \preceq y=x_n$. However, since we have $\neg(f(x) \preceq f(y))$, then, as we know from \eqref{charac order by topo},
\begin{equation*}
%\label{charac order by topo}
    x \preceq y \iff  x \in O \text{ implies } y \in O \text{ } \text{for all } O \in \sigma(P)
\end{equation*}
and there exists some $O \in \sigma(Q)$ such that $f(x) \in O$ and $f(y) \not \in O$. Hence, $f(x_n)=f(y) \not \in O$ and $(f(x_n))_{n\geq0}$ does not converge to $f(x)$. Thus, $f$ is not sequentially continuous and, by contrapositive, any sequentially continuous $f$ is monotone. We show now $f$ preserves suprema of increasing sequences. Take an increasing sequence $(x_n)_{n \geq 0}$ and notice it converges to $\sqcup (x_n)_{n \geq 0}$ in $\sigma(P)$. This is the case since, given $O \in \sigma(P)$ such that $\sqcup (x_n)_{n \geq 0} \in O$, then, by definition of $\sigma(P)$, there exists some $n_0 \geq0$ such that $x_{n_0} \in O$ and, since $(x_n)_{n \geq 0}$ is increasing and $O$ is upper closed, $x_n \in O$ $ \text{for all } n \geq n_0$. Since $f$ is monotone, we have $f(x_n) \preceq f(\sqcup (x_n)_{n \geq 0})$ $\text{for all } n \geq0$ and $(f(x_n))_{n \geq 0}$ is an increasing sequence, which implies $\sqcup (f(x_n))_{n \geq 0}$ exists. Define, for simplicity, $x:=\sqcup (f(x_n))_{n \geq 0}$ and $y:=f(\sqcup (x_n)_{n \geq 0})$. By definition of supremum, we have $x \preceq y$. Consider now $O \in \sigma(P)$ such that $y \in O$. Since $(x_n)_{n \geq 0}$ converges to $\sqcup (x_n)_{n \geq 0}$ in $\sigma(P)$, we have $(f(x_n))_{n \geq 0}$ converges to $y$ in $\sigma(Q)$ by sequential continuity. Thus, there exists some $n_0 \geq0$ such that $f(x_n) \in O$ $\text{for all } n \geq n_0$. In particular, since $f(x_n) \preceq x$ $\text{for all } n \geq0$ and $O$ is upper closed, we have $x \in O$. Hence, as we have shown, $y \in O$ implies $x \in O$ $\text{for all } O \in \sigma(P)$ and, by  \eqref{charac order by topo}, $y \preceq x$.
We obtain, by antisymmetry,
\begin{equation*}
    \sqcup (f(x_n))_{n \geq 0}=x=y=f(\sqcup (x_n)_{n \geq 0}).
\end{equation*}
%, that is, $f(\sqcup x_n) =\sqcup f(x_n)$ and
In summary, $f$ preserves suprema of increasing sequences.

We show now $(2)$ implies $(3)$. To do so, it is sufficient to show,
%Since $(1)$ implies $(2)$ by definition of continuity in the Scott topology, we only need to show the converse holds, that is, that we have,
under the hypotheses in $(2)$, that $f(\sqcup A)=\sqcup f(A)$ for any directed set $A\subseteq P$. Assume first $\sqcup A \in A$. Since $f$ is monotone, we have $f(a) \preceq f(\sqcup A)$ $\text{for all } a \in A$. If there exists some $z \in Q$ such that $f(a) \preceq z$ $\text{for all } a \in A$, then $f(\sqcup A) \preceq z$. Thus, $\sqcup f(A) = f(\sqcup A)$. Assume now $\sqcup A \notin A$ and take, as in Lemma \ref{countable covers}, an increasing sequence $(d'_n)_{n\geq0}$, such that $\sqcup (d'_n)_{n\geq0} = \sqcup A$, where $\sqcup A \not \in (d'_n)_{n\geq0}$. Notice, by monotonicity of $f$, $(f(d'_n))_{n \geq 0}$ is an increasing sequence, which means $\sqcup (f(d'_n))_{n \geq 0}$ exists and, since $f$ preserves suprema of increasing sequences by hypothesis, we have
\begin{equation*}
    f(\sqcup A) = f(\sqcup (d'_n)_{n\geq0}) = \sqcup  (f(d'_n))_{ n \geq 0}.
\end{equation*}
Thus, if we show $\sqcup (f(d'_n))_{n \geq 0} = \sqcup f(A)$, then the proof is finished. Define, for simplicity, $z_0 \coloneqq \sqcup (f(d'_n))_{n \geq 0}$ and $z_1 \coloneqq \sqcup f(A)$. Since $\text{for all } a \in A$ there exists some $n \geq0$ such that $a \preceq d'_{n_0}$ and $f$ is monotone, we have $f(a) \preceq f(d'_n) \preceq z_0$. Thus, by transitivity, $f(a) \preceq z_0$ $\text{for all } a \in A$ and, by definition of supremum, $z_1 \preceq z_0$. 
Conversely, since $\text{for all } n \geq0$ there exists some $a \in A$ such that $d'_n\preceq a$ and $f$ is monotone, we have $f(d'_n)\preceq f(a) \preceq z_1$ $\text{for all } n \geq0$. Thus, by transitivity, $f(d'_n) \preceq z_1$ $\text{for all } n \geq0$ and, by definition of supremum, $z_0 \preceq z_1$.
By antisymmetry, we conclude $z_0=z_1$. Hence,
\begin{equation*}
    f(\sqcup A) = f (\sqcup d'_n) = \sqcup (f(d'_n))_{ n \geq 0} = \sqcup f(A),
\end{equation*}
that is, $f$ preserves suprema of directed sets.

Lastly, the fact that $(3)$ implies $(1)$ is a well-known topological fact (see \cite{kelley2017general}) for which Debreu separability of $P$ is not needed.
\end{proof}

\begin{rem}[Computability interpretation]
Theorem \ref{thm 3} points towards the fact that, provided $P$ is Debreu separable, a function between dcpos $f:P \to Q$ is computable (in the sense that of Definition \ref{def: comp func}) if and only if it sends computable elements (in particular, sequences of elements converging towards another one) to computable elements. (Note this is not exactly so, since computable function also require an \emph{effectivity} property by \eqref{compu funct def}.)
\end{rem}

Note that Theorem \ref{thm 3} applies to the Cantor domain $\Sigma^\infty$. More importantly, it is interesting to consider the relation between Theorem \ref{thm 3} and the topological countability axioms, like first and second countability. In particular,
the equivalence in Theorem \ref{thm 3} between $(1)$ and $(3)$ also holds under the stronger hypotheses in the first statement of Proposition \ref{count Debreu separable}, since any function $f:X \to Y$ between topological spaces $X$ and $Y$ is sequentially continuous if and only if it is continuous whenever $X$ is a first countable topological spaces \cite{munkres2019topology}.
%the hypothesis in Theorem \ref{thm 3} is weaker.
However, there exist second countable spaces (thus first countable) which are not Debreu separable (see Proposition \ref{strong prop 4}).
Moreover, there are Debreu separable dcpos which are not first countable (see \cite[Proposition 5]{hack2022computation}) and also some which are not topologically separable (like $P \coloneqq (\mathbb{R},=)$, where, if $B \subseteq P$ is a topologically dense subset of $P$, then $B =\mathbb{R}$, since $\{x\} \in \sigma(P)$ for all $x \in P$).
Lastly, Note that the equivalence in Theorem \ref{thm 3} holds for any computable function between dcpos since, whenever they are defined, countable bases exist and, by Lemma \ref{second count = w-cont}, $\sigma(P)$ is second countable.
%thus first countable.

%\textbf{Erase?} In order to study computable maps between dcpos $f:P \to Q$ we can, thus, restrict ourselves to sequentially continuous monotone maps provided $P$ is Debreu separable.

%Notice, under the hypotheses of the first statement in Proposition \ref{count Debreu separable}, we recover the equivalence in Theorem \ref{thm 3} as it is well-known for first countable topological spaces sequential continuity and continuity coincide \cite{kelley2017general}.

To clarify the relation between Debreu separable dcpos and the topological countability axioms, we need couple more definitions. In order to achieve them, we use the following notation: Given a sequence $(x_n)_{n \geq 0} \subseteq X$ and a topological space $(X,\tau)$, we denote by $x_n \to x$ the fact that $(x_n)_{n \geq 0}$ converges to $x \in X$ according to $\tau$.
\begin{defi}[Sequentially open sets and sequential spaces]
    If $(X,\tau)$ is a topological space, $O \subseteq X$ is a \emph{sequentially open set} if, for every sequence $(x_n)_{n \geq 0} \subseteq X$ such that $x_n \to x \in O$, there exists some $n_0$ such that $x_n \in O$ for all $n \geq n_0$. Moreover, we say $(X,\tau)$ is a \emph{sequential space} if every sequentially open set $O \subseteq X$ is an open set $O \in \tau$. 
\end{defi}
Intuitively, sequential spaces are those for which convergence is completely determined by sequences (as oposed to the more general scenario, where they are characterized by nets \cite{kelley2017general}.
The equivalence between $(1)$ and $(3)$ in Theorem \ref{thm 3} points towards the fact that the Scott topology of Debreu separable dcpos is completely determined by sequences, since continuity is. This is indeed the case, as we show directly, that is, without using $(2)$, in Proposition \ref{charac Scott open}.
%More precisely, we show the open sets in $\sigma(P)$ can be characterized as those subsets of $P$ which are sequentially open for increasing sequences in $\sigma(P)$ or, equivalently, those which are upper closed and inaccessible by increasing sequences, where we say $O \subseteq P$ is \emph{inaccessible by increasing sequences} if, for any increasing sequence $(x_n)_{n\geq0}$ such that $\sqcup (x_n)_{n\geq0} \in O$, there exists some $n_0 \geq0$ such that $x_{n_0}\in O$. 

\begin{prop}
\label{charac Scott open}
If $P$ is a Debreu separable dcpo, then $P$ is a sequential space with respect to its Scott topology $\sigma(P)$.
\end{prop}

\begin{proof}
To prove the result, we only need to show that, if a set $O\subseteq P$ is sequentially open, then it is upper closed and, for all directed sets $A \subseteq P$ such that $\sqcup A \in O$, we have $A \cap O \neq \emptyset$. Take, thus, some $x \in O$ and some $y \in P$ such that $x \preceq y$. We Note that $(x_n)_{n\geq0}$, where $x_n \coloneqq y$ $\text{for all } n\geq0$, converges to $x$ in $\sigma(P)$. Thus, by hypothesis, there exists some $n_0$ such that $x_n = y \in O$ for all $n \geq n_0$. Hence, $O$ is upper closed. Take now a directed set $A \subseteq P$ such that $\sqcup A \in O$. If $\sqcup A \in A$, then $A \cap O \neq \emptyset$ and we have finished. If $\sqcup A \notin A$ and $D \subseteq P$ is a countable Debreu dense subset of $P$, then, by  Lemma \ref{countable covers}, we can construct a sequence $(d'_n)_{n\geq0} \subseteq D$, defined by \eqref{def D_A} and \eqref{def d'_n}, such that $\sqcup (d'_n)_{n\geq0} = \sqcup A$ and $\sqcup A \not \in (d'_n)_{n\geq0}$. Since $(d'_n)_{n\geq0}$ fulfills $d'_n \to \sqcup d'_n = \sqcup A \in O$ by construction, then, by hypothesis, there exists some $n_0 \geq 0$ such that $d'_n \in O$ $\text{for all } n \geq n_0$.
%We can take, thus, some $n_0 \geq0$ such that $d'_{n_0} \in O$ and
Hence, by definition of $(d'_n)_{n\geq0}$, there exists some $a \in A$ such that $d'_{n_0} \preceq a$ and,
%by construction and,
since $O$ is upper closed as we showed before, we conclude $a \in A \cap O \neq \emptyset$. We conclude
%As $O$ is upper closed, we get
$O \in \sigma(P)$.
\end{proof}

To summarize, this section shows that Debreu separability allows us to characterize both the completeness (Proposition \ref{equi D-sep partial orders}) and the Scott topology (Theorem \ref{thm 2} and Proposition \ref{charac Scott open}) of a partial order using only sequences.

\section{Functional countability restrictions}
\label{func restrictions}

We conclude, in this section, relating the usual countability restriction in the order-theoretical approaches to computability with some usual functional restrictions in order theory, namely, countable multi-utilities.

A family $V$ of real-valued functions $v: X \rightarrow \mathbb{R}$ is called a \emph{multi-utility (representation) of $\preceq$} \cite{evren2011multi} if 
\begin{equation*}
x \preceq y \iff v(x) \leq v(y)  \text{ }\text{for all } v \in V \, . 
\end{equation*}
Whenever a multi-utility consists of strict monotones it is called a \emph{strict monotone} (or \emph{Richter-Peleg} \cite{alcantud2016richter}) \emph{multi-utility (representation) of $\preceq$}. In this section, we relate multi-utilities and strict monotone multi-utilities to both bases and weak bases.

We begin, in Proposition \ref{LSC m-u existence}, addressing the existence of multi-utilities for dcpos and their relation to both bases and weak bases. Note that we say a real-valued function $f:X \to \mathbb{R}$ is lower semicontinuous if $f^{-1}((a,\infty))$ is an open set of $\tau_X$, the topology of $X$.

\begin{comment}
\begin{lem}
\label{lower inside Scott}
$\tau_{\preceq}^l \subseteq \sigma(P)$ for any dcpo $P$.
\end{lem}

\begin{proof}
We will show $d(x) \coloneqq \{y \in P| y \preceq x\}$ is closed in $\sigma(P)$ $\text{for all } x \in P$. Since $d(x)$ is a lower set $\text{for all } x \in P$, we only need to show given any directed set $A \subseteq d(x)$ we have $\sqcup A \in d(x)$. Take $A \subseteq d(x)$ a directed set. Note $\sqcup A$ exists and $a \preceq x$ $\text{for all } a \in A$ which means, by definition of least upper bound, $\sqcup A \preceq x$ and $\sqcup A \in d(x)$. Thus, $d(x)$ is closed in $\sigma(P)$ $\text{for all } x \in P$.
\end{proof}
\end{comment}

\begin{prop}
\label{LSC m-u existence}
If $P$ is a dcpo and, for all $x \in P$,
\begin{equation}
\label{def: lower sets}
    d(x) \coloneqq \{y \in P| y \preceq x\},
\end{equation}
then $(u_x)_{x \in P}$ is a lower semicontinuous multi-utility, where
\begin{equation*}
%\label{simple RW}
    u_x(y) = \begin{cases}
    1 &\text {if } y \in d(x)^c,\\
   0 &\text{otherwise},
    \end{cases}
\end{equation*}
 for all $x,y \in P$. Moreover, if $P$ has a (basis) weak basis $B \subseteq P$, then $(v_b)_{b\in B}$ is a (lower semicontinuous) multi-utility, where
 \begin{equation*}
%\label{simple RW}
    v_b(x) = \begin{cases}
    1 &\text {if } b \preceq x \text{ } (b \ll x),\\
   0 &\text{otherwise},
    \end{cases}
\end{equation*}
for all $x \in P$, $b \in B$.
However, the converse of the second statement is not true. %representation with the cardinality of $B$, although the converse is not true.
\end{prop}

\begin{proof}
For the first statement, it is easy to see that $(u_x)_{x \in P}$ is a multi-utility for $\sigma(P)$.
%, where $\text{for all } x,y \in P$ $u_x(y)\coloneqq 1$ if $y \in d(x)^c$ and $u_x(y)\coloneqq  0$ otherwise.
To conclude, we show
$d(x)$
%$d(x) \coloneqq \{y \in P| y \preceq x\}$
is closed in $\sigma(P)$ for all $x \in P$, which implies $u_x$ is lower semicontinuous for all $x \in P$. Since $d(x)$ is a lower set for all $ x \in P$, we only need to show, given any directed set $A \subseteq d(x)$, we have $\sqcup A \in d(x)$. Take $A \subseteq d(x)$ a directed set. Note that $\sqcup A$ exists and $a \preceq x$ for all $ a \in A$ which means, by definition of least upper bound, $\sqcup A \preceq x$ and $\sqcup A \in d(x)$. Thus, $d(x)$ is closed in $\sigma(P)$ for all $x \in P$.
 
 For the second statement, note that $\{v_b\}_{b \in B}$ is a multi-utility by Proposition \ref{weak basis implications}, where $B \subseteq P$ is a weak basis.
 %, and $u_b(x) \coloneqq 1$ if $b \preceq x$ and $u_b(x) \coloneqq 0$ otherwise.
 %, is a countable multi-utility by Proposition ... .
 If $B$ is a basis, then $\{v_b\}_{b \in B}$ is a multi-utility again by Proposition \ref{weak basis implications} (in this case, we take $v_b(x) = 1$ if $b \ll x$ $\text{for all } x \in P, b \in B$ and $v_b(x) = 0$ otherwise),
 and $v_b$ is lower semicontinuous $\text{for all } b \in B$ since $B$ is a basis, which implies $\twoheaduparrow{b} \in \sigma(P)$ \cite{abramsky1994domain}.
 %and $v_b$ is lower semicontinuous $\text{for all } b \in B$.We only need to show $\{v_b\}_{b \in B}$ is a multi-utility. This follows from the fact\begin{equation}\label{basis chara order}x \preceq y \iff \text{for all } b \in B \text{ } x \in \twoheaduparrow b \text{ implies } y \in \twoheaduparrow b. \end{equation}
 %Following characterization \eqref{charac order by topo}, we only need to show the right side of \eqref{basis chara order} implies for any $O \in \sigma(P)$ we have $x \in O$ implies $y \in O$. Given such an $O$, $x \in O$ implies there exists some $b \in B$ such that $x \in \twoheaduparrow b \subseteq O$. Then, $y \in \twoheaduparrow b \subseteq O$ by hypothesis and we have finished.
  
 %If $x \preceq y$ then, since $b \ll x \prec y$ implies $b \ll y$ we get $u_b$ is monotone} $\text{for all } b \in B$. If $\neg(x \preceq y)$ then consider a directed set $D \subseteq \twoheaddownarrow{x} \cap B$ such that $\sqcup D = x$ which exists by definition of Basis. If $d \preceq y$ $\text{for all } y \in D$ then $x=\sqcup D \preceq y$ leading to contradiction. Thus, there exists $d \in D$ such that $\neg (d \preceq y)$ implying $\neg (d \ll y)$ which means $u_d(y) < u_d(x)$ with $d \in B$.
 
 For the third statement, take $P$ the dcpo from the proof of Proposition \ref{cont Deb upper sep but not w-cont} as a counterexample. As shown there, $P$ has no countable weak basis. However, $(w_x)_{x \in \Sigma^*}$ is a countable multi-utility, where 
 %On the contrary, considering $y_n$ an enumeration of $\Sigma^*$, $P$ has a (lower semicontinuous) countable multi-utility$(\chi_n)$ where
 $w_x(y) \coloneqq 1$ if $x \preceq y$ and $w_x(y) \coloneqq 0$ otherwise. Note that $w_x$ is lower semicontinuous $\text{for all } x \in \Sigma^*$ since $\Sigma^* \subseteq \text{K}(P)$ as we showed in the proof of Proposition \ref{strong prop 4}.
 \end{proof}
 
  Note that the second statement in Proposition \ref{LSC m-u existence} also holds whenever $B \subseteq P$ is a Debreu dense subset \cite{hack2022representing}. However, as we showed in Propositions \ref{weak basis not debreu sep} and \ref{cont Deb upper sep but not w-cont}, there exist dcpos where either of these hypothesis hold for some countable $B$ and the other does not.
  
  As shown in Proposition \ref{LSC m-u existence}, there exist dcpos where lower semicontinuous multi-utilities exist and the Scott topology is not second countable (note the counterexample in Proposition \ref{LSC m-u existence} is a continuous dcpo and, hence, the existence of a countable basis and second countability of the Scott topology are equivalent \cite{abramsky1994domain}). As we show in the following proposition, the equivalence holds when considering a coarser topology, namely, the lower topology. (Given a partial order $P=(X,\preceq)$, the closed sets of the \emph{lower topology} 
  $\tau_{\preceq}^l$ are the intersections of finite unions of elements in the family $(d(x))_{x \in X}$, which we defined in \eqref{def: lower sets}.)

  Before proving that proposition, we recall some topological concepts that we will use for both its proof and that of the following result, Theorem \ref{thm 4}.
\begin{defi}[Subbasis, net, and net convergence {\cite{munkres2019topology}}]
\label{def: topo}
If $(X,\tau)$ is a topological space, then a \emph{subbasis} is a family of subsets $(O_i)_{i \in I} \subseteq X$ whose union is $X$ and such that any element in the topology $\tau$ can be generated as the union of finite intersections of elements in the subbasis. Moreover, if $I$ is a directed set (in the sense of some preorder $\preceq_I$), a \emph{net} $\{x_\alpha\}_{\alpha \in I} \subseteq X$ is a function $g:I \to X$, $\alpha \mapsto x_\alpha$. Lastly, we say a net $\{x_\alpha\}_{\alpha \in I} \subseteq X$ \emph{converges} to $x \in X$ provided there exists, for each $O \in \tau$, some $\alpha_O \in I$ such that $x_\alpha \in O$ whenever $\alpha_0 \preceq_I \alpha$.
\end{defi}
Note that subbases contrast with the more common topological notion of bases, where the elements in $\tau$ are generated using only \emph{unions} of elements in a basis. Regarding net convergence, it is clear that, if we take a dcpo $P$, we fix $I=D$, we take the identity as $g$, and we use $(P,\sigma(P))$ as our topological space, then any directed set $D \subseteq P$ is a net which converges to $\sqcup D$.
  
  \begin{prop}
\label{lower topo}
If $(P, \tau_{\preceq}^l)$ is the topological space consisting of a dcpo $P$ equipped with the lower topology $\tau_{\preceq}^l$, then there exists a lower semicontinuous countable multi-utility if and only if $\tau_{\preceq}^l$ is second countable.
\end{prop}

\begin{proof}
 Assume first there exists a countable lower semicontinuous multi-utility $(u_n)_{n \geq 0}$. Note, by definition of the lower topology, $(A_x)_{x\in P}$ is a subbasis for $\tau_{\preceq}^l$, where
 \begin{equation*}
     A_x \coloneqq d(x)^c= \{y \in P| \neg(y \preceq x)\}.
 \end{equation*}
 %$A_x \coloneqq d(x)^c= \{y \in P| \neg(y \preceq x)\}$.
 Take $y \in A_x$ for some $x \in P$ and note there exist a pair $n \geq 0$ and $q \in \mathbb{Q}$ such that $y \in O_{n,q} \subseteq A_x$, where $O_{n,q} \coloneqq u_n^{-1}((q,\infty))$ for all $n \geq 0$ and $q \in \mathbb{Q}$. This is the case since we have, by definition, $\neg(y \preceq x)$. Hence, there exists some $n \geq 0$ such that $u_n(x) < u_n(y)$ and, as a result, some $q \in \mathbb{Q}$ such that $y \in O_{n,q} \subseteq A_x$. Note $(O_{n,q})_{n\geq0, q \in \mathbb{Q}} \subseteq \tau^l_{\preceq}$ given that $u_n$ is lower semicontinuous for all $n \geq 0$.
 %where $O_{n,q} \coloneqq v_n^{-1}((q,\infty)$ $\text{for all } (n,q) \in \mathbb{N} \times \mathbb{Q}$. Since $\neg(x \preceq z)$ holds, there exists $n \in \mathbb{N}$ such that $u_n(z) < u_n(x)$ and there exists some $q \in \mathbb{Q} \cap \big(v_n(z),v_n(x)\big)$. We get $x \in O_{n,q} \subseteq A_z$.
 As a result, $(O_{n,q})_{n\geq 0,q \in \mathbb{Q}}$ is a countable subbasis of $\tau_{\preceq}^l$ and, hence, $\tau_{\preceq}^l$ is second countable. For the converse, take $(B_n)_{n \geq 0}$ a countable basis for $\tau_{\preceq}^l$ and note $(u_n)_{n \geq 0}$ is a lower semicontinuous countable multi-utility, where for all $n \geq 0$ we have $u_n(x) \coloneqq 1$ if $x \in B_n$ and $u_n(x) \coloneqq 0$ otherwise. To see this holds, take $x,y \in P$. If $x \preceq y$ and $y \not \in B_n$ for some $n \geq 0$, then $y \in O$, where $O \subseteq B_n^C$ is some finite union of intersections of sets in the family
%there exists some $z \in P$ such that $y \preceq z$, since
$(d(x))_{x \in P}$. Note such an $O$ exists by definition of $\tau_{\preceq}^l$. By transitivity of $\preceq$, $x \in O$. Thus, $x \not \in B_n$ and $u_n$ is monotone for all $ n \geq 0$. If $\neg(x \preceq y)$, then $x \in A_y \in \tau_{\preceq}^l$ and there exists some $n \geq 0$ such that $x \in B_n \subseteq A_y$. Hence, we have $u_n(x) > u_n(y)$ and $(u_n)_{n\geq0}$ is a multi-utility. Note $u_n$ is lower semicontinuous since $B_n$ is open for all $n \geq 0$.
\end{proof}

In particular, note, if $\tau_{\preceq}^l$ is second countable, then there exists a lower semicontinuous countable multi-utility for $\sigma(P)$.

In the following theorem, we show the main results of this section. In particular, we note we can partially reproduce the implications in Theorem \ref{thm 2} by requiring the existence of a finite strict monotone multi-utility instead of conditional connectedness.
  
\begin{teo}
\label{thm 4}
If $P$ is a dcpo with a finite lower semicontinuous strict monotone multi-utility, then the following hold:
%\vspace{-10pt}
%\begin{enumerate}[label=(\textit{\roman*})]
\begin{enumerate}[label={(\arabic*)}]
    \item $P$ is $\omega$-continuous if and only if $\text{K}(P)$ is countable.
    \item If $P$ has a (countable) basis $B \subseteq P$, then $B$ is a (countable) Debreu dense and Debreu upper dense subset.
    %it is Debreu upper seprable.
\end{enumerate}
\end{teo}

\begin{proof}
$(1)$ If $P$ is $\omega$-continuous, then, clearly, it is continuous and $\text{K}(P)$ is countable. (The latter follows since $\text{K}(P) \subseteq B$ for any basis $B \subseteq P$ \cite{abramsky1994domain}.) To show the converse holds, we begin proving any dcpo $P$ is continuous whenever a finite lower semicontinuous strict monotone multi-utility exists. In order to do so, we establish first that, under the same hypothesis, $x \prec y$ implies $x \ll y$ for all $ x,y \in P$.

\begin{lem}
\label{preceq to ll}
If $P$ is a dcpo with a finite lower semicontinuous strict monotone multi-utility, then $x \prec y$ implies $x \ll y$ for all $ x,y \in P$.
\end{lem}

\begin{proof}
Fix $(v_n)_{n \leq N}$ a finite lower semicontinuous strict monotone multi-utility and take $ x,y \in P$ such that $x \prec y$ and $D \subseteq P$ a directed set such that $y \preceq \sqcup D$. We intend to show there exists some $d \in D$ such that $x \preceq d$. Note $x \prec \sqcup D$ by transitivity and $v_n(x) < v_n(\sqcup D)$ for all $ n \leq N$ by definition of $(v_n)_{n=1}^N$. If $\sqcup D \in D$, then we can take $d=\sqcup D$
%$a \preceq \sqcup D \in D$
and we have finished. If $\sqcup D \not \in D$, then there exists some $d_1 \in D$ such that $v_1(x) < v_1(d_1) < v_1(\sqcup D)$, since $D$ converges as a net to $\sqcup D$ (see the comment after Definition \ref{def: topo}) and $v_1$ is lower semicontinuous. For the same reason, there exists a set $\{d_1,..,d_N\} \subseteq D$ such that $v_n(x) < v_n(d_n) < v_n(\sqcup D)$ for all $n \leq N$. Given the fact $D$ is directed, there exists some $c_1\in D$ such that $d_1 \preceq c_1$ and $d_2 \preceq c_1$. We define $c_n$ recursively in the same way using $d_n$ and $c_{n-1}$ for all $n$ such that $1< n \leq N$. Note that $d \coloneqq c_{N-1}$ has all the desired properties. In particular, by definition, $d \in D$ and $v_n(x) < v_n(d)$ for all $n \leq N$. By definition of $(v_n)_{n=1}^N$, we have $x \prec d \in D$. Thus, $x \ll y$.

Alternatively, we can also show that $x \prec y$ implies $x \ll y$ for all $x,y \in P$ as follows: Given that $x \prec y$, we have
\begin{equation*}
    y \in U \coloneqq \bigcap_{n \leq N} v_n^{-1}\big((\alpha_n, \infty)\big) \in \sigma(P),
\end{equation*}
where $U$ is open since it is a finite intersection of open sets (which are open by lower semicontinuity of $(v_n)_{n=1}^N$) and $(\alpha_1,\dots,\alpha_N) \in \mathbb R^N$ such that $v_n(x)<\alpha_n<v_n(y)$ for $n=1,\dots,N$. Hence, given a directed set $D$ such that $y \preceq \sqcup D$, we have $\sqcup D \in U \in \sigma(P)$ and, thus, there exists some $d \in D$ such that $d \in U$. Since $v_n(x)<v_n(d)$ for all $n \leq N$, we have $x \prec d$ and we have finished the alternative argument supporting that $x \ll y$ for all $x,y \in P$ such that $x \prec y$.
\end{proof}

We proceed now to show $P$ is continuous whenever a finite lower semicontinuous strict monotone multi-utility exists.

\begin{lem}
\label{strict implies cont}
If $P$ is a dcpo with a finite lower semicontinuous strict monotone multi-utility, then $P$ is continuous.
\end{lem}

\begin{proof}
Take some $x \in P$. We ought to show there exists some directed set $D_x \subseteq P \cap \twoheaddownarrow x$ such that $\sqcup D_x=x$.  If there exists some directed set $D \subseteq P\setminus \{x\}$ such that $\sqcup D =x$, then we have finished, since we have $d \ll x$ for all $d\in D$ given that $d \prec x$ for all $d \in D$. Assume now there is no directed set $D \subseteq P\setminus \{x\}$ such that $\sqcup D=x$. We will show, in this case, $x \in \text{K}(P)$, which concludes the proof that $P$ is continuous by taking $D_x \coloneqq \{x\}$. Consider, hence, a directed set $D \subseteq P$ such that $x \preceq \sqcup D$. If $\sqcup D=x$, then $x \in D$ and we have finished. If $x \prec \sqcup D$, then $x \ll \sqcup D$ by Lemma \ref{preceq to ll} and there exists some $d \in D$ such that $x \preceq d$. Thus, $x \in \text{K}(P)$ in this scenario and $P$ is continuous.
\end{proof}

To conclude, we ought to show $P$ is $\omega$-continuous whenever there exists a finite lower semicontinuous strict monotone multi-utility and $\text{K}(P)$ is countable. To establish this, we show that the family consisting of (a) the union of the sets of elements way-above each compact element and (b) the sets of elements strictly above each set of rational values a finite lower semicontinuous strict monotone multi-utility may take,
%, given $(v_n)_{n \leq N}$ a finite lower semicontinuous strict monotone multi-utility, we have
\begin{equation}
\label{basis thm 4}
    B \coloneqq  \Big(\twoheaduparrow x\text{ }\Big)_{x \in \text{K}(P)} \bigcup \Big( \bigcap_{n \leq N} v_n^{-1}\big((q_n, \infty)\big)\Big)_{(q_1,..,q_N) \in \mathbb{Q}^N},
\end{equation}
is a countable basis for $\sigma(P)$ whenever $\text{K}(P)$ is countable. In order to show that \eqref{basis thm 4} is indeed a basis, take $x  \in O \in \sigma(P)$ and notice, since $P$ is continuous as we showed, there exists some $y \in O$ such that $y \ll x$. If $y=x$, then $x \in \text{K}(P)$. Thus, $x \in \twoheaduparrow x \subseteq O$. If $y \neq x$, then $y \prec x$. As a result, there exists some $(q_1,..,q_N) \in \mathbb{Q}^N$ such that $v_n(y) < q_n < v_n(x)$ for all $n \leq N$ and $x \in U$, where
\begin{equation*}
    U \coloneqq \bigcap_{n \leq N} v_n^{-1}\big((q_n, \infty)\big).
\end{equation*}
%$U \coloneqq \bigcap_{n \leq N} v_n^{-1}\big((q_n, \infty)\big)$.
Note $U\subseteq O$ since any $z \in U$ fulfills $v_n(y) <v_n(z)$ for all $ n \leq N$ and, by definition of $(v_n)_{n=1}^N$, we get $y \prec z$, thus $z \in O$ as $O$ is upper closed by definition.

$(2)$ 
%We begin showing if $P$ is a continuous dcpo with a finite lower semicontinuous strict monotone $(v_n)_{n=1}^N$, then $x \prec y$ implies $x \ll y$ for all $ x,y \in P$. Take $D \subseteq P$ a directed set such that $b \preceq \sqcup D$. We intend to show there exists some $d \in D$ such that $a \preceq d$. Note $a \prec \sqcup D$ by transitivity and $v_n(a) < v_n(\sqcup D)$ for all $ n \leq N$ by definition of $(v_n)_{n=1}^N$. If $\sqcup D \in D$, then we can take $d=\sqcup D$ and we have finished. If $\sqcup D \not \in D$, then there exists some $d_1 \in D$ such that $v_1(a) < v_1(d_1) < v_1(\sqcup D)$, since $D$ converges as a net to $\sqcup D$ and $v_1$ is lower semicontinuous. For the same reason, there exists a set $\{d_1,..,d_N\} \subseteq D$ such that $v_n(a) < v_n(d_n) < v_n(\sqcup D)$ for all $n \leq N$. Given the fact $D$ is directed, there exists some $c_1\in D$ such that $d_1 \preceq c_1$ and $d_2 \preceq c_1$. We define recursively $c_n$ in the same way recursively using $d_n$ and $c_{n-1}$ for all $n$ such that $1< n \leq N$. Note that $d \coloneqq c_{N-1}$ has all the desired properties. In particular, by definition, $d \in D$ and $v_n(a) < v_n(d)$ for all $n \leq N$. By definition of $(v_n)_{n=1}^N$, we have $a \prec d \in D$. Thus, $x \ll y$.
We show a (countable) basis $B \subseteq P$ is a Debreu dense and Debreu upper dense subset of $P$. Note $B$ is a Debreu upper dense subset of $P$ by Proposition \ref{weak basis implications}. To conclude, we show $B$ is also Debreu dense. Take, thus, $x,y \in P$ such that $x \prec y$. As we showed in $(1)$, we have $x \ll y$ since there exists a finite lower semicontinuous strict monotone multi-utility. Hence, by the interpolation property \cite[Lemma 2.2.15]{abramsky1994domain}, there exists some $b \in B$ such that $x \ll b \ll y$. In particular, we have $x \preceq b \preceq y$. Hence, $B$ is a (countable) Debreu dense subset of $P$.
\end{proof}

In order to interpret Theorem \ref{thm 4} $(1)$ in terms of computability, we provide an explicit basis (in the dcpo sense) in the following proposition. We do so since the proof of Theorem \ref{thm 4} $(1)$ only shows that $\sigma(P)$ is second countable whenever a dcpo has a finite lower semicontinuous strict monotone multi-utility and $\text{K}(P)$ is countable.
\begin{prop}
\label{effective for strict}
Take $P$ a dcpo with a finite lower semicontinuous strict monotone multi-utility $(v_i)_{i=1}^N$, $T \coloneqq \{(q,r) \in \mathbb Q^N \times \mathbb Q^N | q_i<r_i \text{ for } i=1,\dots,N\}$,  a numeration of $T$, $\gamma : \mathbb N \to T$, whose first (last) $N$ components we denote by $\gamma_1$ ($\gamma_2$),
\begin{equation*}
\begin{split}
m_0 &\coloneqq \text{min} \{n \geq 0| \exists x \in P \text{ s.t. } v_i(\alpha_1(n)) < v_i(x) < v_i(\alpha_2(n))\} \text{ and}\\
m_n &\coloneqq \text{min} \{n \geq m_{n-1}+1| \exists x \in P \text{ s.t. } v_i(\alpha_1(n)) < v_i(x) < v_i(\alpha_2(n))\} \text{ for all } n \geq 1.
\end{split}
\end{equation*}
% and, for $n \geq 1$,
 %\begin{equation*}
  %    m_n \coloneqq \text{min} \{n \geq m_{n-1}+1| \exists x \in P \text{ s.t. } v_i(\alpha_1(n)) < v_i(x) < v_i(\alpha_2(n))\}.
 %\end{equation*}
    If $\text{K}(P)$ is countable, then $\text{K}(P) \cup (t_n)_{n \geq 0}$ is a countable basis, where, for all $n \geq 0$, we take as $t_n$ some $x \in P$ such that $v_i(\gamma_1(m_n))<v_i(x)<v_i(\gamma_2(m_n))$. %(Note such an $x \in P$ exists by definition.) 
\end{prop}
\begin{proof}
    We ought to show that, for each $x \in P$ and $B \coloneqq \text{K}(P) \cup (t_n)_{n \geq 0}$, there exists a directed set $B_x \subseteq \twoheaddownarrow x \cap B$ such that $\sqcup B_x = x$. If $x \in \text{K}(P)$, then we take $B_x = \{x \}$ and we have finished. If $x \not \in \text{K}(P)$, then consider some $(q^0_1,\dots,q^0_N) \in \mathbb Q^N$ such that $q^0_i<v_i(x)<q^0_i+1$ for $i=1,\dots,N$. Since $(v_i)_{i=1}^N$ is lower semicontinuous and $P$ is continuous by Lemma \ref{strict implies cont}, there exists some $y_0 \in \twoheaddownarrow x \cap P$ such that $q^0_i<v_i(y_0)<v_i(x)$ for $i=1,\dots,N$ (the latter inequalities follow since $x \not \in \text{K}(P)$ and, hence, $y_0 \prec x$).
    We can then consider some $(r^0_1,\dots,r^0_N) \in \mathbb Q^N$ such that $v_i(y_0)<r^0_i<v_i(x)$ for $i=1,\dots,N$. By construction, there exists some $x_0 \in (t_n)_{n \geq 0}$ such that $q^0_i<v_i(x_0)<r^0_i$ for $i=1,\dots,N$.
    We can then take, for all $m>0$,  $q^m_i=r^{m-1}_i$ for $i=1,\dots,N$ (we do this provided $r^{m-1}_i<v_i(x)<r^{m-1}_i+2^{-m}$, otherwise we take some $(q^m_1,\dots,q^m_N) \in \mathbb Q^N$ fulfilling these inequalities), find some $y_m \in \twoheaddownarrow x \cap P$ such that $q^m_i<v_i(y_n)<v_i(x)$ for $i=1,\dots,N$ and some $(r^m_1,\dots,r^m_N) \in \mathbb Q^N$ such that $v_i(y_m)<r^m_i<v_i(x)$ for $i=1,\dots,N$. Finally, by construction, there exists, for all $m >0$, some $x_m \in (t_n)_{n \geq 0}$ such that $q^m_i<v_i(x_m)<r^m_i$ for $i=1,\dots,N$. To conclude, we show that $\sqcup (x_m)_{m \geq 0} = x$. By construction, for all $m \geq 0$, $v_i(x_m) < v_i(x_{m+1})$ for $i=1,\dots,m$. Hence, $(x_m)_{m \geq 0}$ is an increasing sequence and $\sqcup (x_m)_{m \geq 0}$ exists. Moreover, by construction, $\sqcup (x_m)_{m \geq 0} \preceq x$. Now, if $\sqcup (x_m)_{m \geq 0} \prec x$, then $v_i(\sqcup (x_m)_{m \geq 0}) <  v_i(x)$ for $i=1,\dots,N$, which contradicts the definition of $(x_m)_{m \geq 0})$. Hence, since $x \prec y$ implies $x \ll y$ by Lemma \ref{preceq to ll}, $B=\text{K}(P) \cup (t_n)_{n \geq 0}$ is a countable basis.
\end{proof}
We can now interpret Theorem \ref{thm 4} $(1)$ in terms of computability.

\begin{rem}[Implication for computability]
By Theorem \ref{thm 4}, we can define computable elements and functions (in the sense of Definitions \ref{def: comp ele II} and \ref{def: comp func}) on a dcpo $P$ with a finite lower semicontinuous strict monotone multi-utility whenever $\text{K}(P)$ is countable and $\text{K}(P) \cup (t_n)_{n \geq 0}$ is effective. (Note that $(t_n)_{n \geq 0}$ was defined in Proposition \ref{effective for strict}.)
\end{rem}

Note we have shown in the proof of Theorem \ref{thm 4} $(1)$ that the equivalence in Proposition \ref{way-below upper comp} also holds when substituting conditional connectedness by the existence of a finite lower semicontinuous strict monotone multi-utility. Note, as we stated in Proposition \ref{w-cont implies Deb upper dense}, there exist $\omega$-continuous dcpos which are not Debreu separable. The inclusion of $\omega$-continuity in both clauses of Theorem \ref{thm 4} is necessary in order for $\text{K}(P)$ to be countable and for $P$ to be Debreu upper separable, as we show in  Proposition \ref{nece thm 2} $(1)$. Moreover, in  Proposition \ref{nece thm 2} $(2)$, we show the converse of Theorem \ref{thm 4} $(2)$ is false. That is, although the equivalence between the clauses $(2)$ and $(3)$ in Theorem \ref{thm 2} and the fact they imply Theorem \ref{thm 2} $(1)$ are achieved requiring the existence of a finite lower semicontinuous strict monotone multi-utility instead of conditional connectedness, Theorem \ref{thm 2} $(1)$ does not imply neither Theorem \ref{thm 2} $(2)$ nor Theorem \ref{thm 2} $(3)$. Lastly, note there are dcpos where $\text{K}(P)$ is countable and $P$ is not even continuous, like the one in Lemma \ref{compact implies contained if supremum}, where $\text{K}(P)=\emptyset$ and, as argued there, there is no $x \in P$ such that $x \ll 0 \in P$.

\begin{prop}
\label{nece thm 2}
There exist dcpos $P$ with finite lower semicontinuous strict monotone multi-utilities and either of the following properties:
%\vspace{-10pt}
%\begin{enumerate}[label=(\textit{\roman*})]
\begin{enumerate}[label={(\arabic*)}]
\item $\text{K}(P)$ is uncountable and $P$ is not Debreu upper separable. 
\item $P$ is Debreu upper separable and $\text{K}(P)$ is uncountable.
\end{enumerate}
\end{prop}

\begin{proof}
$(1)$ Take the dcpo $P$ which consists of the set $[0,1]$ endowed with the trivial ordering and note $V \coloneqq \{i_d,-i_d\}$ is a strict monotone multi-utility, where $i_d$ is the identity function. Both functions in $V$ are lower semicontinuous in $\sigma(P)$ since $\preceq$ is the trivial ordering and, whenever $\sqcup D =x$ for some directed set, we have $D=\{x\}$. Because of that, $\text{K}(P)=[0,1]$ and $P$ is not Debreu upper separable.

$(2)$ Take the dcpo $P \coloneqq ([0,1],\preceq)$, where, for all $x,y \in P$, $x \preceq y$ if and only if $x=y$ or $x \in \mathbb{Q}$, $y \not \in \mathbb{Q}$ and $x<y$. Note $P$ is, essentially, the counterexample in Proposition \ref{cont Deb upper sep but not w-cont}. As we showed there, $\text{K}(P)=P$ is uncountable and $P$ is Debreu upper separable. To conclude, one can see $V \coloneqq \{v_1,v_2\}$ is a strict monotone multi-utility, where $v_1$ is the identity function and $v_2(x) \coloneqq -x$ if $x \not \in \mathbb{Q}$ and $v_2(x) \coloneqq -x-1$ if $x \in \mathbb{Q}$. Note the functions in $V$ are lower semicontinuous by the same reason the functions in $(1)$ are.
\end{proof}

Note Theorem \ref{thm 4} $(2)$ can be used to conclude certain dcpos have no finite strict monotone multi-utility that is lower semicontinuous in their Scott topology. We can apply this, in particular, to the examples in Section \ref{examples} that are not Debreu upper separable (see Lemma \ref{example domains}).

\begin{coro}
\label{no finite r-p}
Both $(\Lambda^n,\preceq_M)$ for $n \geq 3$ and $(I,\sqsubseteq)$ have no finite lower semicontinuous strict monotone multi-utility representation.
\end{coro}

The examples in Corollary \ref{no finite r-p} are also useful to show we cannot improve on Theorem \ref{thm 4} $(2)$ by weakening the hypothesis from finite strict monotone multi-utilities to multi-utilities, as we show in the following proposition.

\begin{prop}
\label{counter finite lsc}
There exist $\omega$-continuous dcpos which, despite having finite lower semicontinuous multi-utilities, are not Debreu upper separable.
\end{prop}

\begin{proof}
We can take  majorization $(\Lambda^n,\preceq_M)$ with any $n \geq 3$ as a counterexample. As we show in Lemma \ref{example domains}, $(\Lambda^n,\preceq_M)$ is $\omega$-continuous and not Debreu separable for all $n\geq 3$. To conclude, we show $(\Lambda^n,\preceq_M)$ has a finite lower semicontinuous multi-utility $\text{for all } n \geq 2$. In particular, we show $(s_i)_{i=1}^{n-1}$ is a finite lower semicontinuous multi-utility of majorization $(\Lambda^n,\preceq_M)$ for all $n \geq 2$, where $s_i(x) \coloneqq \sum_{j=1}^i x_j$.
%which we know is $\omega$-continuous and has no countable Debreu dense subset. To conclude the proof, we notice there exists a finite lower semicontinuous multi-utility.

For simplicity of notation, we define $P \coloneqq (\Lambda^n,\preceq_M)$. We show $s_k$ is lower semicontinuous $\text{for all } k < n$, that is, that $s_k^{-1}(r, \infty) \in \sigma(P)$ $\text{for all } r \in \mathbb{R}$, $k < n$. Take $k<n$ and some $r \in \mathbb{R}$ such that $k/n \leq r <1$ and note the other cases are straightforward. (If $r<k/n$, then $s_k^{-1}(r, \infty) = \Lambda^n$ since $s_k(\perp)=k/n$ for $k=1,\dots,n-1$, and $\Lambda^n \in \sigma(P)$ by definition of topology. Moreover, if $1 \leq r$, then $s_k^{-1}(r, \infty) = \emptyset$ since $s_k((1,0,\dots,0))=1$ for $k=1,\dots,n-1$, and $\emptyset \in \sigma(P)$ by definition of topology.) Notice, given $p \in s_k^{-1}(r ,\infty)$, there exists some $q \in s_k^{-1}(r, \infty)$ such that $p \in \twoheaduparrow q \subseteq s_k^{-1}(r, \infty)$. To see this, we can take some $\varepsilon < s_k(p)-r$ and apply Lemma \ref{arbitrary q}, obtaining some $q \in \mathbb{Q}^n \cap \Lambda^n$ such that $s_k(p)-\varepsilon < s_k(q) < s_k(p)$ $\text{for all } k<n$. We have, in particular, $q \ll p$ by \eqref{way-below majo}.  This concludes the proof, since we have $s_k^{-1}(r, \infty) \in \sigma(P)$ given that
%as $P$ is continuous (cite) which means
$\twoheaduparrow q \in \sigma(P)$ $\text{for all } q \in P$ \cite[Corollary 2.2.16]{abramsky1994domain}.
\end{proof}

Note Proposition \ref{counter finite lsc} also shows we cannot improve on the result by asking for the existence of countably infinite lower semicontinuous strict monotone multi-utilities, since they also exist for the counterexample in Proposition \ref{counter finite lsc}. In fact, they exist whenever lower semicontinuous countable multi-utilities do \cite{alcantud2016richter}.

To summarize, the main results in this section are Proposition \ref{LSC m-u existence} and Theorem \ref{thm 4}. In the first one, we show that any dcpo has a lower semicontinuous multi-utility and, moreover, that we can pick one with the cardinality of any basis the dcpo may have. In the latter, we show that, whenever finite strict monotone multi-utilities exist, the existence of countable bases is equivalent to the countability of $\text{K}(P)$ and, moreover, that any basis is both Debreu dense and Debreu upper dense.

\section{Conclusion}

In this paper, we have illustrated the role of countability restrictions in the attempt of translating computability from Turing machines to uncountable spaces using ordered structures. We have connected the countability restrictions in a general order-theoretic approach to computability that was recently introduced \cite{hack2022computation} and the ones in domain theory \cite{scott1970outline,abramsky1994domain} to the usual ones in order theory, namely, order density properties and multi-utilities. In particular, we have
%introduced a new order-theoretical approach which relies on directed complete partial orders with countable weak bases and allows the introduction of computable elements. Following this approach,we have
established several connections between order density properties, such as Debreu separability, order density or Debreu upper separability, and the existence of countable weak bases in the more general approach. We have also explored the influence of order density properties in
%related the new setup to 
domain theory, establishing their equivalence with countable bases for the class of dcpos that are 
%which is based on countable bases and allows the introduction of both computable elements and functions. We have highlighted the relevance of order density properties in this approach, introducing
conditionally connected,
%directed complete partial orders,
which includes the prominent example of the Cantor domain.
%, and proving the equivalence between Debreu upper separability and the existence of countable bases for them (Theorem \ref{thm 2}).
%for directed complete partial orders with this property.
After connecting order density with
%other restrictions in these approaches to computability, namely,
both order completeness and continuity in the Scott topology, we finished relating bases, weak bases and order density to multi-utilities.
Regarding computability, we obtained several results which show, for a given dcpo with either some functional (multi-utility) or density countability restriction, how computability can be defined starting from these constraints.
Several questions remain open. For example, it would be relevant to further clarify the role of multi-utilities in computability, since they play a leading role in the study of partial orders.

\newpage
\begin{appendix}
\section{Appendix}
\label{appendix}

\subsection{Proofs}
\label{appen proofs}

\begin{prop}
\label{exists chain inside}
If $P$ is a Debreu separable dcpo and $A \subseteq P$ is a directed set, then there exists an increasing chain $(a_n)_{n\geq0} \subseteq A$ with the same supremum as $A$, $\sqcup (a_n)_{n\geq0}=\sqcup A$.
\end{prop}

\begin{proof}
%Consider a directed set $A$ and
Take $D\subseteq P$
a Debreu dense subset of $P$. If $\sqcup A \in A$, then we take $(a_n)_{n\geq0}$ with $a_n \coloneqq \sqcup A$ $\text{for all } n \geq 0$ and we have finished. Otherwise, consider the increasing sequence $(d_n')_{n\geq0}$ $ \subseteq D$ such that $\sqcup (d_n')_{n\geq0} =\sqcup A$ from the proof of Lemma \ref{countable covers}. By construction, there exists some $b_n \in A$ such that $d'_n \preceq b_n$ $\text{for all } n \geq 0$. Notice $(b_n)_{n\geq0}$ is a directed set, since given $n,m \geq 0$ there exist some $c \in A$ such that $b_n,b_m \preceq c$ and, by construction, some $p \geq0$ such that $b_n,b_m \preceq c \preceq d'_p \preceq b_p$. Thus, we construct an increasing chain $(a_n)_{n\geq0} \subseteq A$ from $(b_n)_{n\geq0}$, like we constructed $D'_A$ from $D_A$ in the proof of Lemma \ref{countable covers}. Notice $\sqcup (a_n)_{n\geq0}$ exists, since $P$ is directed complete. We only need to show $\sqcup (a_n)_{n\geq0}=\sqcup A$. Be definition, $a_n \preceq \sqcup A$ $\text{for all } n \geq0$. Assume there exists some $z \in P$ such that $a_n \preceq z$ $\text{for all } n$. Then, $d_n'  \preceq z$ $\text{for all } n \geq 0$ and, thus, $\sqcup A=\sqcup (d_n')_{n\geq0} \preceq z$. Thus, $\sqcup (b_n')=\sqcup A$.
\end{proof}

%\sum_{í=1}^k b_i < \sum_{í=1}^k x_i + \varepsilon_i < \sum_{í=1}^k x_i + \sum_{í=1}^k y_i - \sum_{í=1}^k x_i = \sum_{í=1}^k y_i
\begin{comment}
We will show for each $x \in P$ we have $x = \sqcup \{q \in | s_k(q) < s_k(x) \text{ }\text{for all } k<n \}$ while for the case $x=\botom$ we take the set with the bottom and done. Notice we have by definition $b \preceq x$ $\text{for all } b$ so we only need to show if there is some $y$ such that $b \preceq y$ $\text{for all } b$ then $x \preceq y$. To do so, notice it is sufficient to show for each $\varepsilon>0$ there exists some $b$ such that $s_k(x)-\varepsilon < s_k(b)$ $\text{for all } k<n$ which leads to $s_k(x) \leq s_k(y)$ $\text{for all } k$ and $x \preceq y$.
Fix some $\varepsilon >0$ and consider $k$ the first index such that after it the ordered components of $y$ are either equivalent to $y_k$ or zero and $h$ the number of components in $y$ equal to $y_k$. Consider now 
\begin{equation*}
0<\varepsilon' < \text{min} \{\frac{\varepsilon}{k}, (1+\frac{k-1}{h})^{-1} (y_{k-1}-y_k)\}
\end{equation*}

\begin{equation*}
0<\beta < \text{min} \{\frac{\varepsilon}{n-h}, (1+\frac{n-h}{h-(k-1)})^{-1} (q_k),\frac{s_k(q)-(s_k(x)-\varepsilon)}{n-h}\}
\end{equation*}
\end{comment}

\begin{lem}
\label{upper comp charact}
If $P$ is a partial order, then $P$ is conditionally connected if and only if any directed set is a chain.
\end{lem}

\begin{proof}
 Consider a directed set $A \subseteq P$ and $x,y \in A$. Since $A$ is directed, there exists $z \in A$ such that $x \preceq z$ and $y \preceq z$. We have, by conditional connectedness, $\neg(x \bowtie y)$ and, thus, $A$ is a chain. Conversely, take $x,y \in P$ such that there exists $z \in P$ where $x \preceq z$ and $y \preceq z$ hold. Take $A \coloneqq \{x,y,z\}$. By construction, $A$ is directed and, by hypothesis, a chain. In particular, $\neg(x \bowtie y)$ and $P$ is conditionally connected.
\end{proof}

\begin{comment}
\begin{prop}
If $P$ is a Debreu separable dcpo, then $\sigma(P)$ is a sequential space. 
\end{prop}

\begin{proof}
We will show that any sequentially open subset $O\subseteq P$ is upper closed. The proof can be concluded following the proof of Proposition \ref{charac Scott open}. Take, thus, some $x \in O$ and some $y \in P$ such that $x \preceq y$. We note that $(x_n)_{n\geq0}$, where $x_n \coloneqq y$ $\text{for all } n\geq0$, converges to $x$ in $\sigma(P)$. Thus, since $O$ is sequentially open, $y \in O$ and $O$ is upper closed. 
\end{proof}
\end{comment}

\subsection{Results for majorization and the interval domain}
\label{results major + int}

\subsubsection{Proof of Proposition \ref{majo eff weak basis}}
\label{proof prop 1}

By Lemma \ref{example domains}, $\mathbb{Q}^n \cap \Lambda^n$ is a countable weak basis for any $n \geq 2$. Hence, it is sufficient to show there exists a finite map $\alpha: \mathbb{N} \to \mathbb{Q}^n \cap \Lambda^n$ such that $\{\langle n,m \rangle| \alpha(n) \preceq \alpha(m)\}$ is recursively enumerable. We begin with a finite map $\alpha_0: \mathbb{N} \to \mathbb{Q} \cap [0,1]$ which, aside from $0$ and $1$, orders the rationals in $[0,1]$ lexicographically, considering first the denominators and then the numerators (see $\alpha$ in \cite[Proposition 1]{hack2022computation}).
\begin{comment}
If $m=0,1$, then $\alpha_0(m)=m$. If $m>1$, then we start with $t=2$ and increase $t$ by one unit until we find one such that $m <2+\sum_{i=2}^t (i-1)$. We take then
\begin{equation*}
\alpha_0(m)=\frac{m-\big(1+\sum_{i=2}^{t-1} (i-1)\big)}{t}.
\end{equation*}
Thus, $\alpha_0$ is a finite map.
\end{comment}
Using $\alpha_0$, we construct now $\alpha$
for the case $n=2$ by just selecting some pairs in $\alpha_0(\mathbb{N}) \times \alpha_0(\mathbb{N})$. If $m=0$, then $\alpha(m)=(\alpha_0(1),\alpha_0(0))$ and we define $p_0=1$ and $q_0=0$. Notice $\alpha_0(1) + \alpha_0(0)=1$. If $m\geq 1$, we begin with $p_m=p_{m-1}-1$ and $q_m=q_{m-1}+1$, if $p_{m-1}>0$, and with $p_m=q_{m-1}+1$ and $q_m=0$, if $p_{m-1}=0$. If $\alpha_0(p_m) + \alpha_0(q_m)=1$, then $\alpha(m)=(\alpha_0(p_m),\alpha_0(q_m))$ ordering them decreasingly, if necessary. Otherwise, we decrease $p_m$ one unit and increase $q_m$ one unit and continue doing so until we get either two rational numbers whose sum is one or $p_m=0$. In the former case, if we achieved our goal after $k$ decreases, then we fix $p_m=p_{m-1}-k$ and $q_m=k$ and take $\alpha(m)= (\alpha_0(p_m),\alpha_0(q_m))$, ordered if necessary. In the latter case, we consider $p_m=p_{m-1}+1$ and $q_m=0$ and repeat the one-unit decrease of $p_m$ and one-unit increase of $q_m$ process until we find a pair of rationals whose sum is one. Once ordered, we take this pair as $\alpha(m)$ and we fix $p_m$ and $q_m$ accordingly.
Notice we can follow an analogous procedure to construct a finite map for any $n>2$.
%as we need to order all possible combinations of $0\leq x_1,..,x_n \leq m$ such that $\sum_{i=1}^n x_i=m$ and we can go through them in the same way we did for the $n=2$ case.
From now on, we consider an arbitrary $n \geq 2$. We now only need to show $\{\langle n,m \rangle| \alpha(n) \preceq \alpha(m)\}$ is recursively enumerable, that is, we need to construct some computable function $f: \mathbb{N} \to \{\langle n,m \rangle| \alpha(n) \preceq \alpha(m)\}$. Given $m \in \mathbb{N}$, we get $p.q \in \mathbb{N}$ such that $m=\langle p,q \rangle$ and calculate $\alpha(p),\alpha(q)$. If $s_k(\alpha(p)) \leq s_k(\alpha(q))$ $\text{for all } k\leq n$, then $f(m)=m$. Otherwise, $f(m)=0$, since $0=\langle0,0\rangle$ and $\alpha(0) \preceq \alpha(0)$.
%$f(0)=0$ as $0=\langle0,0\rangle$ and $s_k(\alpha(0))=s_k(\alpha(0))$ $\text{for all } k \leq n$.

\pagebreak

\subsubsection{Proof of Lemma \ref{example domains}}
\label{lemma 5 proof}

$(1)$ As we known from \cite{edalat1999domain}, $(\mathcal{I}, \sqsubseteq)$ is $\omega$-continuous. Take $Z \subseteq \mathcal{I}$ a Debreu dense subset of $ \mathcal{I}$. Take for each any $x \in \mathbb{R}$ some $y_x < x$ and notice $[y_x,x] \sqsubset [x,x]$. By Debreu separability, there exists some $z \in Z$ such that $[y_x,x] \sqsubseteq z \sqsubseteq [x,x]$. Thus, defining $z \coloneqq [z_1,z_2]$, we have $z_2=x$,  since, by definition, $x \leq z_2 \leq x$. If we fix for each $x$ such a $z$ and denote it by $z_x$, we have $x$ determines $z_x$ uniquely. Hence, the map $f:\mathbb{R} \to Z$, $x \mapsto z_x$ is injective and, by injectivity of $f$, $\mathfrak{c} \leq |Z|$. Thus, $Z$ has the cardinality of the continuum.

$(2),(3)$ By \cite[Lemma 5 (i) and (ii)]{hack2022representing}, we know both that $(\Lambda^n,\preceq_M)$ is order separable, thus Debreu upper separable, if $n=2$ and that any Debreu dense subset has the cardinality of the continuum if $n \geq 3$. To conclude, we show $(\Lambda^n,\preceq_M)$ is $\omega$-continuous $\text{for all } n \geq 2$. In particular, we show for each $x \in \Lambda^n$ there exists some $B_x \subseteq \mathbb{Q}^n \cap \Lambda^n$ such that $x=\sqcup B_x$ and $b\ll x$ $\text{for all } b \in B_x$ and obtain, as a result, $\mathbb{Q}^n \cap \Lambda^n$ is a countable basis. If $x=\perp$, then $B_\perp \coloneqq \{\perp\} \subseteq \mathbb{Q}^n \cap \Lambda^n$ does the job, since $\perp \in K(\Lambda^n)$. If $x \neq \perp$, then take
\begin{equation*}
    B_x \coloneqq \{q \in \mathbb{Q}^n \cap \Lambda^n| s_k(q) < s_k(x) \text{ }\text{for all } k<n \}.
\end{equation*}
%$B_x \coloneqq \{q \in \mathbb{Q}^n \cap \Lambda^n| s_k(q) < s_k(x) \text{ }\text{for all } k<n \}$.
Notice we have $q \ll x$ by \eqref{way-below majo}, thus, $q \preceq x$ $\text{for all } q \in B_x$. To finish, we need to show $\sqcup B_x=x$.
%As $q \ll x$ $\text{for all } q \in B_x$ we automatically have $q \preceq x$.
Assume there exits some $y \in \Lambda^n$ such that $q \preceq y$ $\text{for all } q \in B_x$. By Lemma \ref{arbitrary q}, that would mean for any $\varepsilon >0$ we have $s_k(x)-\varepsilon < s_k(y)$ $\text{for all } k<n$. Thus, $s_k(x) \leq s_k(y)$ $\text{for all } k\leq n$ and $x \preceq y$. As a result, $\sqcup B_x=x$. An alternative proof that $(\Lambda^n,\preceq_M)$ is $\omega$-continuous $\text{for all } n \geq 2$, where Lemma \ref{arbitrary q} is not used, can be found in Appendix \ref{alternative}.

\subsubsection{Second proof that majorization has a countable basis}
\label{alternative}

We prove here majorization is $\omega$-continuous $\text{for all } n \geq 2$ without using Lemma \ref{arbitrary q}.
By \cite[Theorem 1.3]{martin2006entropy}, we know $(\Lambda^n,\preceq_M)$ is a continuous dcpo $\text{for all } n\geq 2$. We will show $B \coloneqq \mathbb{Q}^n\cap \Lambda^n$ is a countable basis. In order to do so, it is sufficient to show, $\text{for all } x,y \in \Lambda^n$ such that $x \ll y$, there exists some $b \in B$ such that $x \ll b \ll y$ \cite[Proposition 2.4]{mao2018measurement}. By \eqref{way-below majo}, if $x \ll y$, then
%can be divided into two cases:
either $x=\perp$ or $\sum_{i=1}^k x_i < \sum_{i=1}^k y_i$ $\text{for all } k<n$. Since $\perp \in B$ and $\perp \ll \perp$, we can take $b=\perp$ in the first case.
%is already covered.
Assume the second case holds. Note $0<x_i<1$ $\text{for all } i \leq n$, since the opposite contradicts the fact $\sum_{i=1}^k x_i < \sum_{i=1}^k y_i$ $\text{for all } k<n$. Take $(\varepsilon_i)_{i=1}^{k-1}$, where
\begin{comment}
  \begin{equation*}
    \begin{cases}
    0<\varepsilon_i < \text{min} \{x_n,\text{ } y_i-x_i\} & \text{if} \text{ } i=1,\\
    0<\varepsilon_i < \text{min} \{ x_n - \sum_{j=1}^{i-1} \varepsilon_j,\text{ } y_i-x_i\} & \text{if} \text{ } 1<i< n.
    \end{cases}
\end{equation*}
\end{comment}
  \begin{equation*}
    \begin{cases}
    0<\varepsilon_i < \text{min} \Big\{y_i-x_i,\text{ } s_{i+1}(y)-s_{i+1}(x) %\sum^{i+1}_{j=1}(y_j-x_j)
    \Big\} & \text{if} \text{ } i=1,\\
    0<\varepsilon_i < \text{min} \Big\{ s_i(y)-s_i(x) - s_{i-1}(\varepsilon)
    %\sum^{i}_{j=1}(y_j-x_j) -
    %-\sum_{j=1}^{i-1} \varepsilon_j
    ,\\
    s_{i+1}(y)
    -s_{i+1}(x)-s_{i-1}(\varepsilon)\Big\}
    %\sum^{i+1}_{j=1}(y_j-x_j) - 
    %-\sum_{j=1}^{i-1} \varepsilon_j\}
    & \text{if} \text{ } 1<i \leq n-2 \\
    0<\varepsilon_i <
    s_i(y)-s_i(x)-s_{i-1}(\varepsilon)
    %\sum^{i}_{j=1}(y_j-x_j) - \sum_{j=1}^{i-1} \varepsilon_j 
    & \text{if} \text{ } i= n-1
    \end{cases}
\end{equation*}
\begin{comment}
\[
\left\{
\begin{aligned}
0<\varepsilon_i < \text{min} &\Big\{y_i-x_i,\text{ } s_{i+1}(y)-s_{i+1}(x) %\sum^{i+1}_{j=1}(y_j-x_j)
    \Big\} &\text{if} \text{ } i=1,\\
    0<\varepsilon_i < \text{min} &\Big\{ s_i(y)-s_i(x) - s_{i-1}(\varepsilon)
    %\sum^{i}_{j=1}(y_j-x_j) -
    %-\sum_{j=1}^{i-1} \varepsilon_j
    , &\\
    &s_{i+1}(y)
    -s_{i+1}(x)-s_{i-1}(\varepsilon)\Big\}
    %\sum^{i+1}_{j=1}(y_j-x_j) - 
    %-\sum_{j=1}^{i-1} \varepsilon_j\}
    &\text{if} \text{ } 1<i \leq n-2 \\
    0<\varepsilon_i <
    s_i(y)&-s_i(x)-s_{i-1}(\varepsilon)
    %\sum^{i}_{j=1}(y_j-x_j) - \sum_{j=1}^{i-1} \varepsilon_j 
    &\text{if} \text{ } i= n-1
\end{aligned}
\right.
\]
\end{comment}
Note the first upper bound in the definition of
%inside \emph{min} is the actual property of
$\varepsilon_i$ is the property we will use here $\text{for all } i\leq n-2$,
%in \eqref{way-be ok},
%while $\text{for all } i \leq n-2$
while the second upper bound is included to make sure $\varepsilon_{i+1}$ is well-defined. Take $\text{for all } i<n$ some $b_i \in \big(x_i,x_i + \varepsilon_i\big) \cap \mathbb{Q}$ such that $b_i \geq b_{i+1}$ $\text{for all } i<n-1$ and define $b:=\big(b_1,..,b_n\big)$, where $b_n:=1- \sum_{j=1}^{n-1} b_j$, implying $b \in \mathbb{Q}^n$. Note $b_n= 1 - \sum_{j=1}^{n-1} b_j < 1- \sum_{j=1}^{n-1} x_j = x_n \leq x_{n-1} < b_{n-1}$, which implies $b \in \Lambda^n$. By definition of $(\varepsilon_i)_{i=1}^{k-1}$, we have 
\begin{equation*}
\label{way-be ok}
    \sum_{j=1}^k x_j < \sum_{j=1}^k b_j < \sum_{j=1}^k (x_j + \varepsilon_j) < \sum_{j=1}^k x_j +\sum_{j=1}^k (y_j -x_j)= \sum_{j=1}^k y_j
\end{equation*}
$\text{for all } k < n$. Thus, $x \ll b \ll y$ and we have finished.

\subsubsection{Proof of Lemma \ref{arbitrary q}}
\label{proof lemma majo}

Consider $x \in \Lambda^n\setminus \{\perp\}$. Assume first $x=(\frac{1}{m},..,\frac{1}{m},0,.,0)$ for some $m<n$. Fix w.l.o.g. $0<\varepsilon <\frac{1}{m}$, consider some $\varepsilon' \in \mathbb{Q}$ such that \begin{equation*}
    0<\varepsilon'< \text{min} \Big\{\frac{\varepsilon}{m},\frac{\varepsilon}{n}\Big(n-m\Big)\Big\}
\end{equation*} 
and define $\beta = \frac{m}{n-m} \varepsilon'$. We define then the components of $q$, $q_i=\frac{1}{m}-\varepsilon'$ for $1 \leq i \leq m$ and $q_i=\beta$ for $m<i \leq n$. To assure $q \in \Lambda^n$ we need to show $s_n(q)=1$ and $q_i \leq q_{i-1}$ for $2 \leq i< n$. Notice $s_n(q)=s_n(x)-m\varepsilon' +(n-m) \beta =1$ by definition of $\beta$ while for the second part it suffices to show $\frac{1}{m} -\varepsilon' > \beta$, which holds as we have
\begin{equation*}
    \frac{1}{m} -\varepsilon' > \beta \iff \frac{1}{m} -\varepsilon' > \frac{m}{n-m} \varepsilon' \iff \frac{1}{m} > \frac{n}{n-m} \varepsilon'
\end{equation*}
and, by definition of $\varepsilon'$,
\begin{equation*}
    \varepsilon' < \varepsilon\frac{n-m}{n} < \frac{n-m}{n m}.
\end{equation*}
We show now $s_k(q)>s_k(x)-\varepsilon$ $\text{for all } k<n$ holds. If $i\leq m$ we have $s_i(q)=s_i(x)-i\varepsilon'>s_i(x)-\varepsilon$, since $\varepsilon'<\frac{\varepsilon}{m}$ holds by definition. If $m<i<n$, then we get $s_i(q)=1-m\varepsilon'+(i-m)\beta > 1-\varepsilon$, as the following holds 
\begin{equation*}
     \varepsilon'< \frac{\varepsilon}{m} \implies \varepsilon' < \frac{\varepsilon (n-m)}{m(n-i)} \iff m\varepsilon'+(m-i)\beta <\varepsilon,
\end{equation*}
where the first inequality is true by definition of $\varepsilon'$ and the first implication holds since $m<i<n$. Note $s_k(q)=\frac{1}{m}-k \varepsilon' < s_k(x)$ for $1 \leq k\leq m$ and $s_k(q)=1 - (n-k) \beta < 1=s_k(x)$ for $m < k < n$.

We assume now $x \neq (\frac{1}{m},..,\frac{1}{m},0,.,0)$ and define $k= \text{min} \big\{i < n|x_j=x_i \text{ or } x_j=0 \text{ }\text{for all } j\geq i\big\}$, if it exists, and, otherwise, $k=n$. We also define $h=\text{min} \{k\leq i\leq n| x_{i+1}=0\}$, if it exists, and, otherwise, $h=n$. Lastly, consider $\alpha=h-(k-1)$. We fix now $\varepsilon>0$, assume $h=n$ and consider some  $\varepsilon'$ such that
\begin{equation*}
    0<\varepsilon'< \text{min} \Big\{\frac{\varepsilon}{k},\Big(1+\frac{k-1}{\alpha}\Big)^{-1} \Big(x_{k-1}-x_k\Big)\Big\}.
\end{equation*} 
Notice $\varepsilon'$ is well defined as we have $x_{k-1}>x_k$. We take now $q_i \in (x_i-\varepsilon',x_i)\cap \mathbb{Q}$ $\text{for all } i<k$ such that $q_i \leq q_{i-1}$ for $2 \leq i<k$ and $q_i = \tau$ for $k \leq i \leq h$ where $\tau=\frac{1}{\alpha}(1-\sum_{i=1}^{k-1} q_i)$. Notice as $h=n$ we have $s_n(q)=1$ and we also have $\tau<q_{k-1}$ as
\begin{equation*}
\begin{split}
    \tau &= \frac{1}{\alpha}\Big(1-\sum_{i=1}^{k-1} q_i\Big) < \frac{1}{\alpha}\Big(1-\sum_{i=1}^{k-1} x_i +(k-1) \varepsilon'\Big) = x_k +\frac{k-1}{\alpha} \varepsilon' \\ 
    &\stackrel{(i)}{<} x_{k-1}-\varepsilon'<q_{k-1}
    \end{split}
\end{equation*}
where we have used the definition of $\varepsilon'$ in $(i)$. We show now we have $s_i(q)>s_i(x)-\varepsilon$ $\text{for all } i<h$ which is sufficient as we are assuming $h=n$. If $i\leq k-1$ we have
\begin{equation*}
    s_i(q)>s_i(x)-i \varepsilon' >s_i(x)-\frac{i}{k} \varepsilon > s_i(x)- \varepsilon.
\end{equation*}
while if $k\leq i < h$ we have
\begin{equation*}
    s_i(q)= \sum_{j=1}^{k-1} q_j + \sum_{j=k}^i q_j \stackrel{(i)}{>} s_k(x)- \varepsilon + \sum_{j=k}^i x_j = s_i(x)- \varepsilon.
\end{equation*}
where in $(i)$ we have used the fact $q_i>x_i$ $\text{for all } i \geq k$ as we have
\begin{equation*}
    x_i=x_k=\frac{1}{\alpha}\Big(1-s_{k-1}(x)\Big)<\frac{1}{\alpha}\Big(1-s_{k-1}(q)\Big)=\tau=q_i.
\end{equation*}
Thus, if $h=n$ we have finished. Note we can see $s_i(q) < s_i(x)$ for $1 \leq i < n$ similarly to the previous case.

Assume now $h<n$. Notice defining some $q' $ where $q'_i=q_i$ if $i\leq h$ and $q'_i=0$ if $h<i \leq n$ does not work, since we would have $s_h(q)=s_h(x)=1$. Thus, we need $q'_i >0$ if $h<i\leq n$. Consider some $\beta \in \mathbb{Q}$ such that  
\begin{equation*}
0<\beta < \text{min} \Big\{\frac{\varepsilon}{n-h}, \Big(1+\frac{n-h}{h-(k-1)}\Big)^{-1} q_k,\frac{1}{n-h}\Big(s_k(q)-\Big(s_k(x)-\varepsilon\Big)\Big)\Big\},
\end{equation*}
where $q$ is defined as in the case $h=n$, and define $\beta'=\frac{n-h}{h-(k-1)} \beta$. We consider now $q'$ where $q'_i=q_i$ if $i < k$, $q'_i=q_i-\beta'$ if $k\leq i \leq h$ and $q'_i=\beta$ if $h<i$. Notice we have $q'_i \leq q'_{i+1}$ for $2 \leq i <n$, since we have it already from $q$ in case $i \leq h$, and for $h<i$ it suffices to show $q_k'>\beta$, which is true as
\begin{equation*}
    q'_k=q_k-\beta'=q_k-\frac{n-h}{h-(k-1)} \beta > \beta,
\end{equation*}
where the inequality holds by definition of $\beta$. Notice, also, we have $s_n(q')=s_h(q) - \beta'(h-(k-1))+\beta (n-h) =s_h(q)=1$. To conclude, we only need to show $s_i(q')>s_i(x)-\varepsilon$ $\text{for all } i<n$. If $i<k$, then we already have it as $s_i(q')=s_i(q)$ and we know it holds for $q$. If $k\leq i \leq h$, then it also holds as 
\begin{equation*}
    s_i(q')=s_i(q)-(i-(k-1))\beta' \geq s_i(q)-(h-(k-1))\beta' \stackrel{(i)}{>} s_i(x)-\varepsilon,
\end{equation*}
where, in $(i)$, we apply the fact that for $k \leq i \leq h$ we have
\begin{equation*}
   \beta'<\frac{1}{h-(k-1)} \Big(s_k(q)-\Big(s_k(x)-\varepsilon\Big)\Big) \stackrel{(ii)}{\leq} \frac{1}{h-(k-1)} \Big(s_i(q)-\Big(s_i(x)-\varepsilon\Big)\Big),
\end{equation*}
where the first inequality holds by definition of $\beta$ and $\beta'$
%and definition of $q$ and where 
and $(ii)$ also does, as we have for $k\leq i \leq h$
\begin{equation*}
\begin{split}
  s_i(q)-(s_i(x)-\varepsilon) &= s_k(q) +(i-k)q_k -(s_k(x)-\varepsilon)-(i-k)x_k \\
  &=s_k(q) -(s_k(x)-\varepsilon) + (i-k)(q_k-x_k) \\
  &\leq s_k(q) -(s_k(x)-\varepsilon),
  \end{split}
\end{equation*}
where we used the fact $q_k>x_k$ in the inequality.
If $h<i<n$, then
\begin{equation*}
s_i(q')=1-(n-i)\beta >1-\varepsilon = s_i(x)-\varepsilon,
\end{equation*}
where the inequality holds since we have $\beta < \frac{\varepsilon}{n-h} < \frac{\varepsilon}{n-i}$. Note we can see $s_i(q) < s_i(x)$ for $1 \leq i < n$ similarly to both previous cases.

\end{appendix}

\newpage
\bibliographystyle{plain}
\bibliography{main}

\begin{thebibliography}{10}

\bibitem{abramsky1994domain}
Samson Abramsky and Achim Jung.
\newblock Domain theory. {H}andbook of logic in computer science.
\newblock {\em Claendon Press, Oxford}, 6(4):1--168, 1994.

\bibitem{alcantud2016richter}
Jos{\'e} Carlos~R Alcantud, Gianni Bosi, and Magal{\`\i} Zuanon.
\newblock Richter--{P}eleg multi-utility representations of preorders.
\newblock {\em Theory and Decision}, 80(3):443--450, 2016.

\bibitem{badaev2000theory}
Serikzhan Badaev and Sergey Goncharov.
\newblock The theory of numberings: open problems.
\newblock {\em Contemporary Mathematics}, 257:23--38, 2000.

\bibitem{badaev2008computability}
Serikzhan Badaev and Sergey Goncharov.
\newblock Computability and numberings.
\newblock In {\em New Computational Paradigms}, pages 19--34. Springer, 2008.

\bibitem{blanck2008reducibility}
Jens Blanck.
\newblock Reducibility of domain representations and {C}antor-{W}eihrauch domain representations.
\newblock {\em Math. Struct. Comput. Sci.}, 18(6):1031--1056, 2008.

\bibitem{braverman2005complexity}
Mark Braverman.
\newblock On the complexity of real functions.
\newblock In {\em 46th Annual IEEE Symposium on Foundations of Computer Science (FOCS'05)}, pages 155--164. IEEE, 2005.

\bibitem{bridges2013representations}
Douglas~S Bridges and Ghanshyam~B Mehta.
\newblock {\em Representations of preferences orderings}, volume 422.
\newblock Springer Science \& Business Media, 2013.

\bibitem{cartwright2016domain}
Robert Cartwright, Rebecca Parsons, and Moez AbdelGawad.
\newblock Domain theory: An introduction.
\newblock {\em arXiv preprint arXiv:1605.05858}, 2016.

\bibitem{cutland1980computability}
Nigel Cutland.
\newblock {\em Computability: {A}n introduction to recursive function theory}.
\newblock Cambridge university press, 1980.

\bibitem{debreu1954representation}
Gerard Debreu.
\newblock Representation of a preference ordering by a numerical function.
\newblock {\em Decision processes}, 3:159--165, 1954.

\bibitem{di1996real}
Pietro Di~Gianantonio.
\newblock Real number computability and domain theory.
\newblock {\em Information and Computation}, 127(1):11--25, 1996.

\bibitem{edalat1999domain}
Abbas Edalat and Philipp S{\"u}nderhauf.
\newblock A domain-theoretic approach to computability on the real line.
\newblock {\em Theoretical Computer Science}, 210(1):73--98, 1999.

\bibitem{ershov1972computable}
Yu~L Ershov.
\newblock Computable functionals of finite types.
\newblock {\em Algebra and Logic}, 11(4):203--242, 1972.

\bibitem{ershov1999theory}
Yuri~Leonidovich Ershov.
\newblock Theory of numberings.
\newblock {\em Handbook of computability theory}, 140:473--506, 1999.

\bibitem{evren2011multi}
{\"O}zg{\"u}r Evren and Efe~A Ok.
\newblock On the multi-utility representation of preference relations.
\newblock {\em Journal of Mathematical Economics}, 47(4-5):554--563, 2011.

\bibitem{gierz2012compendium}
Gerhard Gierz, Karl~Heinrich Hofmann, Klaus Keimel, Jimmie~D Lawson, Michael Mislove, and Dana~S Scott.
\newblock {\em A compendium of continuous lattices}.
\newblock Springer Science \& Business Media, 2012.

\bibitem{giles2016mathematical}
Robin Giles.
\newblock {\em Mathematical foundations of thermodynamics: International series of monographs on pure and applied mathematics}, volume~53.
\newblock Elsevier, 2016.

\bibitem{gottwald2019bounded}
Sebastian Gottwald and Daniel~A Braun.
\newblock Bounded rational decision-making from elementary computations that reduce uncertainty.
\newblock {\em Entropy}, 21(4):375, 2019.

\bibitem{hack2022computation}
Pedro Hack, Daniel~A Braun, and Sebastian Gottwald.
\newblock Computation as uncertainty reduction: a simplified order-theoretic framework.
\newblock {\em arXiv preprint arXiv:2206.13885}, 2022.

\bibitem{hack2022representing}
Pedro Hack, Daniel~A Braun, and Sebastian Gottwald.
\newblock Representing preorders with injective monotones.
\newblock {\em Theory and Decision}, 93(4):663--690, 2022.

\bibitem{hack2022classification}
Pedro Hack, Daniel~A Braun, and Sebastian Gottwald.
\newblock The classification of preordered spaces in terms of monotones: complexity and optimization.
\newblock {\em Theory and Decision}, 94(4):693--720, 2023.

\bibitem{keimel2017domain}
Klaus Keimel.
\newblock Domain theory its ramifications and interactions.
\newblock {\em Electronic Notes in Theoretical Computer Science}, 333:3--16, 2017.

\bibitem{kelley2017general}
John~L Kelley.
\newblock {\em General topology}.
\newblock Courier Dover Publications, 2017.

\bibitem{ko2012complexity}
Ker Ko.
\newblock {\em Complexity theory of real functions}.
\newblock Springer Science \& Business Media, 2012.

\bibitem{kreitz1985theory}
Christoph Kreitz and Klaus Weihrauch.
\newblock Theory of representations.
\newblock {\em Theoretical computer science}, 38:35--53, 1985.

\bibitem{landsberg1970main}
PT~Landsberg.
\newblock Main ideas in the axiomatics of thermodynamics.
\newblock {\em Pure and Applied Chemistry}, 22(3-4):215--228, 1970.

\bibitem{lawson1998computation}
Jimmie Lawson.
\newblock Computation on metric spaces via domain theory.
\newblock {\em Topology and its Applications}, 85(1-3):247--263, 1998.

\bibitem{lieb1999physics}
Elliott~H Lieb and Jakob Yngvason.
\newblock The physics and mathematics of the second law of thermodynamics.
\newblock {\em Physics Reports}, 310(1):1--96, 1999.

\bibitem{mao2018measurement}
Xuxin Mao and Luoshan Xu.
\newblock The measurement topology and the density topology of posets.
\newblock {\em Journal of Mathematical Research with Applications}, 38(4):341--350, 2018.

\bibitem{marshall1979inequalities}
Albert~W Marshall, Ingram Olkin, and Barry~C Arnold.
\newblock {\em Inequalities: theory of majorization and its applications}, volume 143.
\newblock Springer, 1979.

\bibitem{martin2006entropy}
Keye Martin.
\newblock Entropy as a fixed point.
\newblock {\em Theoretical Computer Science}, 350(2-3):292--324, 2006.

\bibitem{martin2000foundation}
Keye Martin and Michael Mislove.
\newblock A foundation for computation.
\newblock {\em Tulane University, New Orleans, LA}, 2000.

\bibitem{martin2008technique}
Keye Martin and Prakash Panangaden.
\newblock A technique for verifying measurements.
\newblock {\em Electronic Notes in Theoretical Computer Science}, 218:261--273, 2008.

\bibitem{mehta1986existence}
Ghanshyam~B Mehta.
\newblock Existence of an order-preserving function on normally preordered spaces.
\newblock {\em Bulletin of the Australian Mathematical Society}, 34(1):141--147, 1986.

\bibitem{mislove1998topology}
Michael~W Mislove.
\newblock Topology, domain theory and theoretical computer science.
\newblock {\em Topology and its Applications}, 89(1-2):3--59, 1998.

\bibitem{munkres2019topology}
James~R Munkres.
\newblock {\em Topology}.
\newblock Pearson Education, 2019.

\bibitem{ok2002utility}
Efe~A Ok et~al.
\newblock Utility representation of an incomplete preference relation.
\newblock {\em Journal of Economic Theory}, 104(2):429--449, 2002.

\bibitem{roberts1968axiomatic}
Fred~S Roberts and R~Duncan Luce.
\newblock Axiomatic thermodynamics and extensive measurement.
\newblock {\em Synthese}, pages 311--326, 1968.

\bibitem{rogers1987theory}
Hartley Rogers~Jr.
\newblock {\em Theory of recursive functions and effective computability}.
\newblock MIT press, 1987.

\bibitem{scott1970outline}
Dana Scott.
\newblock {\em Outline of a mathematical theory of computation}.
\newblock Oxford University Computing Laboratory, Programming Research Group Oxford, 1970.

\bibitem{scott1972continuous}
Dana Scott.
\newblock Continuous lattices.
\newblock In {\em Toposes, algebraic geometry and logic}, pages 97--136. Springer, 1972.

\bibitem{scott1982lectures}
Dana Scott.
\newblock Lectures on a mathematical theory of computation.
\newblock In {\em Theoretical Foundations of Programming Methodology}, pages 145--292. Springer, 1982.

\bibitem{smyth1983power}
Michael~B Smyth.
\newblock Power domains and predicate transformers: A topological view.
\newblock In {\em International Colloquium on Automata, Languages, and Programming}, pages 662--675. Springer, 1983.

\bibitem{stoltenberg2001notes}
Viggo Stoltenberg-Hansen.
\newblock Notes on domain theory.
\newblock {\em Lecture notes from the summer school on Proof and System Reliability in Marktoberdorf}, 2001.

\bibitem{stoltenberg1994mathematical}
Viggo Stoltenberg-Hansen, Tom Lindstrom, Ingrid Lindstr{\"o}m, and ER~Griffor.
\newblock {\em Mathematical theory of domains}, volume~22.
\newblock Cambridge University Press, 1994.

\bibitem{stoltenberg2008computability}
Viggo Stoltenberg-Hansen and John~V Tucker.
\newblock Computability on topological spaces via domain representations.
\newblock In {\em New Computational Paradigms}, pages 153--194. Springer, 2008.

\bibitem{turing1937computable}
Alan~Mathison Turing.
\newblock On computable numbers, with an application to the {E}ntscheidungsproblem.
\newblock {\em Proceedings of the London mathematical society}, 2(1):230--265, 1937.

\bibitem{turing1938computable}
Alan~Mathison Turing.
\newblock On computable numbers, with an application to the {E}ntscheidungsproblem. a correction.
\newblock {\em Proceedings of the London Mathematical Society}, 2(1):544--546, 1938.

\bibitem{weihrauch2012computability}
Klaus Weihrauch.
\newblock {\em Computability}, volume~9.
\newblock Springer Science \& Business Media, 2012.

\bibitem{weihrauch2012computable}
Klaus Weihrauch.
\newblock {\em Computable analysis: an introduction}.
\newblock Springer Science \& Business Media, 2012.

\end{thebibliography}
\end{document}